\documentclass[10pt]{article}

\usepackage{url}
\usepackage{mathtools}
\usepackage{amssymb}
\usepackage{amsthm}
\usepackage{empheq}
\usepackage{latexsym}
\usepackage{enumitem}
\usepackage{eurosym}
\usepackage{dsfont}
\usepackage{appendix}
\usepackage{color} 
\usepackage[unicode]{hyperref}
\usepackage{frcursive}
\usepackage[utf8]{inputenc}
\usepackage[T1]{fontenc}
\usepackage{geometry}
\usepackage{multirow}
\usepackage{todonotes}
\usepackage{lmodern}
\usepackage{anyfontsize}
\usepackage{stmaryrd}
\usepackage{natbib}
\usepackage{cleveref}

\setcitestyle{numbers,open={[},close={]}}

\definecolor{red}{rgb}{0.7,0.15,0.15}
\definecolor{green}{rgb}{0,0.5,0}
\definecolor{blue}{rgb}{0,0,0.7}
\hypersetup{colorlinks, linkcolor={red},citecolor={green}, urlcolor={blue}}
			
\makeatletter \@addtoreset{equation}{section}

\newtheorem{theorem}{Theorem}[section]
\newtheorem{assumption}[theorem]{Assumption}
\newtheorem{corollary}[theorem]{Corollary}

\newtheorem{lemma}[theorem]{Lemma}
\newtheorem{proposition}[theorem]{Proposition}

\newtheorem{definition}[theorem]{Definition}
\newtheorem{remark}[theorem]{Remark}

\DeclareUnicodeCharacter{014D}{\=o}
\def \E{\mathbb{E}}
\def \F{\mathbb{F}}
\def \G{\mathbb{G}}
\def \H{\mathbb{H}}

\def \L{\mathbb{L}}
\def \M{\mathbb{M}}
\def \N{\mathbb{N}}

\def \P{\mathbb{P}}
\def \Q{\mathbb{Q}}
\def \R{\mathbb{R}}
\def \S{\mathbb{S}}

\def \X{\mathbb{X}}

\def \Z{\mathbb{Z}}

\def \Pr{\mathrm{P}}

\def\Ac{{\cal A}}
\def\Bc{{\cal B}}
\def\Cc{{\cal C}}
\def\Dc{{\cal D}}

\def\Fc{{\cal F}}
\def\Gc{{\cal G}}
\def\Hc{{\cal H}}

\def\Kc{{\cal K}}
\def\Lc{{\cal L}}

\def\Pc{{\cal P}}

\def\Rc{{\cal R}}
\def\Sc{{\cal S}}

\def\Uc{{\cal U}}
\def\Vc{{\cal V}}
\def\Wc{{\cal W}}

\def\Zc{{\cal Z}}

\def\Fb{{\bar F}}
\def\Gb{{\overline \G}}
\def\Gcb{\overline \Gc}

\def\x{\times}

\def\Om{\Omega}

\def\om{\omega}

\def\Omb{\overline{\Om}}

\def\Fb{\overline{\F}}
\def\Gcb{\overline{\Gc}}
\def\Fcb{\overline{\Fc}}

\def\Xt{\widetilde{X}}

\def\etab{\overline{\eta}}
\def\0{\mathbf{0}}

\def \xb{\mathbf{x}}

\def \mub{\overline{\mu}}

\def \zetab{\overline{\zeta}}

\def \muh{\widehat{\mu}}

\def \mut{\widetilde{\mu}}
\def \Ar{\mathrm{A}}

\def \alphat{\widetilde{\alpha}}

\def \Qr{\mathrm{Q}}

\def \Prt{\widetilde{\Pr}}

\def \qr{{\rm q}}
\def \qrt{\widetilde{\rm q}}
\def \pirt{\widetilde{\rm \pi}}

\def \St{\widetilde{S}}

\def \Zt{\widetilde{Z}}

\def \betat{\widetilde{\beta}}

\def \alepht{\widetilde{\aleph}}

\def \Lambdat{\widetilde{\Lambda}}
\def \Rt{\widetilde{R}}

\def \Xh{\widehat{X}}

\def\normeL2#1{\left\|{#1}\right\|_{L^2}}

\def\Pcb{\overline \Pc}

\def \ib {\boldsymbol{i}}

\def \Lim{\displaystyle\lim}

\def \Xbb{\mathbf{X}}

\def \Ybb{\mathbf{Y}}

\def \alphab {\boldsymbol{\alpha}}

\def \Xbb{\mathbf{X}}

\def \Ybb{\mathbf{Y}}

\def\Er{{\rm E}}
\def\Gr{{\rm G}}
\def\Ir{{\rm I}}
\def\Jr{{\rm J}}
\def\Rr{{\rm R}}
\def\Kr{{\rm K}}

\def\Cf{\mathfrak{C}}

\def\Zf{\mathfrak{Z}}

\def\br{{\rm b}}

\def\rr{{\rm r}}

\def \Qr{\mathrm{Q}}

 \title{Stackelberg Mean Field Games: convergence and existence results to the problem of Principal with multiple Agents in competition\footnote{The author thanks Dylan {\sc Possama\"{i}}  for his helpful suggestions and interesting discussions.}}

\author{
    Mao Fabrice {\sc Djete}\thanks{Ecole Polytechnique Paris, Centre de Math\'ematiques Appliqu\'ees,
mao-fabrice.djete@polytechnique.edu. This work benefits from the financial support of the Chairs {\it Financial Risk} and {\it Finance and Sustainable Development}. }
    }
             \date{\today}

\begin{document}

\maketitle
 
\begin{abstract}
    In a situation of moral hazard, this paper investigates the problem of Principal with $n$ Agents when the number of Agents $n$ goes to infinity. There is competition between the Agents expressed by the fact that they optimize their utility functions through a Nash equilibrium criterion. Each Agent is offered by the Principal a contract which is divided into a Markovian part involving the state/production of the Agent and a non--Markovian part involving the states/productions of all the other Agents. The Agents are in interactions. These interactions are characterized by common noise, the empirical distribution of states/productions and controls, and the contract which is not assumed to be a map of the empirical distribution. By the help of the mean field games theory, we are able to formulate an appropriate $limit$ problem involving a Principal with a $representative$ Agent. We start by solving the problem of both the Principal and the $representative$ Agent in this $limit$ problem. Then, when $n$ goes to infinity, we show that the problem of Principal with $n$ Agents converges to the $limit$ problem of Principal with a $representative$ Agent. That is, first, any convergent sequence of (approximate) best responses from the $n$ Agents associated to an (approximate) $n$--dependent optimal contract from the Principal in the problem of Principal with $n$ Agents converges to the best response of the $representative$ Agent in the $limit$ problem when the Principal offers the limit contract of the (approximate) $n$--dependent optimal contract. The limit contract is also optimal for the Principal's problem in the $limit$ problem. Second, given an optimal contract for the Principal and the associated best response from the $representative$ Agent, we are able to construct the (approximate) contract and the corresponding (approximate) best responses for the problem of Principal with $n$ Agents. A notable result is that, despite allowing a general type of contracts, it is approximately optimal for the Principal to offer contracts to the $n$ Agents that are maps of the empirical distribution of states/productions and controls of the Agents.

\end{abstract}



\section{Introduction}

By analyzing the relationship between a Principal and an Agent, the Principal--Agent theory turns out to play an important role in economics, management, game theory, $\cdots$ This theory focuses on situations where one party, the Principal (employer, manager, or owner), delegates decision--making authority or tasks to another party, the Agent (employee, worker, or manager), who acts on the Principal's behalf. The Principal--Agent relationship is prevalent in various settings, such as corporate governance, employment contracts, government administration, and more. By addressing the challenges that arise from the delegation of tasks and decision--making, this theory helps to design better contracts, governance structures, and incentive systems, leading to more effective and successful outcomes in various real--world settings.

\medskip
The primary focus of the Principal--Agent theory is to design incentive mechanisms and contracts that align the interests of the Principal and the Agent, ensuring that the Agent acts in the best interest of the Principal. The central issue is how to design a contract that motivates the Agent to take actions that maximize the Principal's welfare, even when the Agent's own objectives may not align perfectly with those of the Principal. There is therefore an asymmetry of information which is  known in the literature as $``$moral hazard$"$, as the Agent's actions are not fully observable by the Principal. In the situation described, from a game theory perspective, this desire to find the appropriate contracts corresponding to the interests of the Principal and the Agent leads to the search for a Stackelberg equilibrium. Indeed, in simple terms, given a contract, the Principal will compute the potential best response of the Agent. Then, with this best response from the Agent, the Principal will try to maximize his utility and find the resulting optimal contract.

\medskip
The problem of Principal--Agent has been extensively studied in the field of economics, and there is a significant body of literature that addresses its mathematical treatment. \citeauthor*{LaffontMartimort2002}\cite{LaffontMartimort2002}, \citeauthor*{SchmitzReview}\cite{SchmitzReview}, \citeauthor*{laffont1993theory}\cite{laffont1993theory}, \citeauthor*{salanie1997economics}\cite{salanie1997economics} are some examples of classical references that provide some general overview on the topic while trying to solve some questions raised by the subject. The mathematical resolution has been treated in only static or discrete--time for a long time before the seminal paper of \citeauthor*{Milgrom1987} \cite{Milgrom1987} who provided a way to deal with this problem in continuous time. \citeauthor*{schattler1993first} \cite{schattler1993first,schattler1997optimal}, \citeauthor*{sung1995linearity} \cite{sung1995linearity}, \citeauthor*{muller1998first}\cite{muller1998first,muller2000asymptotic} follow the basis of \cite{Milgrom1987} and extend their work in various situations. By using techniques borrowed from the stochastic control theory (SCT), the notable work of \citeauthor*{sannikov2008continuous} \cite{sannikov2008continuous,sannikov2012contracts} brought some tools in the approach of the resolution of this problem. This use of SCT inspired \citeauthor*{cvitanic2015dynamic} \cite{cvitanic2015dynamic,cvitanic2014moral} to provide a general framework for solving the Princial--Agent problem. Their approach provides a generic way to deal with this problem that can be summarize as follows: given a contract, a resolution of Backward Stochastic Differential Equation (BSDE) helps solving the problem of the Agent. Then, the problem of the Principal is tackled by solving a classical stochastic control problem with two states variables: the output controlled by the Agent, and his continuation utility.

\medskip
An important case from the point of view of applications is the consideration of several Agents in the presence of one Principal. It is sometimes referred to in the literature as the $``$multi--Agent Principal--Agent problem$"$. One Key challenge in the multi--Agent Principal--Agent problem is that the Principal needs to account for the interactions and dependencies between different Agents' actions. The decisions of one Agent may affect the incentives or outcomes for other Agents, leading to potential coordination problems. Under various approaches/techniques, the multi--Agent Principal--Agent problem has been treated by several papers/books such as \citeauthor*{holmstrom1982moral} \cite{holmstrom1982moral}, \citeauthor*{Mookherjee1984} \cite{Mookherjee1984}, \citeauthor*{green1983comparison} \cite{green1983comparison}, \citeauthor*{demski1984optimal} \cite{demski1984optimal}, \citeauthor*{SungOptimal2008} \cite{SungOptimal2008}, \citeauthor*{possamai2019contracting} \cite{possamai2019contracting}, $\cdots$ The Agents are generally assumed to be in competition, which is represented by the use of Nash equilibrium criteria for the maximization of their utility functions.
From the perspective of Stochastic Control Theory with the presence of competition between Agents, which is the approach we would like to emphasize here, it is worth emphasizing the paper of \cite{possamai2019contracting} (see also \citeauthor*{espinosa2015optimal} \cite{espinosa2015optimal}) who uses the approach of \cite{cvitanic2015dynamic} in the situation of mutlti--Agent. The resolution in terms of BSDE--(classical)Stochastic control Problem is still valid. However, among other things, the main difficulty is the handling of multi--dimensional BSDE. Contrary to the one dimensional setting, quadratic multi--dimensional BSDEs are known to suffer some ill--posedness see for instance \citeauthor*{Frei2011AFM} \cite{Frei2011AFM}. \citeauthor*{possamai2019contracting} \cite{possamai2019contracting} deal with this problem by imposing the wellposedness as a requirement for the admissibility of the contracts offered by the Principal. Since the optimal contracts they find at the end turn out to verify the admissibility they imposed, this restriction does not seem to be a huge limitation.

\medskip
In many situations, the multi--Agent model is considered with a very large number of Agents (government policies, electricity management, crowd control, $\cdots$). It is well--known that as the number of Agents increases, the complexity of the problem also rises significantly especially in the presence of competitiveness among the Agents. The usual approach to treat this situation in the literature is to represent this large number of Agents through an appropriate unique entity. It is sometimes referred to as the problem of Principal--Agent with a continuum of Agents.
In the past recent years, when it comes down to deal with a large number/continuum of Agents in competition (Nash equilibria), it is becoming  quite natural and appropriate to use Mean Field Games (MFG). Mean Field Games were introduced in the seminal works of \citeauthor*{lasry2006jeux} \cite{lasry2007mean} and \citeauthor*{huang2003individual} \cite{huang2006large} as a tool for the study of $n$--player game when the number of players $n$ is very large. Since then it has been the subject of intensive study and use in many areas, see  \citeauthor*{carmona2018probabilisticI} \cite{carmona2018probabilisticI} for an overview. 

\medskip
This approach by MFG to solve this problem has been used by \citeauthor*{mastroliaAtale2019} \cite{mastroliaAtale2019}. In \cite{mastroliaAtale2019}, given a contract, the Principal computes the best response of the $representative$ Agent by solving a MFG problem through some BSDEs. Here, the $representative$ Agent can be seen as an entity characterizing all the Agents (the continuum of Agents). Then, with this best response, the Principal solves his problem. In their particular setting, \citeauthor*{mastroliaAtale2019} \cite{mastroliaAtale2019} show that this Principal's problem turns out to be a McKean--Vlasov stochastic control problem. Later, in order to analyze the electricity demand, \citeauthor*{HubertMean2021} \cite{HubertMean2021} extend some aspects of \cite{mastroliaAtale2019} by adding a source of external noise impacting all the players called in the literature common--noise, and by considering the case where the diffusion coefficient can be controlled. In order to provide some ideas for epidemic control, \citeauthor*{DayanikliOptimal2022}\cite{DayanikliOptimal2022} have used this MFG(BSDE)--Mckean Vlasov stochastic control problem  approach to treat the problem of Principal--Agent with a continuum of agents (see also \citeauthor*{Carmona2018FiniteStateCT}\cite{Carmona2018FiniteStateCT}).  This current paper puts itself in this problematic of Principal--Agent with multi--Agent/continuum of Agents.


\medskip
Over the years, a lot of efforts have been made in the literature for solving the problem of Principal with $n$ Agents ($n \ge 1$) and Principal with a continuum of Agents. These resolutions have been made under certain assumptions, either by providing explicit solutions or by showing existence results. With these resolutions in mind, some natural questions appear. In a certain sense, does the problem of Principal with $n$ Agents converges to the problem of Principal with a continuum of Agents when $n \to \infty$ ? How much the assumptions usually considered can be weaken ? This paper tries to provide some answer regarding these questions among other things.

\medskip
Although using quite different techniques, like \cite{possamai2019contracting,mastroliaAtale2019}, in a situation of moral hazard, in this paper we address the problem of the Principal with $n$ Agents / continuum of Agents using the framework and techniques of stochastic control theory. The Agents are considered to be in an environment of competitiveness, that is to say that they optimize their criteria through Nash equilibria.  We begin by proposing a formulation of the problem of Principal with $n$ Agents more general than the usual framework treated in the literature. We refer to \Cref{sec:finite-player-game} for the proper mathematical description. Indeed, in the presence of a common noise, the output/production of the Agents is potentially impacted by the contract proposed by the Principal. This may seem like a small addition at first glance, but it has never been considered in the literature and generates technical problems. Given a contract, the Principal computes the best response of the Agents by looking for the Nash equilibria. Having this best response, the Principal determines what is the best contract to offer that maximizes his utility. We then formulate the associated limit problem of the Principal with a continuum of Agents. With an adequate notion of contracts (see \Cref{section:limitproblem} for the proper definition), the problem turns out to be the following: given a contract, the Principal solves a specific MFG problem and obtains the best action of the representative Agent. After getting this best response, the Principal solves his problem which is a stochastic control of MFG solutions.

\medskip
Under some general assumptions, we start by showing that this limit problem admit a solution i.e. there is an optimal contract maximizing the utility of the Principal associated to the (well--defined) best response of the representative Agent. Next, we prove that, with the optimal contract, we can construct an (approximate) optimal contract for the problem of Principal with $n$ Agents with an associated (approximate) best response for the $n$ Agents. Also, given any sequence of (approximate) optimal contracts and best responses (indexes by $n$) for the problem of Principal with $n$ Agents, putting in an appropriate space, we are able to show that the sequence of  optimal contracts and best responses are relatively compact. In addition, any limit contract is optimal for the limit problem formulated and any limit best response is the best response of the representative Agent. All these results allow us to provide the convergence of the problem of the Principal with $n$ Agents to the problem of the Principal with a continuum of Agents when $n \to \infty$.

\medskip
If we put aside the proof given in \cite[Theorem 5.3.]{mastroliaAtale2019} in a very specific case/example, this paper seems to be the first to give a rigorous proof of the connection between the problem of the Principal with $n$ Agents and the problem of the Principal with a continuum of Agents when $n \to \infty$. Although we do not assume any specific dependence in the contracts, a notable fact is that our result leads to see that it is (approximately) optimal to consider contracts depending on the empirical distribution of the agents' output in the problem of Principal with $n$ Agents. Also, the assumptions we consider here seem to be more general than what is usually considered in the literature. Especially, our existence result in the limit problem is more general than the one mentioned in the literature. Our approach/technique turns out to be quite different from what is done in the literature of Principal--Agent. Namely, contrary to the use of BSDEs as done in \cite{mastroliaAtale2019}, although presented differently, our proof is inspired by the notion of relaxed solutions in stochastic control problem initiated by \citeauthor*{el1987compactification}\cite{el1987compactification} (see \citeauthor*{lacker2016general}\cite{lacker2016general} and \citeauthor*{djete2019general} \cite{djete2019general} for the MFG and Mckean--Vlasov control problem). It is worth mentioning that the framework and the techniques used here share some similitude to \citeauthor*{closed-loop-MFG_MDF} \cite{closed-loop-MFG_MDF}.
Despite the general aspects of our framework, our results suffer the fact that, given an Agent $i ( \le n)$, the Principal needs to only offer contract that is separated in one non--Markovian part which represents the impact of the production of all the other Agents in the contract and one Markovian part which is the impact of the production of the Agent $i$. Another limitation is the inability to consider a control in the diffusion coefficients. Despite this limitation, this paper seems to provide results that provide answers to some natural questions arising in the Principal--Agent problem not addressed so far.

\medskip
The paper is organised as follows. After briefly recalling some notations, \Cref{sec:set-resutls} describes the proper setting considered as well as the main results of this article. Namely, \Cref{sec:finite-player-game} formulates the problem of the Principal with $n$ Agents. Then, in \Cref{section:limitproblem}, the associated limit formulation  i.e. the problem of the Principal with a representative Agent is given. \Cref{subsec:main-results} is finally dedicated to the statement of the main results. All proofs are provided in \Cref{sec:proofs}.

\medskip
{\bf \large Notations}.
	$(i)$
	Given a {\color{black}Polish} space $(E,\Delta)$ and $p \ge 1,$ we denote by $\Pc(E)$ the collection of all Borel probability measures on $E$,
	and by $\Pc_p(E)$ the subset of Borel probability measures $\mu$ 
	such that $\int_E \Delta(e, e_0)^p  \mu(de) < \infty$ for some $e_0 \in E$.
	We equip $\Pc_p(E)$ with the Wasserstein metric $\Wc_p$ defined by
	\[
		\Wc_p(\mu , \mu') 
		~:=~
		\bigg(
			\inf_{\lambda \in \Lambda(\mu, \mu')}  \int_{E \x E} \Delta(e, e')^p ~\lambda( \mathrm{d}e, \mathrm{d}e') 
		\bigg)^{1/p},
	\]
	where $\Lambda(\mu, \mu')$ denotes the collection of all probability measures $\lambda$ on $E \x E$ 
	such that $\lambda( \mathrm{d}e, E) = \mu$ and $\lambda(E,  \mathrm{d}e') = \mu'( \mathrm{d}e')$. Equipped with $\Wc_p,$ $\Pc_p(E)$ is a Polish space (see \cite[Theorem 6.18]{villani2008optimal}). For any  $\mu \in \Pc(E)$ and $\mu$--integrable function $\varphi: E \to \R,$ we define
	\begin{align*}
	    \langle \varphi, \mu \rangle
	    =
	    \langle \mu, \varphi \rangle
	    :=
	    \int_E \varphi(e) \mu(\mathrm{d}e),
	\end{align*}
	and for another metric space $(E^\prime,\Delta^\prime)$, we denote by $\mu \otimes \mu^\prime \in \Pc(E \x E')$ the product probability of any $(\mu,\mu^\prime) \in \Pc(E) \x \Pc(E^\prime)$.

\medskip
    \noindent $(ii)$ Let $(S,\Delta)$ and $(S',\Delta')$ be two Polish space with the Borel $\sigma$--fields $\Fc$ and $\Fc'$ respectively. An application $V:S \to S'$ will be called universally measurable if $V$ is a measurable map from $(S, \Fc^U)$ to $(S',\Fc')$ where $\Fc^U$ is the universal completion of $\Fc$ i.e. $\Fc^U:=\cap_{\P \in \Pc(S)} \Fc^{\P}$ where $\Fc^\P$ is the $\P$--completed $\sigma$--field of $\Fc$.  
	Given a probability space $(\Om, \Hc, \P)$ supporting a sub-$\sigma$-algebra $\Gc \subset \Hc$ then for a Polish space $E$ and any random variable $\xi: \Om \longrightarrow E$, both the notations $\Lc^{\P}( \xi | \Gc)(\om)$ and $\P^{\Gc}_{\om} \circ (\xi)^{-1}$ are used to denote the conditional distribution of $\xi$ knowing $\Gc$ under $\P$.

\medskip
	\noindent $(iii)$	
	For any $(E,\Delta)$ and $(E',\Delta')$ two Polish spaces, we use $C_b(E,E')$ to denote the set of continuous functions $f$ from $E$ into $E'$ such that $\sup_{e \in E} \Delta'(f(e),e'_0) < \infty$ for some $e'_0 \in E'$.
	Let $\N^*$ denote the set of positive integers. Given non--negative integers $m$ and $n$, we denote by $\S^{m \x n}$ the collection of all $m \x n$--dimensional matrices with real entries, equipped with the standard Euclidean norm, which we denote by $|\cdot|$ regardless of the dimensions. 
	We also denote $\S^n:=\S^{n \times n}$, and denote by $0_{m \times n}$ the element in $\S^{m \times n}$ whose entries are all $0$, and by $\mathrm{I}_n$ the identity matrix in $\S^n$. 

\medskip
    \noindent $(iv)$
	Let $T > 0$ and $(\Sigma,\rho)$ be a Polish space, we denote by $C([0,T]; \Sigma)$ the space of all continuous functions on $[0,T]$ taking values in $\Sigma$.
	When $\Sigma=\R^k$ for some $k\in\N$, we simply write $\Cc^k := C([0,T]; \R^k),$ also we shall denote by $\Cc^{k}_{\Wc}:=C([0,T]; \Pc(\R^k)).$ 

\medskip
	With a Polish space $E$, we denote by $\M(E)$ the space of all Borel measures $q( \mathrm{d}t,  \mathrm{d}e)$ on $[0,T] \x E$, 
	whose marginal distribution on $[0,T]$ is the Lebesgue measure $ \mathrm{d}t$, 
	that is to say $q( \mathrm{d}t, \mathrm{d}e)=q(t,  \mathrm{d}e) \mathrm{d}t$ for a family $(q(t,  \mathrm{d}e))_{t \in [0,T]}$ of Borel probability measures on $E$.
	For any $q \in \M(E)$ and $t \in [0,T],$ we define $q_{t \wedge \cdot} \in \M(E)$ by
	\begin{equation}\label{eq:lambda}
	q_{t \wedge \cdot}(\mathrm{d}s, \mathrm{d}e) :=  q(\mathrm{d}s, \mathrm{d}e) \big|_{ [0,t] \x E} + \delta_{e_0}(\mathrm{d}e) \mathrm{d}s \big|_{(t,T] \x E},\; \text{for some fixed $e_0 \in E$.}
	\end{equation}
    We will say that $q \in \M_0(E)$ if $q \in \M(E)$ and there exists a Borel map $[0,T] \ni t \mapsto \overline{e}(t) \in E$ such that $q=\delta_{\overline{e}(t)}(\mathrm{d}e)\mathrm{d}t$.
	A map $h:[0,T] \x \R^k \x C([0,T];\Sigma) \x \M(E)$ is called a progressively Borel measurable if it verifies $h(t,x,\pi,q)=h(t,x,\pi_{t \wedge \cdot},q_{t \wedge \cdot}),$ for any $(t,x,\pi,q) \in [0,T]\x \R^k \x C([0,T];\Sigma) \x \M(E).$

\section{ Setup and main results} \label{sec:set-resutls}

The general assumptions used throughout this paper are now formulated. The dimension $d \ge 1$, the nonempty Polish spaces $(A,\rho_A)$, $(\Er,\rho_\Er)$ and $(\Ir,\rho_\Ir)$, and the {\color{black}time horizon} $T>0$ are fixed. We shall denote $\Pc_A$ for the space of all Borel probability measures on $\R^d \x A$, i.e. $\Pc_A:=\Pc(\R^d \x A).$ Also, we set $p \ge 2,$ $\nu \in \Pc_{p'}(\R^d)$ with $p'>p,$ and the probability space $(\Om,\H:=(\Hc_t)_{t \in [0,T]},\Hc,\P)$\footnote{ \label{footnote_enlarge}The probability space $(\Om,\H,\P)$ contains as many random variables as we want in the sense that: each time we need a sequence of independent uniform random variables or Brownian motions, we can find them on $\Om$ with the distribution $\P$ without mentioning an enlarging of the space. }. We are given the following progressively Borel measurable functions:
	\[
		\left(b, L \right):[0,T] \x \R^d \x \Cc_{\Wc} \x A \longrightarrow \R^d \x \R^d,\;\;\left(\overline{b}, \overline{L}, \overline{L}_{\Pr} \right):[0,T] \x \Pc_A \x \Er \longrightarrow \R^d \x \R \x \R\;\;\mbox{and}\;\;\sigma:[0,T] \x \R^d \to \S^{d},
	\]
 and the Borel maps  $U:\R \to \R$, $\Upsilon:\R^d \to \R$, $(g,g_{\Pr}):\R^d \x \Ir \to \R \x \R$ and $(\overline{g},\overline{g}_{\Pr}):C_{\Wc} \x \Er \to \R \x \R$

    \begin{assumption} \label{assum:main1} 
		
		$(i)$ $A$, $\Er$ and $\Ir$ are nonempty convex compact {\color{black}Polish sets}; 
		
		
		\medskip
		$(ii)$ The map $[0,T] \x \R^d \x \Cc_{\Wc} \x A \x \Pc_A \x \Er \ni (t,x,\pi,a,m,e) \mapsto \left( b(t,x,\pi,a),\overline{b}(t,m,e), \sigma(t,x) \right) \in \R^d \x \R^d \x \S^{d}$ is bounded Lipschitz in $(x,\pi,m)$ uniformly in $(t,a,e)$, and continuous in $(x,\pi,a,m,e)$ for any $t$;

    \medskip
		$(iii)$ {\rm Non--degeneracy condition:} 
        $\inf_{(t,x)} \sigma \sigma^{\top}(t,x) >0$ and $\sigma_0 \in \S^{d}$ is an invertible matrix; 
		
    \medskip
		$(iv)$ 
        The map $ (t,x,\pi,a,m,e,i,v,u) \mapsto \left( L(t,x,\pi,a), (\overline{L}, \overline{L}_{\Pr})(t,m,e), (g,g_{\Pr})(x,i), (\overline{g},\overline{g}_{\Pr})(\pi,e),\Upsilon(v),U(u) \right)$  is continuous in $(x,\pi,a,m,e,i,v,u)$ for each $t$, and  with $p$--linear growth in $(x,\pi,m,v,u)$ uniformly in $(t,a,e,i)$.
		
	\end{assumption}
	
	\begin{remark}
	   The maps $(b,L,g,\overline{b},\overline{L},\overline{g},\sigma)$ are related to the Agents and the maps $(g_{\Pr},\overline{L}_{\Pr}, \overline{g}_{\Pr},\Upsilon, U)$ will be associated to the Principal. Some of the conditions can be relaxed especially the fact that the maps $(b,\overline{b},\sigma)$ are bounded and the sets $(A,\Er,\Ir)$ compact. We have chosen these conditions to allow the reader to concentrate on the main points instead of being confused by long classical technicalities generated by these considerations.
   \end{remark}

\subsection{Principal and $n$ Agents } \label{sec:finite-player-game}

Let $n \ge 1$. Before introducing the proper mathematical framework we will consider here. Let us first defining what we call controls and contracts in the situation of the problem of Principal with $n$ Agents.

\medskip
$\boldsymbol{\rm Controls}$: We denote by $\Ac_n$ the set of all progressively Borel measurable maps $\alpha:[0,T] \x (\Cc^d)^n \to A$. Any element $n$--tuple $\alphab=(\alpha^1,\cdots,\alpha^n) \in (\Ac_n)^n$ belongs to $\Dc_{n}$ if there exists a progressively Borel measurable map  $[0,T] \x \R^d \x \Cc_{\Wc} \ni (t,x,\pi) \mapsto \widehat{\alpha}(t,x,\pi) \in A $ Lipschitz in $(x,\pi)$ uniformly in $t$ s.t. $\alpha^i(t,\boldsymbol{x})=\widehat{\alpha}(t,x^i(t),\pi^n)$ for each $t \in [0,T]$ and $1 \le i \le n$ where $\boldsymbol{x}=(x^1,\dots,x^n) \in (\Cc^d)^n$ and $\pi^n(t):=\frac{1}{n} \sum_{i=1}^n \delta_{x^i(t)}$.

\medskip
$\boldsymbol{\rm Contracts}$: We will say $\Cf^n:=(\phi^n,\xi^n,\aleph^n)$ is a contract if: $\phi^n:\R^d \to \Ir$ and $\xi^n:(\Cc^d)^n \to \Er$ are Borel maps, and $\aleph^n:[0,T] \x (\Cc^d)^n \to \Er$ is a progressively Borel measurable map. A contract $\Cf^n:=(\phi^n,\xi^n,\aleph^n)$ will be called ${distributed}$ contract if there exists  a Borel map $[0,T] \x \Cc_{\Wc} \ni (t,\pi) \mapsto \left( \widehat{\aleph}^n(t,\pi), \widehat{\xi}^n(\pi) \right) \in \Er \x \Er$ Lipschitz in $\pi$ uniformly in $t$ s.t. $\aleph^n(t,\boldsymbol{x})=\widehat{\aleph}^n(t,\pi^n)$ and $\xi^n(\boldsymbol{x})=\widehat{\xi}^n(\pi^n)$ where $\boldsymbol{x}=(x^1,\dots,x^n) \in (\Cc^d)^n$ and $\pi^n(t):=\frac{1}{n} \sum_{i=1}^n \delta_{x^i(t)}$.

\begin{remark}
    A contract $\Cf^n=(\phi^n, \xi^n,\aleph^n)$ offered by the Principal consists, on the one hand, of $(\phi^n,\xi^n)$ which represents the final payments received at time $T$ by the Agents. And, on the other hand, of $\aleph^n$ which is the instantaneous payments given to the Agents at time $t \in [0,T]$.
\end{remark}

\subsubsection{The problem of the $n$ Agents} \label{sec:n_agents}
On the filtered probability space $(\Om,\H,\Hc,\P),$ let $(W^i)_{i \in \N^*}$ be a sequence of independent $\R^d$--valued $\H$--adapted  Brownian motions, $B$ be an $\R^d$--valued $\H$--adapted Brownian motion and $(\iota^i)_{i \in \N^*}$ a sequence of iid $\Hc_0$--random variables of law $\nu.$ Besides, $(W^i)_{i \in \N^*},$ $B$ and $(\iota^i)_{i \in \N^*}$ are independent. Then, given a contract ${\Cf}^n:=(\phi^n,\xi^n,\aleph^n)$ and the controls $\alphab^n:=(\alpha^{1,n},\dots,\alpha^{n,n}) \in (\Ac_n)^n$, we denote by $\Xbb^{\aleph^n,\alphab^n}:=\Xbb=(X^1_{\cdot},\dots,X^n_{\cdot})$ the process satisfying: for each $i \in \{1,\dots,n\},$ $X^i_0:=\iota^i$, $\E\big[\sup_{t \in [0,T]}|X^i_t|^p \big]< \infty,$ and $\P$--a.e.
    \begin{align} \label{eq:N-agents_StrongMV_CommonNoise-law-of-controls}
        \mathrm{d}X^i_t
        =
        \overline{b}\left(t, \overline{\varphi}^{n}_{t}[\alphab^n], \aleph^n(t,\Xbb) \right)
        +
        b\left(t,X^i_{t}, \varphi^{n}[\alphab^n] ,\alpha^{i,n}(t,\Xbb) \right) \;\;\mathrm{d}t 
        +
        \sigma(t,X^i_t) \mathrm{d}W^i_t
        +
        \sigma_0 \mathrm{d}B_t
    \end{align}
	with 
	\[
	    \varphi^{n}_{t}[\alphab^n](\mathrm{d} x) := \frac{1}{n}\sum_{i=1}^n \delta_{X^{i}_{t}}(\mathrm{d} x)\;\;\mbox{and}\;\;
	    \overline{\varphi}^{n}_{t}[\alphab^n](\mathrm{d} x, \mathrm{d} a) := \frac{1}{n}\sum_{i=1}^n \delta_{\big(X^{i}_{t},\;\alpha^{i,n}(t,\Xbb) \big)}(\mathrm{d} x, \mathrm{d} a)
	    ,~
	    \mbox{for all}
	    ~~
	    t \in [0,T].
	\] 
The reward value of player $i$ associated with the contract $\Cf^n=(\phi^n,\xi^n,\aleph^n)$ and the control rule/strategy $\alphab^n:=(\alpha^{1,n},\dots,\alpha^{n,n})$ is then defined by
	\begin{align*}
	    &J^{\Cf^n}_{n,i}(\alphab^n)
     \\
	    &:=
	    \E \left[
        \int_0^T 
        \overline{L}\left(t, \overline{\varphi}^{n}_{t}[\alphab^n],\aleph^n(t,\Xbb) \right)
        +
        L\left(t,X^{i}_t,{\varphi}^{n}[\alphab^n],\alpha^{i,n}(t,\Xbb) \right)\;\mathrm{d}t 
        +
        \overline{g} \left( {\varphi}^{n}[\alphab^n],\xi^n(\Xbb) \right)
        + 
        {g} \left( X^i_T,\phi^n(X^i_T) \right)
        \right].
	\end{align*}

\begin{remark}
    {\rm $(i)$} By using Girsanov's Theorem, we can see that {\rm\Cref{eq:N-agents_StrongMV_CommonNoise-law-of-controls}} is uniquely defined in distribution. However, since we only consider that the maps $(\alpha^{1,n},\cdots,\alpha^{n,n})$ and $\aleph^n$ are Borel measurable, the strong well--posedness of {\rm\Cref{eq:N-agents_StrongMV_CommonNoise-law-of-controls}} is unclear. The process $\Xbb$ should therefore be seen as a weak solution. We may need to extend $(\Om,\H,\P)$ to write $\Xbb$, but to avoid heavy notations, we assumed that $(\Om,\H,\P)$ is such that we do not need any further extensions $($see {\rm \Cref{footnote_enlarge}}$)$.
    
    \medskip
    {\rm $(ii)$}The processes $\Xbb$ represent the production of all the $n$ Agents and $\alphab^n$ their associated controls.
    In the final payment $(\phi^n,\xi^n)$, given an Agent $i$, $\phi^n$ is the Markovian part involving the terminal production of the Agent i.e. $X^i_T$ and $\xi^n$ is the non--Markovian part involving the production of all the Agents i.e. $\Xbb$. However, the instantaneous payment $\aleph^n$ is only non--Markovian and is composed of the production of all the Agents.

    \medskip
    {\rm $(iii)$} It is worth mentioning that, even without the presence of the empirical distribution $\overline{\varphi}^{n}[\alphab^n]$, the processes $(X^1,\cdots,X^n)$ are in interactions because of the presence of the common noise and especially of $\aleph^n$. Besides, $\aleph^n$ is not necessary a map of the empirical distribution $\overline{\varphi}^{n}[\alphab^n]$, but it just depends on $(X^1,\cdots,X^n)$.
\end{remark}

    We now give the optimization criterion considered between the Agents, that is to say what we will call approximate Nash equilibria.
	
\begin{definition}{\rm((approximate) equilibria)} \label{def:Nplayers--equilibria}

\medskip    
    Let $n \ge 1$, $\varepsilon \in \R_{+},$ and the contract $\Cf^n:=(\phi^n,\aleph^n,\xi^n)$. We will say that a control rule/strategy $(\alpha^{1,n},\dots,\alpha^{n,n}) \in (\Ac_n)^n$ is an  $\varepsilon$--{\rm Nash equilibrium} if 
	    \[
	        J^{\Cf^n}_{n,i}(\alpha^{1,n},\dots,\alpha^{n,n}) \ge \sup_{\beta \in \Ac_n} J^{\Cf^n}_{n,i}\big( \alpha^{1,n},\dots,\alpha^{i-1,n},\beta,\alpha^{i+1,n},\dots,\alpha^{n,n} \big)-\varepsilon,\;\mbox{for each}\;i \in \{1,\dots,n\}.
	    \]    
\end{definition}
We set
\begin{align*}
    \mbox{\rm NE}[\Cf^n,\varepsilon]
    :=
    \left\{ \alphab^n=(\alpha^{1,n},\cdots,\alpha^{n,n})\mbox{ an $\varepsilon$--Nash equilibrium given }\Cf^n=(\phi^n,\xi^n,\aleph^n) \right\}.
\end{align*}
In addition, we denote by $\mbox{\rm NE}^{\rm dist}[\Cf^n,\varepsilon]$ the subset of $\mbox{\rm NE}[\Cf^n,\varepsilon]$ containing all the $n$--tuple $\alphab^n$ belonging to $\Dc_n$.

\begin{remark}
    If the contract $(\phi^n,\xi^n,\aleph^n)$ is a $distributed$ contract independent of $n$, by the help of the {\rm MFG} theory, the set ${\rm NE}^{\rm dist}[\Cf^n,\varepsilon]$ turns out to be non--empty $($see for instance {\rm\cite[Theorem 2.12.]{closed-loop-MFG_MDF}}$)$.  However, as the map $(\xi^n,\aleph^n)$ is only a Borel measurable map of $n$ variables, it is worth mentioning that the set ${\rm NE}[\Cf^n,\varepsilon]$ can be empty.  
\end{remark}

\subsubsection{The Principal's problem} \label{subsub-principal_n}

In case ${\rm NE}[\Cf^n,\varepsilon]$ is non--empty, most of the time ${\rm NE}[\Cf^n,\varepsilon]$ is not a singleton, it contains many solutions. In this situation, the choice of the controls of the Agents is given to the Principal as it is done in the literature (see for instance \cite{possamai2019contracting,mastroliaAtale2019}). The interpretation is that, first, the Principal must know whether the Agents will agree to the proposed contract and anticipate the response/behaviour of the Agents. Second, once the contract is offered, the Principal also recommends a behavior to the Agents. However, there is a minimum level of utility below which Agents will not consider the contract, it is called the reservation utility and denoted here $R_0-\varepsilon_{0}.$  Keeping this mechanism in mind, we then formulate the problem of the Principal.

\medskip
We fix $R_0 \in \R$. Let $\varepsilon, \varepsilon_{0} \ge 0$, we define $\underline{\rm NE}[\Cf^n,\varepsilon,\varepsilon_0]$ the set of $\alphab \in {\rm NE}[\Cf^n,\varepsilon]$ s.t. $J^{\Cf^n}_{n,i}(\alphab) \ge R_0 - \varepsilon_0$ for each $1 \le i \le n$. In a same way we define $\underline{\rm NE}^{\rm dist}[\Cf^n,\varepsilon,\varepsilon_0]$. 
We now denote by $\Xi_n[\varepsilon,\varepsilon_0]$ the set of admissible contracts $\Cf^n=(\phi^n,\xi^n,\aleph^n)$ verifying: $\underline{\rm NE}[\Cf^n,\varepsilon,\varepsilon_0]$ is non--empty. Besides,  $\Cf^n=(\phi^n,\xi^n,\aleph^n) \in \Xi_n[\varepsilon,\varepsilon_0]$ will belong to $\Xi_{n}^{\rm dist}[\varepsilon,\varepsilon_0]$ if $\Cf^n$ is a $distributed$ contract and $\underline{\rm NE}^{\rm dist}[\Cf^n,\varepsilon,\varepsilon_0]$ is non--empty.

\medskip
Let us consider the two problems faced by the Principal
\begin{align*}
    V_{n,\Pr}[\varepsilon,\varepsilon_0]
    :=
    \sup_{\Cf^n \in \Xi_n[\varepsilon,\varepsilon_0]} \;\;\sup_{\alphab^n \in \underline{\rm NE}[\Cf^n,\varepsilon,\varepsilon_0]} J_{n,\Pr}^{{\color{black}\alphab^n}}(\Cf^n)\;\;\;\mbox{and}\;\;\; V^{\rm dist}_{n,\Pr}[\varepsilon,\varepsilon_0]
    :=
    \sup_{\Cf^n \in \Xi^{\rm 
 dist}_{n}[\varepsilon,\varepsilon_0]} \;\;\sup_{\alphab^n \in \underline{\rm NE}^{\rm dist}[\Cf^n,\varepsilon,\varepsilon_0]} J_{n,\Pr}^{{\color{black}\alphab^n}}(\Cf^n)
\end{align*}
where the reward of the Principal is
\begin{align*}
    J_{n,\Pr}^{{\color{black}\alphab^n}}(\Cf^n)
    :=
    &\E \left[ U \left( \frac{1}{n} \sum_{i=1}^n \Upsilon \left(X^i_T\right) - g_{\Pr}\left(X^i_T,\phi^n(X^i_T) \right)
    -\overline{g}_{\Pr} \left(\varphi^n[\alphab^n],\xi^n(\Xbb) \right)
    -
    \int_0^T \overline{L}_{\Pr} \left(t,\overline{\varphi}^{n}_{t}[\alphab^n],\aleph^n(t,\Xbb) \right)\;\mathrm{d}t\right) \right].
\end{align*}

\begin{remark} \label{rm:utility_reservation}
    The actual reservation utility is $R_0$. Nevertheless, we consider that the Agents are willing to go a little below i.e. $R_0- \varepsilon_0$. The interpretation is that the Principal is in a strong position to negotiate and compromise with the finite number of Agents. However, when the number of Agents is growing, the Agents are less and less willing to budge from  their positions as they are feeling strong, therefore $\varepsilon_0$ becomes more and more small. In the end, in the case of an infinite number of Agents that we will see below, they will no longer be willing to go below $R_0$.
\end{remark}


\subsection{Limit problem : Stochastic Control of Mean Field Games Solutions} \label{section:limitproblem}

For convenience, we are still on the probability space $(\Om,\H,\P)$. On this space, we consider $(W,B)$ an $\R^d \x \R^d$--valued $\H$--Brownian motion and an $\Hc_0$--random variable $\iota$ s.t. $\Lc(\iota)=\nu$.

\medskip
We set $\M:=\M(\Pc_A \x \Er)$ and $\M_0:=\M_0(\Pc_A \x \E)$. For any progressively Borel measurable map $h:[0,T] \x \R^d \x \Cc_\Wc \x \M_0 \to \R^d$ and any $\left(t,x,\pi,q=\delta_{\overline{e}(t)}(\mathrm{d}e)\mathrm{d}t \right) \in [0,T] \x \R^d \x \Cc_\Wc \x \M_0$, for notational convenience, we will write $h(t,x,\pi,\overline{e})$ instead of $h(t,x,\pi,q)$. Also, whenever we write $h(t,x,\pi,\overline{e})$ where $[0,T] \ni t \mapsto \overline{e}(t) \in \Pc_A \x \Er$ is a Borel map, it should be understood as $h\left(t,x,\pi,\delta_{\overline{e}(t)}(\mathrm{d}e)\mathrm{d}t \right)$.

\medskip
$\boldsymbol{\rm Controls}$: In the case of the limit problem, we denote by $\Ac$ the set of all progressively Borel measurable maps $\alpha:[0,T] \x \R^d \x \M_0 \to A$. 

\medskip
$\boldsymbol{\rm Contracts}$: We will say $\Cf:=(\phi,\xi,\aleph)$ is a contract if: $\phi:\R^d \to \Ir$ and $\xi:\M_0 \to \Er$ are Borel maps, and 
$\aleph=(\aleph_t)_{t \in [0,T]}$ is an $\Er$--valued $\H$--predictable process. Besides, $(\aleph,B)$, $\iota$ and $W$ are independent.

\begin{remark}
    To make the connection with the framework of the Principal and $n$ Agents, here for a contract $\Cf=(\phi,\xi,\aleph)$, $(\phi,\xi)$ represents the final payment and $\aleph$ is the instantaneous fees. 
\end{remark}


\subsubsection{The representative Agent problem}

Given $\varepsilon \ge 0$ and  a contract $\mathfrak{C}:=(\phi,\xi,\aleph)$, we introduce the set of (approximate) MFG solutions $\mbox{{\rm MFG}}[\Cf,\varepsilon]$. We say $(\alpha,\mub)$ belongs to $\mbox{{\rm MFG}}[\Cf,\varepsilon]$ if: $\alpha \in \Ac$,  $\mub=(\mub_t)_{t \in [0,T]}$ is an $\Pc_A$--valued $\H$--predictable process,
    \begin{align*}
        J^{\Cf}_{A,\mub}({\color{black}\alpha}) \ge J^{\Cf}_{A,\mub}(\beta)-\varepsilon,\;\;\mbox{for any }\beta \in \Ac
    \end{align*}
    with the reward $J^{\Cf}_{A,\mub}$ defined by
    \begin{align*}
        J_{A,\mub}^{\Cf}(\beta)
    :=
    \E \bigg[
        \int_0^T 
        \overline{L}\left(t,\mub_t,\aleph_t \right) 
        +
        L\left(t,X^{\beta}_t,\mu_t ,\beta(t,X^\beta_t,\mub,\aleph) \right)\mathrm{d}t 
        +
        \overline{g} \left( \mu_T, \xi(\mub,\aleph) \right)
        + 
        g \left( X^\beta_T,\phi(X^\beta_T) \right)
        \bigg],
    \end{align*}
    and $\E^\P\left[\sup_{t \in [0,T]} |X^\beta_t|^p \right] < \infty$, $X^\beta_0:=\iota$, $\P$--a.e.
    \begin{align*}
        \mathrm{d}X^\beta_t
            =
            \overline{b}\left(t,\mub_t,\aleph_t\right)
            +
            b\left(t,X^\beta_t,\mu_t,\beta(t,X^\beta_t,\mub,\aleph)\right)\mathrm{d}t + \sigma(t,X^\beta_t)\mathrm{d}W_t + \sigma_0\mathrm{d}B_t,\;\;\;\mu_t=\mub_t(\mathrm{d}x,A).
    \end{align*}
    In addition, for $R:=\delta_{\left(\mub_t, \aleph_t \right)}(\mathrm{d}m,\mathrm{d}e)\mathrm{d}t$, the variables $\left(B,R\right)$, $W$ and $X_0$ are independent,   the filtration $\G$ is defined by               $\G=\left(\Gc_t:=\sigma\{R_{t \wedge \cdot}, B_{t \wedge \cdot}\} \right)_{t \in [0,T]}$ 
    with, $\mathrm{d}\P \otimes \mathrm{d}t$--a.e. 
    \begin{align}\label{eq:consitency1}
        \Lc(X^\alpha_t|\Gc_T)=\mu_t,\;\;\E^\P \left[ L \left( t, X^\alpha_t, \Lc\left(X^\alpha_t|\Gc_t \right),\alpha(t,X^\alpha_t,\mub,\aleph) \right) | \Gc_T \right]
        =
        \int_{\R^d \x A} L \left(x,\mu_t,a \right)\mub_t(\mathrm{d}x,\mathrm{d}a)
    \end{align}
    and
    \begin{align} \label{eq:consitency2}
        \E^\P \left[ \varphi(X^\alpha_t) b \left( t, X^\alpha_t, \Lc\left(X^\alpha_t|\Gc_t \right),\alpha(t,X^\alpha_t,\mub,\aleph) \right) | \Gc_T \right]
        =
        \int_{\R^d \x A} \varphi(x)b \left(x,\mu_t,a \right)\mub_t(\mathrm{d}x,\mathrm{d}a)\;\;\mbox{for all $\varphi \in C_c$.}
    \end{align}

When $\varepsilon =0,$ we will simply write $\mbox{{\rm MFG}}[\Cf]$ instead of $\mbox{{\rm MFG}}[\Cf,0]$. The elements of $\mbox{{\rm MFG}}[\Cf, \varepsilon]$ are in the spirit of $weak$ solutions of MFG in closed--loop setting (see \citeauthor*{Lacker-closedloop}\cite{Lacker-closedloop}, \citeauthor*{leflemclosed2023}\cite{leflemclosed2023}, \cite{closed-loop-MFG_MDF}). We say that the map $(\alpha, \xi,\aleph)$ belong to $\Cc_s$ if $[0,T] \x \R \x \Cc_{\Wc} \ni (t,x,\pi) \mapsto  \left(\alpha(t,x,\pi), \aleph(t,\pi), \xi(\pi) \right) \in A \x \Er \x \Er $ is Lipschitz in $(x,\pi)$ uniformly in $t$. We use this definition to introduce the set of (approximate) \underline{strong solution} $\mbox{{\rm MFG}}_S[\Cf,\varepsilon]$ by
\begin{align*}
    \mbox{{\rm MFG}}_S[\Cf,\varepsilon]
    :=
    \Big\{ (\alpha,\mub) \in  \mbox{{\rm MFG}}[\Cf,\varepsilon]:\;\exists\;(\widehat{\alpha},\widehat{\xi},\widehat{\aleph}) \in \Cc_s\;\mbox{s.t.}\;&\mub_t=\Lc\left(X^\alpha_t,\widehat{\alpha}(t,X^\alpha_t,\mu) |B_{t \wedge \cdot}\right),
    \\
    &\aleph_t=\widehat{\aleph}(t,\mu),\;\xi(\mub,\aleph)=\widehat{\xi}(\mu)\;\;\;\mbox{ 
 a.e. }(\om,t) \in \Om \x [0,T]
    \Big\}.
\end{align*}

\begin{remark}
    {\rm$(i)$} Notice that, for $(\alpha, \mub) \in \mbox{{\rm MFG}}[\Cf,\varepsilon]$, the $($approximate$)$ optimal control $\alpha$ does not necessary verify the classical consistency condition of the {\rm MFG} of controls literature i.e. $\mub_t=\Lc(X^\alpha_t,\alpha(t,X^\alpha_t,\mub,\aleph)|\Gc_T)$ $($see {\rm \citeauthor*{LackerCarmona-Ext}\cite{LackerCarmona-Ext}, \citeauthor*{kobeissiOnclassical2022} \cite{kobeissiOnclassical2022}, \cite{closed-loop-MFG_MDF}} $)$. We only need to check that $\alpha$ satisfies {\rm\Cref{eq:consitency1}} and {\rm\Cref{eq:consitency2}} which is associated to the maps $b$ and $L$.  

\medskip
    {\rm $(ii)$} For any contract $\Cf$, under some conditions, we will see that $\mbox{{\rm MFG}}[\Cf]$ is non--empty. Besides, for $[0,T] \x \Cc_{\Wc} \ni(t,\pi) \mapsto \left(\overline{e}(t,\pi), z(\pi) \right)\in \Er \x \Er$ a Lipschitz map in $\pi$ uniformly in $t$, there is $(\alpha,\mub) \in \mbox{{\rm MFG}}_S[\Cf,\varepsilon]$ for each $\varepsilon >0$ where $\Cf=(\phi,\xi,\aleph)$ with $\aleph_t=\overline{e}(t,\mu)$ and $\xi(\mub,\aleph)=z(\mu)$.
\end{remark}

\subsubsection{The Principal's problem} For any $\varepsilon, \varepsilon_0 \ge 0$, with the reservation utility $R_0$, like in \Cref{subsub-principal_n}, we introduce $\underline{\rm MFG}[\Cf,\varepsilon,\varepsilon_0]$ the set of $(\alpha,\mub) \in \mbox{{\rm MFG}}[\Cf,\varepsilon]$ s.t. $J^{\Cf}_{A,\mub}(\alpha) \ge R_0 - \varepsilon_0$. Similarly we define $\underline{\rm MFG}_S[\Cf,\varepsilon,\varepsilon_0]$.   We now consider $\Xi[\varepsilon,\varepsilon_0]$ the set of admissible contracts $\Cf=(\phi,\xi,\aleph)$ i.e. $\underline{{\rm MFG}}[\Cf,\varepsilon,\varepsilon_0]$ is non--empty. Also, $\Xi_S[\varepsilon,\varepsilon_0]$ will denote the subset of $\Cf \in \Xi[\varepsilon,\varepsilon_0]$ s.t. $\underline{\rm MFG}_S[\Cf,\varepsilon,\varepsilon_0]$ is non--empty. When $\varepsilon=\varepsilon_0=0$, we simply write $\Xi$, $\underline{{\rm MFG}}[\Cf]$, $\Xi_S$ and $\underline{{\rm MFG}}_S[\Cf]$.  
Let us consider the following problems
\begin{align} \label{eq:principal_pr}
    V_{\Pr}
    :=
    \sup_{\Cf \in \Xi} \;\;\sup_{{\color{black}(\alpha,\mub)} \in \underline{{\rm MFG}}[\Cf]} J_{\Pr}^{\alpha,\mub}(\Cf)~~~~\mbox{and}~~~~V_{\Pr}^S[\varepsilon,\varepsilon_0]
    :=
    \sup_{\Cf \in \Xi_S[\varepsilon,\varepsilon_0]} \;\;\sup_{{\color{black}(\alpha,\mub)} \in \underline{{\rm MFG}}_S[\Cf,\varepsilon]} J_{\Pr}^{\alpha,\mub}(\Cf)
\end{align}
where
\begin{align*}
    J_{\Pr}^{\alpha,\mub}(\Cf)
    :=\E \left[ U \left( \E\left[  \Upsilon\left(X^\alpha_T \right) - g_{\Pr} \left(X^\alpha_T,\phi(X^\alpha_T) \right) 
    -\overline{g}_{\Pr} \left(\mu, \xi(\mub,\aleph) \right)
    -
    \int_0^T \overline{L}_{\Pr} \left(t, \mub_t,\aleph_t \right)\;\mathrm{d}t \Big|\Gc_T \right]\right) \right].
\end{align*}

\begin{remark}
    To follow up on the discussion started in {\rm \Cref{rm:utility_reservation}}, in $V_{\Pr}$ the representative Agent, representing an infinite number of Agents, is considering a utility greater than or equal to $R_0$ and will not agree to drop below this level. However, in $V^S_{\Pr}$, the representative Agent will be willing to compromise by accepting $R_0- \varepsilon_0$. The idea is that the representative Agent agrees to go below $R_0$ if he obtains contracts with a simpler structure than in $V_{\Pr}$, this is the case in $V^S_{\Pr}$.
\end{remark}

\subsection{Main limit results} \label{subsec:main-results}

In this section, we provide the main results of this paper. We start by an existence result for the problem of the Principal with a representative Agent. For this purpose, let us consider additional assumptions. Since $\Pc(A)$ is compact, we say $m \in \mathrm{L}^\infty(\R^d;\Pc(A))$ if $m:\R^d \to \Pc(A)$ is just a Borel measurable map. For any $\eta \in \Pc(\R^d)$ and $m \in \mathrm{L}^\infty(\R^d;\Pc(A))$, we denote by $[\eta,m] \in \Pc(\R^d \x A)$ defined by $[\eta,m]:=m(x)(\mathrm{d}a)\eta(\mathrm{d}x)$. Also, for $\pi \in \Cc_{\Wc}$, we set $\Kr_{b,\pi,m}(x):=\int_A b(t,x,\pi,a)m(x)(\mathrm{d}a)$ for a.e. $x \in \R^d$. Notice that $K_{b,\pi,m}:\R^d \to \R^d$ is a bounded Borel map i.e. $K_{b,\pi,m} \in \mathrm{L}^\infty(\R^d;\R^d)$.

\begin{assumption} \label{assum:convexity}
The utility map of the Principal $U:\R \to \R$ is non--decreasing and concave. For any $(t,x,\pi) \in [0,T] \x \R^d \x \Cc_{\Wc}$,  $\left\{
        \left( b(t,x,\pi,a),\;z \right):\;\;z \le L(t,x,\pi,a),\;a \in A
    \right\}$ and
\begin{align*}
    \bigg\{
        \Big( K_{b,\pi,m},\;\left(\overline{b}, \overline{L} \right)\left(t, [\pi(t),m],e \right),\;z, \;\overline{z}_{\Pr} \Big):\;&m \in \mathrm{L}^\infty\left(\R^d;\Pc(A)\right)\;\;\mbox{and}\;\; e \in \Er\;\;\mbox{with}\;\;
        \\
        &z \le \langle L(t,\cdot,\pi,\cdot),[\pi(t),m] \rangle ,\;\;\;\overline{z}_{\Pr} \le - \overline{L}_{\Pr} \left( t, [\pi(t),m], e \right)
    \bigg\}
\end{align*}
are closed convex sets. Furthermore, for each $(\pi,x) \in \Cc_{\Wc} \x \R^d$,
\begin{align*}
    \left\{
        \left( \overline{g}(\pi,e),\;\overline{z}_{\Pr} \right):\;\;\overline{z}_{\Pr} \le -\overline{g}_{\Pr}(\pi,e),\;e \in \Er
    \right\}\;\;\mbox{and}\;\;\left\{
        \left( g(x,i),\;z_{\Pr} \right):\;\;z_{\Pr} \le -g_{\Pr}(x,i),\;i \in \Ir
    \right\}
\end{align*}
are closed convex sets.
\end{assumption}

\begin{remark}
    These assumptions will be used in the same spirit as in {\rm \citeauthor*{Filippov1962}\cite{Filippov1962}} and {\rm\citeauthor*{Roxin1962}\cite{Roxin1962}}. A simple situation where {\rm \Cref{assum:convexity}} is verified is when: the maps $a \mapsto b(t,x,\pi,a)$ and $m \mapsto \left(K_{b,\pi,m},\;\left(\overline{b}, \overline{L} \right)\left(t, [\pi(t),m],e \right) \right)$ are affine and, $(a,m,e,i,u) \mapsto \left( L(t,x,\pi,a),\; - \overline{L}_{\Pr} \left( t, [\pi(t),m], e \right),\; -\overline{g}_{\Pr}(\pi,e), -g_{\Pr}(x,i) \right)$ is concave $($we considered the product order to check the concavity inequality$)$.
\end{remark}

\begin{theorem} \label{thm:existence_mfg}
    Under {\rm \Cref{assum:main1}} and {\rm \Cref{assum:convexity}}, for any contract $\Cf=(\phi,\xi,\aleph)$, the set ${\rm MFG}[\Cf]$ is non--empty.
\end{theorem}

\begin{remark}
    {\rm \Cref{thm:existence_mfg}} is in fact an extension of {\rm \citeauthor*{djete2020some}\cite[Theorem 7.2.4.]{djete2020some}}, however the proof given here is simpler and does not use any discretization arguments as is usually the case with weak solutions of {\rm MFG} with common noise $($see {\rm \cite{Lacker_carmona_delarue_CN}}$)$.  Putting aside the dependence on the contract, the assumptions considered in {\rm \Cref{thm:existence_mfg}} seems general in comparison to the literature regarding the wellposedness of {\rm MFG} solutions $(${\rm MFG} of controls solutions to be more precise, see for instance {\rm \citeauthor*{zengOptimal2022}\cite{zengOptimal2022}, \citeauthor*{JaberEquilibrium2023}\cite{JaberEquilibrium2023}}$)$.   
\end{remark}

\begin{theorem} \label{thm:existence_contract}
    Let {\rm \Cref{assum:main1}} and {\rm \Cref{assum:convexity}} hold true. If $R_0$ is s.t. $\Xi$ is non--empty then there exist progressively Borel measurable map $[0,T] \x \R^d \x \Cc_{\Wc} \x \Pc_A \ni (t,x,\pi,m)  \mapsto  \widehat{\alpha}(t,x,\pi,m) \in \Er \x A$, a contract $\Cf^\star=(\phi^\star,\xi^\star,\aleph^\star) \in \Xi$ and $\left(\alpha^\star,\mub^\star \right) \in \underline{\rm MFG}[\Cf^\star]$ verifying: 
    \begin{align*}
        \alpha^\star(t,X^{\alpha^\star}_t,\mub^\star,\aleph^\star)=\widehat{\alpha}(t,X^{\alpha^\star}_t,\mu^\star,\mub^\star_t)\;\;\;\mathrm{d}t \otimes \mathrm{d}\P\mbox{--a.e.},~~~~\mbox{\rm and}~~~~V_{\Pr} = J_{\Pr}^{{\alpha^\star},\mub^\star}(\Cf^\star). 
    \end{align*}
    We will say that $\Cf^\star=(\phi^\star,\xi^\star,\aleph^\star)$ solves the problem $V_{\Pr}$.
\end{theorem}

\begin{remark}
    $(i)$ Compared to the results known in the literature of Principal--Agent, {\rm \Cref{thm:existence_contract}} provides existence results under general assumptions $($see {\rm\cite{mastroliaAtale2019,DayanikliOptimal2022}} $)$. The condition $``$$R_0$ is s.t. $\Xi$ is non--empty$"$ may seem restrictive but it is actually a reasonable condition. Indeed, if the reservation utility is unreasonably high, for instance $\overline{L}, L,g$ and $\overline{g}$ are all negative and $R_0 >0$, the problem has no interest since there will not be a solution. Therefore, this condition is there to avoid these kinds of situations and to make sure that the problem is worth solving. Given the maps $\overline{L}, L,g$ and $\overline{g}$, notice that sometimes it is not difficult to get $R_0$ s.t. $\Xi$ is non--empty. For example, if $(\overline{L},L,g,\overline{g})$ are lower bounded, it is straightforward to check that for any $R_0 \in \left(\inf (\overline{L} + L + g+\overline{g}), -\infty \right)$, $\Xi$ is non--empty.

    \medskip
    $(ii)$ When we see the statement of {\rm \Cref{thm:existence_contract}}, it might lead to think that the optimal control $\alpha^\star$ does not need to use the information of the contract $\aleph^\star$. This is not the case. Since the process $\mub^\star$ achieves an optimum when $\aleph^\star$ is fixed, it contains part of the information of the contract $\aleph^\star$. The shape of the optimal contract given in {\rm \Cref{thm:existence_contract}} provides some insight into the type of information needed to reach the optimum. It should be mentioned that the form of the coefficients plays an important role, in particular the separability of the variables $(m,e)$ on the one hand and $(x,\pi,a)$ on the other hand.

\end{remark}

\medskip
We now give some convergence results. In the next Theorem, we first provide the characterization of the limits of sequence of contracts and (approximate) MFG solutions. Second, we state the approximation of any MFG solution associated to a contract by a sequence of contracts and of (approximate) MFG solutions.
\begin{theorem}\label{thm:convergence_limit-contract}
    Let {\rm \Cref{assum:main1}} hold true.
    \begin{itemize}
        \item We consider a sequence of non--negative number $(\varepsilon_\ell)_{\ell \ge 1}$ s.t. $\lim_{\ell \to \infty} \varepsilon_{\ell}=0.$ Let  $(\alpha^\ell,\mub^\ell)_{\ell \ge 1}$ and $(\Cf^\ell)_{\ell \ge 1}=(\phi^\ell,\xi^\ell,\aleph^\ell)_{\ell \ge 1}$ be sequences such that  $(\alpha^\ell,\mub^\ell) \in \mbox{{\rm MFG}}[\Cf^\ell,\varepsilon_{\ell}]$ for each $\ell \ge 1$ then $\left(J^{\alpha^\ell,\mub^\ell}_{\Pr}(\Cf^\ell),J^{\Cf^\ell}_{A,\mub^\ell}(\alpha^\ell) \right)_{\ell \ge 1}$ is relatively compact and for any convergent sequence $(\ell_j)_{j\ge 1}$, under {\rm \Cref{assum:convexity}}, there exist $\Cf$ and $(\alpha,\mub)$ s.t. $(\alpha,\mub) \in \mbox{{\rm MFG}}[\Cf]$,
        \begin{align*}
            \lim_{j \to \infty}J^{\Cf^{\ell_j}}_{A,\mub^{\ell_j}}({\alpha^{\ell_j}})\le
            J^{\Cf}_{A,\mub}({\alpha})~~~~~\mbox{\rm and}~~~~\lim_{j \to \infty} J^{\alpha^{\ell_j},\mub^{\ell_j}}_{\Pr}(\Cf^{\ell_j}) \le
            J^{\alpha,\mub}_{\Pr}(\Cf).
        \end{align*}

        \item Let $\Cf=(\phi,\xi,\aleph)$ be a contract and $(\alpha,\mub) \in \mbox{{\rm MFG}}[\Cf]$. There exist a sequence of non--negative numbers $(\varepsilon_{\ell})_{\ell \ge 1}$ s.t. $\lim_{\ell \to \infty} \varepsilon_\ell=0,$ and a sequence $(\alpha^\ell,\mub^\ell,\phi^\ell, \xi^\ell,\aleph^\ell)_{\ell \ge 1}$ s.t. for each $\ell \ge 1$, $\Cf^\ell=(\phi^\ell,\xi^\ell,\aleph^\ell)$ is a contract, $(\alpha^\ell, \mub^\ell )\in \mbox{{\rm MFG}}_S[\Cf^\ell,\varepsilon_\ell]$ and 
        \begin{align*}
            \lim_{\ell \to \infty}J^{\Cf^{\ell}}_{A,\mub^\ell}(\alpha^{\ell})=
            J^{\Cf}_{A,\mub}({\alpha})~~~~~\mbox{\rm and}~~~~\lim_{\ell \to \infty} J^{\alpha^{\ell},\mub^\ell}_{\Pr}(\Cf^\ell)=
            J^{\alpha,\mub}_{\Pr}(\Cf).
        \end{align*}
    \end{itemize}
    
\end{theorem}

\begin{remark}
    To the best of our knowledge, {\rm \Cref{thm:convergence_limit-contract}} is the first of this type in the literature of contract theory. From the point of view of {\rm MFG} theory, this result can be seen as the characterization of {\rm MFG} solutions associated with varying parameters. This appears to be new. In the proofs of this result, we provide more general convergence results than those mentioned in the Theorem by using the notion of $``$relaxed$"$ controls. We have chosen this presentation for its ease of reading and because it gives the main information that we wish to share.
\end{remark}

By using {\rm \Cref{thm:convergence_limit-contract}}, we are able to give a convergence result on the problem of the Principal. In simple words, the problem in which the Principal considers contracts leading to approximate strong MFG solutions for the representative Agent is close to the problem in which the Principal offers contracts leading to MFG solutions for the representative Agent.
\begin{corollary} \label{cor:limit_strong_valuefunc}
    Let {\rm \Cref{assum:main1}} be true and $R_0$ s.t. $\Xi$ is non--empty. Then, for each $(\varepsilon, \varepsilon_0) \in (0,\infty) \x (0, \infty),$ the set $\Xi_S[\varepsilon,\varepsilon_0]$ is non--empty and we have
    \begin{align*}
        \lim_{\varepsilon \to 0, \varepsilon >0} \lim_{\varepsilon_0 \to 0, \varepsilon_0 >0} V_{\Pr}^S[\varepsilon,\varepsilon_0]=V_{\Pr}.
    \end{align*}
    Besides, for each $\delta>0$, there exist $\varepsilon, \varepsilon_0 >0$, $(\alpha,\mub)$ and $\Cf$ s.t. $\Cf \in \Xi_S[\varepsilon,\varepsilon_0]$, $(\alpha,\mub) \in \underline{\rm MFG}_S[\Cf,\varepsilon,\varepsilon_0]$ and $V_{\Pr} \le J^{\alpha,\mub}_{\Pr}(\Cf) + \delta$.
\end{corollary}

\medskip
The following Theorem starts by giving an understanding of the limits of sequence of contracts and approximate Nash equilibria. Next, it provides an approximation of any MFG solution associated to a contract by a sequence of approximate Nash equilibria associated to a sequence of contracts for the problem of Principal and $n$ Agents. 
\begin{theorem} \label{thm:conv:n-player}
    Let {\rm \Cref{assum:main1}} hold true.

    \begin{itemize}
        \item Let $(\varepsilon_n)_{n \ge 1}$ be a sequence of non--negative numbers s.t. $\Lim_{n \to \infty} \varepsilon_n=0$. We consider a sequence $(\alphab^n,\phi^n,\xi^n,\aleph^n)$ s.t. $\Cf^n=(\phi^n,\xi^n,\aleph^n)$ is a contract and $\alphab^n \in \mbox{\rm NE}[\Cf^n,\varepsilon_n]$ for each $ n \ge 1$. The sequence 
        $$
            \left( \frac{1}{n}\sum_{i=1}^nJ^{\Cf^n}_{n,i}(\alphab^n),~~ J_{n,\Pr}^{{\color{black}\alphab^n}}(\Cf^n) \right)_{n \ge 1}
        $$
        is relatively compact and for any convergent sub--sequence $(n_j)_{j \ge 1}$, under {\rm \Cref{assum:convexity}}, there exist $\Cf=(\phi,\xi,\aleph)$ and $(\alpha,\mub) \in \mbox{{\rm MFG}}[\Cf]$ verifying
        \begin{align*}
            \lim_{j \to \infty}\frac{1}{n_j}\sum_{i=1}^{n_j}J^{\Cf^{n_j}}_{n_j,i}(\alphab^{n_j}) \le
            J^{\Cf}_{A,\mub}({\alpha})~~~~~\mbox{\rm and}~~~~\lim_{j \to \infty} J_{n_j,\Pr}^{{\color{black}\alphab^{n_j}}}(\Cf^{n_j}) \le
            J_{\Pr}^{\alpha,\mub}(\Cf).
        \end{align*}

        \medskip
        \item Let $\Cf=(\phi,\xi,\aleph)$ be a contract and $(\alpha,\mub) \in \mbox{{\rm MFG}}[\Cf]$. There exist a sequence of non--negative numbers $(\varepsilon_{n})_{n \ge 1}$ s.t. $\lim_{n \to \infty} \varepsilon_n=0,$ and a sequence $(\alphab^n,\phi^n, \xi^n,\aleph^n)_{\ell \ge 1}$ s.t. for each $n \ge 1$, $\Cf^n=(\phi^n,\xi^n,\aleph^n)$ is a distributed contract, $\alphab^n \in {\rm NE}^{\rm dist}[\Cf^n,\varepsilon_n]$ 
        and we have
        \begin{align*}
            \lim_{n \to \infty}\frac{1}{n}\sum_{i=1}^{n}J^{\Cf^n}_{n,i}(\alphab^{n})=
            J^{\Cf}_{A,\mub}({\alpha})~~~~~\mbox{\rm and}~~~~\lim_{n \to \infty} J_{n,\Pr}^{{\color{black}\alphab^n}}(\Cf^n)=J_{\Pr}^{\alpha,\mub}(\Cf).
        \end{align*}
    \end{itemize}
\end{theorem}

\begin{remark}
    Similarly to {\rm \Cref{thm:convergence_limit-contract}}, {\rm \Cref{thm:conv:n-player}} seems to be the first of this kind in the literature of contract theory. From a {\rm MFG} theory perspective, it is worth mentioning that, since the map $(\xi^n,\aleph^n)$ depends on $(X^1,\dots,X^n)$ in a general way, This Theorem goes beyond the classical {\rm MFG} framework. See {\rm \citeauthor*{Lackermeanfield2022}\cite{Lackermeanfield2022}} and {\rm\citeauthor*{JacksonAppro2023}\cite{JacksonAppro2023}} for related analysis beyond the classical mean--field dependence framework. Some classical results of {\rm MFG} literature are also contained in this Theorem. Also as in {\rm \Cref{thm:convergence_limit-contract}}, the results presented could be more general using $``$relaxed$"$ controls, this presentation was chosen for ease of reading.
\end{remark}

In the same spirit as in \Cref{cor:limit_strong_valuefunc}, by the use of {\rm \Cref{thm:conv:n-player}}, the following result simply shows that the problem of the Principal with $n$ agents converges when $n$ goes to infinity to the problem of the Principal with the representative agent. 
\begin{corollary} \label{cor:conv_strong_limit_n-player}
    Let $R_0$ be s.t. $\Xi$ is non--empty and {\rm \Cref{assum:main1}} hold true. For each $(\varepsilon,\varepsilon_0) \in (0,\infty) \x (0,\infty)$, there exists $\overline{n} \ge 1$ s.t. for each $n \ge \overline{n}$,  $\Xi^{\rm dist}_n[\varepsilon,\varepsilon_{0}]$ is non--empty, and we have
    \begin{align*}
        \lim_{\varepsilon \to 0}\lim_{ \varepsilon_0 \to 0}\lim_{n \to \infty}V_{n,\Pr}[\varepsilon,\varepsilon_{0}]=\lim_{\varepsilon \to 0}\lim_{ \varepsilon_0 \to 0}\lim_{n \to \infty}V^{\rm dist}_{n,\Pr}[\varepsilon,\varepsilon_{0}]=V_{\Pr}.
    \end{align*}
    Also, for each $\delta>0$, there exist $n \ge 1$, $(\varepsilon_{n}, \varepsilon_{0,n}) \in (0,\infty) \x (0,\infty)$, $\alphab^n$ and $\Cf^n$ s.t. $\Cf^n$ is a distributed contract, $\alphab^n \in \underline{\rm NE}^{\rm dist}[\Cf^n,\varepsilon_n,\varepsilon_{0,n}]$ and $V_{n,\Pr}[\varepsilon_n,\varepsilon_{0,n}] \le J_{n,\Pr}^{{\color{black}\alphab^n}}(\Cf^n) + \delta$.
\end{corollary}

\begin{remark}
    It is worth emphasizing that the previous result point out the fact that when the number of Agents $n$ is large enough, considering distributed contracts is equivalent to considering contracts of a general type for the Principal when he faces $n$ Agents in competition.
\end{remark}

\section{Proof of the main results} \label{sec:proofs}

This whole section is dedicated to the proofs of our main results. We start in \Cref{sec:reformulation}  by reformulating our problem in an adequate setup for the proofs. Then, in \Cref{sec:Nash_to_MFG}, we give the proofs of the characterization of the limits of sequence of approximate Nash equilibria and approximate MFG solutions. In \Cref{sec:MFG_to_Nash}, we show how to construct a sequence of approximate Nash equilibria from a MFG solution. Next, we show why the set of MFG solutions considered in this paper is non--empty in \Cref{sec:existence_MFG}. Finally, \Cref{sec:proofs_main} presents the actual proofs of the main Theorem/Corollary stated in the paper.

\subsection{Reformulation of the MFG} \label{sec:reformulation}

  Denote by $\M := \M \left(\Pc_A \x \Er \right)$ the collection of all finite (Borel) measures $q(\mathrm{d}t, \mathrm{d}m, \mathrm{d}e)$ on $\Pc_A \x \Er \x [0,T]$, 
	whose marginal distribution on $[0,T]$ is the Lebesgue measure $\mathrm{d}s$ 
	i.e., $q(\mathrm{d}m,\mathrm{d}e,\mathrm{d}s)=q(s,\mathrm{d}m,\mathrm{d}e)\mathrm{d}s$ for a measurable family $(q(s, \mathrm{d}m,\mathrm{d}e))_{s \in [0,T]}$ of Borel probability measures on $\Pc_A \x \Er.$
	For $\Lambda$ the canonical element on $\M$ and the canonical filtration $\F^\Lambda = (\Fc^\Lambda_t)_{0 \le t \le T}$ on $\M$ defined by
	$
		\Fc^\Lambda_t := \sigma \big\{ \Lambda(C \x [0,s]) ~: \forall s \le t,\;C \in \Bc(\Pc_A \x \Er) \big\},
	$
	we have the disintegration property: $q(\mathrm{d}m, \mathrm{d}e, \mathrm{d}t) = q(t, \mathrm{d}m, \mathrm{d}e) \mathrm{d}t$, with a version of the disintegration
	such that $(t, q) \mapsto q(t, \mathrm{d}m, \mathrm{d}e)$ is $\F^\Lambda$--predictable.
	

    \vspace{4mm}
    \noindent
    We introduce the canonical space $\Omb:=\Cc \x \Cc \x \Cc \x \M \x \Cc_{\Wc}$ and its associated canonical variables $(X',W,B,R,\mu).$ Then, the canonical filtration $\Fb = (\Fcb_t)_{t \in [0,T]}$ is defined by: for all $t \in [0,T]$,  
    $$ 
        \Fcb_t
        :=
        \sigma 
        \big\{X'_{t \wedge \cdot},W_{t \wedge \cdot},B_{t \wedge \cdot},  R_{t \wedge \cdot}, \mu_{t \wedge \cdot} \big\} 
    $$
    with $R_{t \wedge \cdot}$ denotes the restriction of $R$ on $\Pc_A \x \Er \x [0,t]$ (see definition \ref{eq:lambda}). Notice that we can choose a version of the disintegration  $R(\mathrm{d}m,\mathrm{d}e,\mathrm{d}t)=R_t(\mathrm{d}m,\mathrm{d}e)\mathrm{d}t$ with $(R_t)_{t \in [0,T]}$ is a $\Pc_A \x \Er$--valued $\Fb$--predictable process.
    Let us also introduce the \underline{common noise} filtration $\Gb:=(\Gcb_t)_{t \in [0,T]}$ by: for all $t \in [0,T]$,
    $$ 
        \Gcb_t
        :=
        \sigma \big\{B_{t \wedge \cdot},  R_{t \wedge \cdot}, \mu_{t \wedge \cdot} \big\}.
    $$

\begin{definition}[control rule] \label{def:RelaxedCcontrol}
    We say that $\mathrm{P} \in \Pc(\Omb)$ belongs to $\Pcb$ i.e. a control rule if:
    \begin{enumerate}
        \item[$(i)$] $(W,B)$ is an $\R^d \x \R^d$--valued $\Pr$--Brownian motion and the distribution of $X'_0$ under $\Pr$ is $\nu$

        \item[$(ii)$] There exists a progressively Borel measurable map $\Lambda':[0,T] \x \R^d \x \M \to \Pc(A)$ such that the tuple $(X',W,B,R,\mu)$ satisfies: $\mathrm{d}\Pr \otimes \mathrm{d}t$--a.e. $(\om,t)$, $R_t\left(\left\{(m,e):\;\mu_t=m(\mathrm{d}x,A) \right\} \right)=1$ and
        \begin{align} \label{eq:process_con trolled}
            \mathrm{d}X'_t
            =
            \int_{\Pc_A \x \Er} \overline{b}\left(t,m,e \right)  R_t(\mathrm{d}m,\mathrm{d}e)
            +
            \int_{A} b\left(t,X_t',\mu,a \right) \Lambda'\left(t,X'_t,R\right)(\mathrm{d}a) \;\mathrm{d}t 
            +
            \sigma(t,X_t')\;\mathrm{d}W_t
            +
            \sigma_0\; \mathrm{d}B_t\;\;\Pr\mbox{--a.e.}
        \end{align}

        \item[$(iii)$] 
        In addition, 
        the variables $(B,R,\mu)$, W and $X'_0$ are $\Pr$--independent.
    \end{enumerate}

\end{definition}

\medskip
For Polish spaces $\Theta$ and $\Theta'$, we will say that $\Kr$ belongs to $\Kc(\Theta,\Theta')$ if $\Kr$ is a Borel map from $\Theta$ to $\Pc(\Theta')$. If for $\Kr \in \Kc(\Theta,\Theta')$, we have $\Kr(\theta)=\delta_{\kappa(\theta)}(\mathrm{d}i)$ for some Borel map $\kappa:\Theta \to \Theta'$, we will say $\Kr \in \Kc_{\rm 0}(\Theta,\Theta')$. For any $(\gamma',\pi',\rr,\pi, \Phi , \mathfrak{Z}) \in  \M(\R^d \x A) \x \Cc_{\Wc} \x \M \x \Cc_{\Wc} \x \Kc(\R,\Ir) \x \Kc(\M,\Er)$, one defines
\begin{align*}
    &{\rm J}^{\Phi , \mathfrak{Z}}_A(\gamma',\pi',\rr,\pi)
    \\
        &:=
        \int_0^T  \langle \overline{L}\left(t,\cdot,\cdot \right),\rr(t) \rangle
        +
        \langle L\left(t, \cdot,\pi,\cdot \right), \gamma'(t) \rangle\;\;\mathrm{d}t +
         \int_{\R \x \Ir} \overline{g}(x,i) \Phi(x)(\mathrm{d}i)\pi'(T)(\mathrm{d}x)
         + \int_{\Er}g(\pi,e)\mathfrak{Z}(\rr)(\mathrm{d}e).
\end{align*}
For any $\Pr \in \Pc(\Omb)$, we set, $\Pr$--a.e., $\mu'_t:=\Lc(X_t'|\Gcb_t)\;\mbox{for all}\;t \in [0,T]$ and when $\Pr \in \Pcb$,   
\begin{align*}
    \Gamma'
    :=
    \E^{\Pr} \left[ \delta_{X_t'}(\mathrm{d}x) \Lambda'\left(t,X'_t,R\right)(\mathrm{d}a) | \Gcb_t\right] \mathrm{d}t.
\end{align*}

\begin{remark} \label{rm:proba_equivalent}
        $(i)$ To make an analogy with the definition given in {\rm \Cref{section:limitproblem}}, the canonical variables $(X',W,B,R,\mu)$ will play the $``$role$"$ of the variables  $(X^\beta,W,B,R,\mu)$ on $(\Om,\H,\P)$ for some $\beta \in \Ac$.

\medskip
    $(ii)$ Notice that, for any $\Pr \in \Pcb$, since $\mathrm{d}\Pr \otimes \mathrm{d}t$--a.e. $(\om,t)$, $R_t\left(\left\{(m,e):\;\mu_t=m(\mathrm{d}x,A) \right\} \right)=1$, we can see that $\mu_t$ is a Borel map of $R_t$.

\medskip    
        $(iii)$ For $\Pr \in \Pc(\Omb)$, we will \underline{abuse the terminology} and say that $\Pr \in \Pcb$ if we can find $\Pr' \in \Pcb$ s.t. $\Pr \circ \left ( \mu', X_0',W,B,R, \mu \right)^{-1}=\Pr' \circ \left ( \mu', X_0',W,B,R, \mu \right)^{-1}$. In other words, our variables of interest are $(\mu', X_0',W,B,R, \mu)$. The process $X'$ is here for simplifying the presentation.


\end{remark}

For any $\varphi \in C_b(\R^d)$ and $(t,x,\pi,a)$, in order to accommodate the notation, we write $(\varphi b)(t,x,\pi,a)$ instead of $\varphi(x) b(t,x,\pi,a)$.

\begin{definition}{\rm (MFG solution)}
     For each $\varepsilon \ge 0$, we say that $\Pr^\star \in \Pcb$ is an $\varepsilon$--{\rm MFG} solution if: $\mathrm{d}\Pr^\star \otimes \mathrm{d}t$--a.e., $\mu_t=\mu_t'=\Lc^{\Pr^\star}(X'_t|\Gcb_T)$, for any $\varphi \in C_b(\R^d)$, 
        \begin{align} \label{eq:consistency}
            \E^{\Pr^\star}\left[  \int_A\left(\varphi b,\;\;L \right)\left(t,X'_t,\mu,a \right) \Lambda'(t,X'_t,R)(\mathrm{d}a) |\Gcb_T \right]
            =
            \int_{\Pc_A} \int_{\R^d \x A} \left(\varphi b,\;\;L\right)\left(t,x, \mu,a \right) m(\mathrm{d}x,\mathrm{d}a) R_t(\mathrm{d}m,\Er),
        \end{align}
        and there exists $(\Phi,\mathfrak{Z}) \in \Kc(\R,\Ir) \x \Kc(\M,\Er)$ s.t. 
        for every $\mathrm{P} \in \Pcb$ verifying $\Lc^{\mathrm{P}^{\star}}\big(B,R\big)=\Lc^{\mathrm{P}}\big(B,R\big)$, 
        we have
        \begin{align} \label{eq:optimality-relaxed}
            \E^{\mathrm{P}^{\star}}\big[{\rm J}^{\Phi,\Zf}_A\left(\Gamma',\mu',R,\mu \right) \big] \ge \E^{\mathrm{P}} \big[{\rm J}^{\Phi,\Zf}_A\left(\Gamma',\mu',R,\mu \right) \big]-\varepsilon.
        \end{align}
\end{definition}
We will say $\Pr^\star$ is an $\varepsilon$--MFG solution associated to $(\Phi,\Zf)$.
We denote by $\Pcb_{\rm \varepsilon\mbox{-}mfg}$ the set of all MFG solutions. When $\varepsilon=0$, we will simply write $\Pcb_{\rm mfg}$.

\medskip
We now introduce $\Pcb_{0,\;\rm \varepsilon\mbox{-}mfg}$ by
\begin{align*}
    &\Pcb_{0,\;\rm \varepsilon\mbox{-}mfg}
    \\
    &:=\Big\{
        \Pr \in \Pcb_{\rm \varepsilon\mbox{-}mfg},\;\exists\;(\overline{m}_t,\overline{e}_t)_{t \in [0,T]}\;\mbox{and}\;\alpha\;\mbox{s.t.}\;\P\mbox{--a.e.}\;R=\delta_{(\overline{m}_t,\overline{e}_t)}(\mathrm{d}m,\mathrm{d}e)\mathrm{d}t,
        \\        &~~~~~~~~~~~~~~~~~~~~~~~~~~~~~~~~~~~~~~~~~~~~~~~~\;\Lambda'(t,X_t',R)(\mathrm{d}a)\mathrm{d}t=\delta_{\alpha(t,X_t',\mu,\overline{m}_t)}(\mathrm{d}a)\mathrm{d}t\;\mbox{and}\;(\Phi,\mathfrak{Z}) \in \Kc_{\rm 0}(\R,\Ir) \x \Kc_{\rm 0}(\M,\Er)
    \Big\}.
\end{align*}

\begin{remark} \label{rm:measurable_B}
    Let us observe that: by taking the conditional expectation w.r.t. $\Gcb_T$ in \eqref{eq:process_con trolled}, for $\Pr \in \Pcb_{\rm \varepsilon\mbox{-}mfg}$, $\Pr$--a.e., for all $t \in [0,T]$,
    \begin{align*}
        \sigma_0 B_t=\int_{\R^d} x \mu_t(\mathrm{d}x) - \int_{\R^d} x \nu(\mathrm{d}x) - \int_0^t\;\int_{\Pc_A \x \Er} \left[\overline{b}\left(r,m,e \right)  
            +
            \int_{\R^d \x A} b\left(r,x,\mu,a \right) m(\mathrm{d}x,\mathrm{d}a) \right]\;R_r(\mathrm{d}m,\mathrm{d}e)\mathrm{d}r.
    \end{align*}
    Since $\sigma_0$ is invertible, similarly to {\rm \cite[Lemma 3.4.]{closed-loop-MFG_MDF}},  there exists a progressively Borel measurable map $[0,T] \x \M \x C([0,T];\Pc_p(\R^d)) \ni (t,\rr,\pi)\mapsto \varphi(t,\rr,\pi) \in \R^d$ continuous in all variables s.t. for any $\Pr \in \Pcb_{\rm \varepsilon\mbox{-}mfg}$, one has, for all $t \in [0,T]$, $B_t=\varphi\left(t,R,\mu \right)$,  $\Pr$--a.e. Furthermore, since $\mu_t$ is a Borel map of $R_t$, the Brownian motion $B$ can be seen as a progressively Borel map of $R$. 
\end{remark}

For each $(\rr,\pi)$, we introduce
\begin{align*}
    \Jr_{\Pr}^{\rr,\pi} \left(\Phi,\Zf \right)
    :=
    \int_{\Er}U \left(\int_{\R^d} \Upsilon(x) \;\pi(T)(\mathrm{d}x) - \int_{\Ir \x \R^d} g_{\Pr} \left(x,i \right) \Phi(x)(\mathrm{d}i)\pi(T)(\mathrm{d}x) - \overline{g}_{\Pr} \left(\pi,e \right) - \int_0^T \langle \overline{L}_{\Pr}(t,\cdot,\cdot), \rr(t) \rangle \; \mathrm{d}t \right)\Zf(\rr)(\mathrm{d}e).
\end{align*}

Simply put, the next Proposition shows that, when {\rm\Cref{assum:convexity}} is satisfied, we can use the set $\Pcb_{0,\;\rm \varepsilon\mbox{-}mfg}$ instead of $\Pcb_{\rm \varepsilon\mbox{-}mfg}$.
\begin{proposition} \label{prop:relaxed_to_weak}
    Under {\rm\Cref{assum:convexity}}, for any $\Pr^\star \in \Pcb_{\rm \varepsilon\mbox{-}mfg}$ associated to $(\Phi,\Zf)$, there exists  $\widetilde{\Pr}^\star \in \Pcb_{0,\;\rm \varepsilon\mbox{-}mfg}$ associated to $(\widetilde{\Phi},\widetilde{\Zf})$ s.t.
\begin{align*}                          \E^{{\Pr}^\star}\big[{\rm J}^{\Phi,\Zf}_A(\Gamma',\mu',R,\mu) \big] \le \E^{\widetilde{\Pr}^\star} \big[{\rm J}^{\tilde{\Phi},\tilde{\Zf}}_A(\Gamma',\mu',R,\mu) \big]\;\;\mbox{and}\;\;\E^{{\Pr}^\star}\big[{\rm J}^{R,\mu}_{\Pr}(\Phi,\Zf) \big] \le \E^{\widetilde{\Pr}^\star} \big[{\rm J}^{R,\mu}_{\Pr}(\widetilde{\Phi},\widetilde{\Zf}) \big].
\end{align*}
\end{proposition}

\begin{proof}
    Let $\mathrm{P}^\star \in \Pcb_{\rm \varepsilon\mbox{-}mfg}$. We use the same ideas of \citeauthor*{Filippov1962}\cite{Filippov1962} and \citeauthor*{Roxin1962}\cite{Roxin1962}. Under \Cref{assum:convexity}, by applying \Cref{prop:projection} and using the fact that $\mu_t$ is a Borel map of $R_t$, first, we can find progressively Borel measurable maps $\overline{m}:[0,T] \x \M \to \Pc_A$ and $\overline{e}:[0,T] \x \M \to \Er$ s.t. $\mathrm{d}\Pr^\star \otimes \mathrm{d}t$--a.e. $\overline{m}(t,R)(\mathrm{d}x,A)=\mu_t$,
    \begin{align*}
        \int_{\Pc_A \x \Er} \overline{L}(t,m,e)R_t(\mathrm{d}m,\mathrm{d}e)
        =
        \overline{L}\left(t,\overline{m}(t,R),\overline{e}\left(t,R \right) \right),\;\int_{\Pc_A \x \Er} \langle L\;(t,\cdot,\mu,\cdot),m \rangle R_t(\mathrm{d}m,\mathrm{d}e)
        \le
        \langle L\;(t,\cdot,\mu,\cdot),\overline{m}(t,R) \rangle,
    \end{align*}
    \begin{align*}
        -\int_{\Pc_A \x \Er} \overline{L}_{\Pr}(t,m,e)R_t(\mathrm{d}m,\mathrm{d}e)
        \le
        -\overline{L}_{\Pr}\left(t,\overline{m}(t,R),\overline{e}\left(t,R \right) \right)
    \end{align*}
    and for $\varphi \in C_c(\R)$
    \begin{align} \label{eq:equality_convex_1}
        \int_{\Pc_A \x \Er} \left(\overline{b}(t,m,e),\;\langle \varphi b\;(t,\cdot,\mu,\cdot),m \rangle \right)R_t(\mathrm{d}m,\mathrm{d}e)
        =
        \left(\overline{b}\left(t,\overline{m}(t,R),\overline{e}\left(t,R\right) \right),\; \langle \varphi b\;(t,\cdot,\mu,\cdot),\overline{m}(t,R) \rangle \right).
    \end{align}
    Second, let $\overline{R}:=\delta_{\left( \overline{m}(t,R), \overline{e}(t,R) \right)}(\mathrm{d}m,\mathrm{d}e)\mathrm{d}t$, we get Borel maps $\R \x \M \ni (x,\rr) \mapsto (\phi(x),z(\rr)) \in \Ir \x \Er  $ s.t. 
    $$
        \E^{\Pr} \left[\int_{\Er} \overline{g}(\mu,e))\Zf(R)(\mathrm{d}e) | \overline{R} \right] = \overline{g}(\mu,z\left(\overline{R} \right))\;\;\mbox{and}\;\;\int_{\Ir} g(x,i))\Phi(x)(\mathrm{d}i) = g(x,\phi\left(x \right)),
    $$
    $$
        \E^{\Pr} \left[-\int_{\Er} \overline{g}_{\Pr}(\mu,e))\Zf(R)(\mathrm{d}e) | \overline{R} \right] \le -\overline{g}_{\Pr}(\mu,z\left(\overline{R} \right))\;\;\mbox{and}\;\;-\int_{\Ir} g_{\Pr}(x,i))\Phi(x)(\mathrm{d}i) \le -g_{\Pr}(x,\phi\left(x \right)).
    $$
    Notice that, since the utility map $U:\R^d \to \R$ is non--decreasing and concave, we have, $\Pr$--a.e.
    \begin{align*}
        \E^{\Pr} \left[\Jr_{\Pr}^{R,\mu} \left(\Phi,\Zf \right) | \overline{R} \right] \le U \left(\int_{\R^d} \Upsilon(x) \;\mu_T(\mathrm{d}x) - \int_{\R^d} g_{\Pr} \left(x,\phi(x) \right) \mu_T(\mathrm{d}x) - \overline{g}_{\Pr} \left(\pi,z(\overline{R}) \right) - \int_0^T \overline{L}_{\Pr} \left(t,\overline{m}(t,R), \overline{e}(t,R) \right) \; \mathrm{d}t \right).
    \end{align*}
    
    Lastly, by the same \Cref{prop:projection}, there exists a Borel progressively measurable $\alpha: [0,T] \x \R^d \x \Cc_{\Wc} \x \Pc_A \to A$ s.t. $\mathrm{d}\Pr^\star \otimes \mathrm{d}t$--a.e. for $\varphi \in C_c(\R)$
    \begin{align} \label{eq:equality_convex_2}
        &\E^{\Pr^\star} \left[\int_{A} \varphi(X_t')b\left(t,X'_t,\mu,a \right) \Lambda'\left(t,X_t',R\right)(\mathrm{d}a) |\Gcb_t\right]
        =
        \int_{\Pc_A}\int_{\R^d \x A} \varphi(x) b(t,x,\mu,a)m(\mathrm{d}x,\mathrm{d}a) R_t(\mathrm{d}m,\Er) \nonumber
        \\
        &=
        \langle \varphi b\;(t,\cdot,\mu,\cdot),\overline{m}(t,R) \rangle
        =\int_{\R^d}\varphi(x)b\left(t,x,\mu,\alpha\left(t,x,\mu,\overline{m}(t,R)\right) \right) \mu_t(\mathrm{d}x)
    \end{align}
    and
    \begin{align*}
        \left\langle L\left(t,X'_t,\mu,\cdot \right),\; \Lambda'\left(t,X_t',R\right) \right\rangle
        \le L\left(t,X'_t,\mu,\alpha\left(t,X'_t,\mu,\overline{m}(t,R)\right) \right).
    \end{align*}
    We define $X^\alpha$ the weak solution of: $\Pr^\star$--a.e., $X^\alpha_0=X'_0$ and 
    \begin{align*}
        \mathrm{d}X^\alpha_t
            =
            \overline{b}\left(t,\overline{m}(t,R),\overline{e}\left(t,R \right)\right)\mathrm{d}t
            +
            b\left(t,X^\alpha_t,\mu,\alpha(t,X^\alpha_t,\mu,\overline{m}(t,R))\right)\mathrm{d}t + \sigma(t,X^\alpha_t)\mathrm{d}W_t + \sigma_0\mathrm{d}B_t.
    \end{align*}
    We set $\mut_t:=\Lc^{\Pr^\star}(X^\alpha_t|R,B)$, $\mub_t:=\overline{m}(t,R)$, $\Pr^\star$--a.e. for all $t \in [0,T]$, $R^\alpha:=\delta_{\left(\mub_t,\;\overline{e}_t \right)}(\mathrm{d}m,\mathrm{d}e)\mathrm{d}t$ and $\overline{e}_t:=\overline{e}(t,R)$. Thanks to the equalities in \eqref{eq:equality_convex_1} and \eqref{eq:equality_convex_2}, we can check that $\mut=\mu$ $\Pr^\star$--a.e. Let us set $\widetilde{\Phi}(x)(\mathrm{d}i):=\delta_{\phi(x)}(\mathrm{d}i)$,  $\widetilde{\Zf}(\rr)(\mathrm{d}e):=\delta_{z(\rr)}(\mathrm{d}e)$, and define the probability $\widetilde{\Pr}^\star$ defined by 
    $$
        \widetilde{\Pr}^\star:=\Pr^\star \circ \left( X^\alpha, W,B,R^\alpha,\mut \right)^{-1}.
    $$
    By combining the previous results, it is straightforward that $\E^{{\Pr}^\star}\big[{\rm J}^{\Phi,\Zf}_A(\Gamma',\mu',R,\mu) \big] \le \E^{\widetilde{\Pr}^\star} \big[{\rm J}^{\tilde{\Phi},\tilde{\Zf}}_A(\Gamma',\mu',R,\mu) \big]$ and $\E^{{\Pr}^\star}\big[{\rm J}^{R,\mu}_{\Pr}(\Phi,\Zf) \big] \le \E^{\widetilde{\Pr}^\star} \big[{\rm J}^{R,\mu}_{\Pr}(\widetilde{\Phi},\widetilde{\Zf}) \big]$. Also, for $\widetilde{\Pr} \in \Pcb$ s.t. $\Lc^{\widetilde{\Pr}}(B,R)=\Lc^{\widetilde{\Pr}^\star}(B,R)$, we can find $\Pr \in \Pcb$ verifying $\Lc^{{\Pr}}(B,R)=\Lc^{{\Pr}^\star}(B,R)$ s.t.
    $$
        \E^{\widetilde{\Pr}} \big[{\rm J}^{\tilde{\Phi},\tilde{\Zf}}_A(\Gamma',\mu',R,\mu) \big]
        =
        \E^{{\Pr}}\big[{\rm J}^{\Phi,\Zf}_A(\Gamma',\mu',R,\mu) \big]
        \le \E^{{\Pr}^\star}\big[{\rm J}^{\Phi,\Zf}_A(\Gamma',\mu',R,\mu) \big]
        \le \E^{\widetilde{\Pr}^\star}\big[{\rm J}^{\tilde\Phi,\tilde\Zf}_A(\Gamma',\mu',R,\mu) \big].
    $$
    Therefore, we obtain that $\widetilde{\Pr}^\star \in \Pcb_{0,\;\rm \varepsilon\mbox{-}mfg}$.
    
\end{proof}

\medskip
We recall that the progressively Borel map $\Lambda':[0,T] \x \R^d \x \M \to \Pc(A)$ given in \Cref{def:RelaxedCcontrol} has to be seen as a control. Under {\rm \Cref{assum:main1}},  the next result shows that, by an approximation result, we can replace $\Lambda'(t,X'_t,R)(\mathrm{d}a)$ by $\delta_{\lambda'(t,X'_t,R)}(\mathrm{d}a)$ for some progressively Borel map $\lambda':[0,T] \x \R^d \x \M \to A$.  
\begin{lemma} \label{lemma:conv_adm-control}
    Let {\rm \Cref{assum:main1}} hold true. Let $\Pr \in \Pcb$ s.t. $\mathrm{d}t \otimes \mathrm{d}\Pr$--a.e.  $B_t=\varphi(t,R)$ for some progressively Borel map $\varphi$. Then, there exist a sequence of Borel maps $(\lambda'^\ell)_{\ell \ge 1}$, a sequence of probabilities $(\Pr^\ell)_{\ell \ge 1} \subset \Pcb$ s.t. for each $\ell \ge 1$, $[0,T] \x \R^d \x \M \ni (t,x,\rr)\mapsto \lambda'^\ell(t,x,\rr) \in A$ is Lipschitz in $(x,\rr)$ uniformly in $t$, 
    $$
        \Pr^\ell \left( \Lambda'(t,X'_t,R)(\mathrm{d}a)\mathrm{d}t=\delta_{\lambda'^\ell(t,X_t',R)}(\mathrm{d}a)\mathrm{d}t \right)=1
    $$
    and 
        \begin{align*}
            \lim_{\ell \to \infty} \Pr^\ell \left( \mu', \Kr(y)(\mathrm{d}i)\mu_T'(\mathrm{d}y), \mu,R,B \right)^{-1}
            =
            \Pr \left( \mu', \Kr(y)(\mathrm{d}i)\mu_T'(\mathrm{d}y), \mu,R,B \right)^{-1}\;\;\mbox{in }\Wc_p\;\;\mbox{for any }\Kr \in \Kc(\R^d,\Ir).
        \end{align*}
\end{lemma}

\begin{proof}
    For $\Pr \in \Pcb$. Recall that, $\Pr$--a.e.
    \begin{align*}
        \mathrm{d}X'_t
            =
            \int_{\Pc_A \x \Er} \overline{b}(t,m,e) R_t(\mathrm{d}m,\mathrm{d}e) + \int_{A} b\left(t,X_t',\mu,a \right) \Lambda'\left(t,X_t', R\right)(\mathrm{d}a)\;\;\mathrm{d}t 
            +
            \sigma(t,X_t')\;\mathrm{d}W_t
            +
            \sigma_0\; \mathrm{d}B_t.
    \end{align*}
    By easy extension of \cite[Proposition A.5.]{closed-loop-MFG_MDF} (see also \cite[Lemma 3.5.]{closed-loop-MFG_MDF}), there exists a sequence $(\lambda'^\ell)_{\ell \ge 1} \subset \Ac$ s.t. for each $\ell \ge 1$, $[0,T] \x \R^d \x \M \ni (t,x,\rr)\mapsto \lambda'^\ell(t,x,\rr) \in A$ is Lipschitz in $(x,\rr)$ uniformly in $t$ and if $X'^\ell$ is the solution of 
    \begin{align*}
        \mathrm{d}X'^\ell_t
            =
            \int_{\Pc_A \x \Er} \overline{b}(t,m,e) R_t(\mathrm{d}m,\mathrm{d}e) + b\left(t,X'^\ell_t,\mu,\lambda'^\ell\left(t,X'^\ell_t, R\right) \right)\;\;\mathrm{d}t 
            +
            \sigma(t,X'^\ell_t)\;\mathrm{d}W_t
            +
            \sigma_0\; \mathrm{d}B_t
    \end{align*}
    then $\lim_{\ell \to \infty} \left( \mu'^\ell,\; \delta_{\lambda'^\ell(t,x,R)}(\mathrm{d}a)\mu'^\ell_t(\mathrm{d}x)\mathrm{d}t \right)=\left(\mu',\;\Gamma' \right)$ $\Pr$--a.e. where $\mu'^\ell_t=\Lc^{\Pr}(X'^\ell_t|R)$ for all $t \in [0,T]$. By using the same argument as in \Cref{prop:charac-convergence}, we are able to get that $\lim_{\ell \to \infty} \Pr \left( \mu'^\ell, \Kr(y)(\mathrm{d}i)\mu'^\ell_T(\mathrm{d}y), \mu,R,B \right)^{-1}
            =
            \Pr \left( \mu', \Kr(y)(\mathrm{d}i)\mu'_T(\mathrm{d}y), \mu,R,B \right)^{-1}$ in $\Wc_p$ for each $\Kr \in \Kc(\R^d,\Ir)$.

\end{proof}

\medskip
We are now ready to show the equivalence result between the canonical formulation used here and the formulation of \Cref{section:limitproblem}.   Let us introduce another set of MFG equilibria. For any contract $\Cf=(\phi,\xi,\aleph)$ and $\varepsilon \ge 0$, we define
    \begin{align*}
        \Sc^{\star}(\Cf,\varepsilon)
        :=
        \left\{
            \P \circ (\mu,R)^{-1}:\;\mbox{where }(\alpha,\mub) \in \mbox{{\rm MFG}}[\Cf,\varepsilon]\;\;\mbox{and}\;\;R:=\delta_{(\mub_t,\aleph_t)}(\mathrm{d}m,\mathrm{d}e)\mathrm{d}t
        \right\}.
    \end{align*}

\begin{proposition}  \label{prop:eq--measure--valuedMFG_control}
        Let $\varepsilon \ge 0$. For any contract $\Cf=(\phi,\xi,\aleph)$ and probability measure $\mathrm{Q}^{\star} \in \Sc^{\star}(\Cf,\varepsilon)$ then  $\mathrm{Q}^{\star}=\Pr^\star \circ(\mu,R)^{-1}$ where $\mathrm{P}^\star \in \Pcb_{\rm \varepsilon \mbox{-} mfg}$. Also, for any $\mathrm{P}^\star \in \Pcb_{0,\;\rm \varepsilon \mbox{-} mfg}$ there exist a contract $\Cf=(\phi,\xi,\aleph)$ and $\mathrm{Q}^{\star} \in \Sc^{\star}(\Cf,\varepsilon)$ s.t.  $\mathrm{Q}^{\star}=\Pr^\star \circ(\mu,R)^{-1}$. 
\end{proposition}

\begin{proof}
    We only prove for $\varepsilon=0,$ the case for any $\varepsilon \ge 0$ follows immediately. Let $\Cf=(\phi,\xi,\aleph)$ be a contract and $\Qr^\star \in \Sc^\star(\Cf)$. Let $(\alpha,\mub) \in \mbox{{\rm MFG}}[\Cf]$ be s.t. $\Qr^\star=\P \circ (\mu,R)^{-1}$. We introduce
    \begin{align*}
        \Pr^\star
        :=
        \P \circ \left(X^\alpha,W,B, R, \mu \right)^{-1}\;\;\mbox{where}\;\;R:=\delta_{\left(\mub_t,\aleph_t\right)}(\mathrm{d}m,\mathrm{d}e)\mathrm{d}t,\;\Phi(\cdot):=\delta_{\phi(\cdot)}(\mathrm{d}i)\;\mbox{and}\;\Zf(\cdot):=\delta_{\xi(\cdot)}(\mathrm{d}e).
    \end{align*}
    We can easily check that $\Pr^\star \in \Pcb$. By the definition of $(B,R,\mu)$, it is easy to see that \eqref{eq:consistency} is satisfied for $\Pr^\star$. Now, let $\mathrm{P} \in \Pcb$ be such that $\Lc^{\mathrm{P}^{\star}}\big(B,R\big)=\Lc^{\mathrm{P}}\big( B, R\big)$. 
    Notice that, by uniqueness in distribution of \eqref{eq:process_con trolled}, we verify that $\Lc^{\Pr}(X',W,B,R,\mu)=\Lc^{\P}(Z,W,B,R,\mu)$ where $\P$--a.e.
    \begin{align*}
        \mathrm{d}Z_t
            =
            \int_{\Pc_A \x \Er} \overline{b}(t,m,e) R_t(\mathrm{d}m,\mathrm{d}e) + \int_{A} b\left(t,Z_t,\mu,u \right) \Lambda'\left(t,Z_t, R\right)(\mathrm{d}a)\;\;\mathrm{d}t 
            +
            \sigma(t,Z_t)\;\mathrm{d}W_t
            +
            \sigma_0\; \mathrm{d}B_t.
    \end{align*}
    We set $\P$--a.e. $\Theta:=\E^{\P} \left[\delta_{Z_t} (\mathrm{d}z)\Lambda'\left(t,Z_t,R\right)(\mathrm{d}a) | \Gc_t\right] \mathrm{d}t\;\mbox{and}\;\;\mut_t:=\Lc^\P(Z_t|\Gc_t)$ for all $t \in [0,T]$. As in \Cref{rm:measurable_B}, $B$ is a progressively Borel map of $R$ under $(\Om,\H,\P)$, under \Cref{assum:main1} by  \Cref{lemma:conv_adm-control}, there exists a sequence $(\lambda'^\ell)_{\ell \ge 1} \subset \Ac$ s.t. if $Z^\ell$ is the solution of 
    \begin{align*}
        \mathrm{d}Z^\ell_t
            =
            \int_{\Pc_A \x \Er} \overline{b}(t,m,e) R_t(\mathrm{d}m,\mathrm{d}e) + b\left(t,Z^\ell_t,\mu,\lambda'^\ell\left(t,Z^\ell_t, R\right) \right)\;\;\mathrm{d}t 
            +
            \sigma(t,Z^\ell_t)\;\mathrm{d}W_t
            +
            \sigma_0\; \mathrm{d}B_t
    \end{align*}
    then $\lim_{\ell \to \infty}J_{A,\mub}^{\Cf}(\lambda'^\ell)
        =
        \E^{\P}\big[{\rm J}^{\Phi,\Zf}_{A}(\Theta,\mut,R,\mu) \big]=\E^{\mathrm{P}}\big[{\rm J}^{\Phi,\Zf}_A(\Gamma',\mu',R,\mu) \big].$ Consequently, as we know that $(\alpha,\mub) \in \mbox{{\rm MFG}}[\Cf]$, we find
    \begin{align*}
        \E^{\mathrm{P}^{\star}}\big[{\rm J}^{\Phi,\Zf}_A(\Gamma',\mu',R,\mu) \big]
        =
        J_{A,\mub}^{\Cf}(\alpha) \ge \lim_{\ell \to \infty}J_{A,\mub}^{\Cf}(\lambda'^\ell)
        =
        \E^{\mathrm{P}}\big[{\rm J}^{\Phi,\Zf}_A(\Gamma',\mu',R,\mu) \big].
    \end{align*}
    Then $\Pr^\star \in \Pcb_{\rm mfg}$.

    \medskip
    Let $\mathrm{P}^\star \in \Pcb_{\rm 0,mfg}$ associated to $\Phi=\delta_{\phi(\cdot)}(\mathrm{d}i)$ and $\Zf=\delta_{z(\cdot)}(\mathrm{d}e)$.
    On the probability space $(\Om,\H,\P)$, we consider the process $(B,\overline{m},\overline{e})$ s.t. $\Lc^{\P}\left(B, \overline{R}\right)=\Lc^{\P}\left(B,\delta_{\left(\overline{m}_t, \overline{e}_t \right)}(\mathrm{d}m,\mathrm{d}e) \mathrm{d}t \right)=\Lc^{\Pr^\star}(B,R)$ where $\overline{R}:=\delta_{\left(\overline{m}_t, \overline{e}_t \right)}(\mathrm{d}m,\mathrm{d}e) \mathrm{d}t$. With the Borel map $\alpha$ associated to $\mathrm{P}^\star (\in \Pcb_{\rm 0,mfg})$, we define $X^\alpha$ the process verifying: $X^\alpha_0=\iota$ and $\P$--a.e.
    \begin{align*}
        \mathrm{d}X^\alpha_t
            =
            \overline{b}\left(t,\overline{m}_t,\overline{e}_t\right)\mathrm{d}t
            +
            b\left(t,X^\alpha_t,\mu,\alpha\left(t,X^\alpha_t,\mu,\overline{m}_t\right) \right)\mathrm{d}t + \sigma(t,X^\alpha_t)\mathrm{d}W_t + \sigma_0\mathrm{d}B_t.
    \end{align*}
    Notice that we can find a Borel map $\varphi$ s.t. $(\mu_{t \wedge \cdot},\overline{m}_t)=\varphi(t,\overline{R})$ $\mathrm{d}t \otimes \mathrm{d}\P$--a.e. 
    We set $\mu_t:=\overline{m}_t(\mathrm{d}x,A)$, $\mathrm{d}t \otimes \mathrm{d}\Pr^\star$--a.e., $\overline{\alpha}(t,x,\rr):=\alpha(t,x,\varphi(t,\rr))$ and $\Cf:=(\phi,z,\overline{e})$.  Then, it is easy to verify that $(\overline{\alpha},\overline{m}) \in {\rm MFG}[\Cf]$ and $\P \circ \left( \mu, R \right)^{-1}= \Pr^\star \circ(\mu,R)^{-1}$.  

\end{proof}

\subsection{From Nash equilibria to MFG solutions} \label{sec:Nash_to_MFG}

In this section we give the characterization of the limits of all sequences of approximate Nash equilibria and approximate MFG solutions.

\medskip
Let us take a control rule $\Pr \in \Pcb$ satisfying $\Lambda'(t,X_t',R)(\mathrm{d}a)\mathrm{d}t=\delta_{\beta(t,X_t',R)}(\mathrm{d}a)\mathrm{d}t$ $\Pr$--a.e. for some progressively Borel measurable map $[0,T] \x \R^d \x \M \ni (t,x,\rr) \mapsto \beta(t,x,\rr) \in A$ Lipschitz in $(x,\rr)$ uniformly in $t$. For each $n \ge 1,$ we consider $n$ progressively Borel measurable functions $\alphab^n:=(\alpha^{1,n},\cdots,\alpha^{n,n}) \in (\Ac_n)^n$  i.e. $\alpha^{i,n}: [0,T] \x (\Cc)^n \to A$ and we define for a fixed $i \in \{1,\cdots,n\}$  
\begin{align} \label{eq:def-controlBeta}
    \beta^{\alphab,i}(t,\boldsymbol{x})
    :=
    \beta \left(t, {x}^i(t),R^{\alphab}[\boldsymbol{x}]  \right)
\end{align}
where
\begin{align*}
    \boldsymbol{x}:=(x^1,\cdots,x^n),\; \overline{m}^{\alphab}(t,\boldsymbol{x}):=\frac{1}{n} \sum_{k=1}^n \delta_{\left({x}^k(t),\;\alpha^{k,n}(t,\boldsymbol{x}) \right)},\;\;\mbox{and}\;\;R^{\alphab}[\boldsymbol{x}]:=\delta_{ \left(\overline{m}^{\alphab}(t,\boldsymbol{x}), \aleph^n(t,\boldsymbol{x}) \right)}(\mathrm{d}m,\mathrm{d}e)\mathrm{d}t
\end{align*}
for some sequence of contracts $(\Cf^n)_{n \ge 1}=(\phi^n,\xi^n,\aleph^n)_{n \ge 1}$. Let $\Gr$ be the density $\Gr(x):=\frac{1}{1+|x|^p} \left(\int_{\R^d} \frac{1}{1+|x'|^p} \mathrm{d}x' \right)^{-1}$. We recall that the notations we will use below have been introduced in \Cref{sec:n_agents}.

\begin{lemma} \label{lemma:deviating-player}
    Let {\rm \Cref{assum:main1}} hold true and a sequence $(\alphab^n)_{n \in \N^*}$ s.t. for each $n \in \N^*,$ $\alphab^n:=(\alpha^{1,n},\cdots,\alpha^{n,n}) \in (\Ac_n)^n$,
    \begin{align} \label{cond:converg}
        \Lim_{n \to \infty} \P \circ \left(B, \varphi^{n}[\alphab^n], R^{\alphab^n}[\Xbb^n] \right)^{-1}
        =
        \Pr \circ \left(B, \mu,R \right)^{-1}\;\;\mbox{in}\;\Wc_p,
    \end{align}
    \begin{align*}
        \lim_{n \to \infty}\delta_{\phi^n(x)}(\mathrm{d}i)\mathrm{G}(x)\mathrm{d}x=\Phi(x)(\mathrm{d}i) \mathrm{G}(x)\mathrm{d}x\;\;\;\mbox{and}\;\;\;\lim_{n \to \infty} \Lc^\P\left(\xi^n(\Xbb^n), R^{\alphab^n}[\Xbb^n] \right)
        =
        \Zf(r)(\mathrm{d}e)\Lc^{\Pr} ( R )(\mathrm{d}r).
    \end{align*}
    Then
\begin{align*}
    &\Lim_{n \to \infty}
    \frac{1}{n} \sum_{i=1}^n
    J_{n,i}^{\Cf^n} ( \alpha^{1,n}, \cdots, \alpha^{i-1,n}, \beta^{\alphab^n,i}, \alpha^{i+1,n},\cdots, \alpha^{n,n} )
        =
        \E^{\Pr}\left[{\rm J}^{\Phi,\Zf}_A(\Gamma',\mu',R,\mu) \right].
\end{align*}
\end{lemma}

\begin{proof}
    The proof is divided into three parts.
    
    \medskip
    $\boldsymbol{\underline{\rm Step\;1}}$ Given the controls $\alphab:=(\alpha^{1,n},\dots,\alpha^{n,n}) \in (\Ac_n)^n,$ $\Xbb^{\alphab,\aleph}:=(X^{1},\dots,X^{n})$ satisfies \eqref{eq:N-agents_StrongMV_CommonNoise-law-of-controls}.
    For each $i \in \{1,\cdots,n\},$ we define 
    $$
        \alphab^i:=\left(\alpha^{1,n},\cdots,\alpha^{i-1,n},\beta^{\alphab,i},\alpha^{i+1,n},\cdots,\alpha^{n,n} \right)\;\;\mbox{and}\;\;\Ybb^i:=(Y^{i,1},\cdots,Y^{i,n}):=\Xbb^{\alphab^i,\aleph}.
    $$
    Notice that, $Y^{i,i}$ satisfies
    \begin{align*} 
        \mathrm{d}Y^{i,i}_t
        =
        \overline{b}\big(t,\mub^{i,n}_{t}, \aleph^n(t,\Ybb^i)\big)
        +        b\big(t,Y^{i,i}_{t},\mu^{i,n} ,\beta^{\alphab,i}(t,\Ybb^{i}) \big)\;\; \mathrm{d}t 
        +
        \sigma \big(t,Y^{i,i}_t \big) \mathrm{d}W^i_t
        +
        \sigma_0 \mathrm{d}B_t\;\mbox{with}\;Y^{i,i}_0=\iota^i
    \end{align*}
	with 
	\[
	    \mu^{i,n}_{t}(\mathrm{d} x) := \frac{1}{n}\sum_{k=1}^n \delta_{Y^{i,k}_{t} }(\mathrm{d} x)\;\;\mbox{and}\;\;
	    \mub^{i,n}_{t}(\mathrm{d} x, \mathrm{d} a) := \frac{1}{n} \left(\sum_{k \neq i}^n \delta_{(Y^{i,k}_{t} ,\;\alpha^{k,n}(t,\Ybb^{i}) )}(\mathrm{d} x, \mathrm{d} a) + \delta_{(Y^{i,i}_{t} ,\;\beta^{\alphab,i}(t,\Ybb^{i})) }(\mathrm{d} x, \mathrm{d} a) \right).
	\] 

 We now introduce
\begin{align*}
    R^{i,n}
    :=
    \delta_{\left( \mub^{i,n}_t, \aleph^n(t,\Ybb^i)   \right)} (\mathrm{d}m,\mathrm{d}e)\mathrm{d}t\;\mbox{and}\;\widetilde{R}^{i,n}:=\delta_{\left(\Tilde\mu^{i,n}_t,\aleph^n(t,\Ybb^i) \right)}(\mathrm{d}m,\mathrm{d}e)\mathrm{d}t\;\;\mbox{with}\;\;\mut^{i,n}_t:=\frac{1}{n} \sum_{k = 1}^n \delta_{(Y^{i,k}_t ,\;\alpha^{k,n}(t,\Ybb^{i}) )}(\mathrm{d} x, \mathrm{d} a).
\end{align*}
We can rewrite 
\begin{align*} 
        \mathrm{d}Y^{i,i}_t
        =
        \langle \overline{b}(t,\cdot,\cdot), R^{i,n}_t \rangle
        +
        b\left(t,Y^{i,i}_{t},\mu^{i,n},\beta (t,Y^{i,i}_t, \widetilde{R}^{i,n}) \right) \;\;\mathrm{d}t 
        +
        \sigma \big(t,Y^{i,i}_t \big) \mathrm{d}W^i_t
        +
        \sigma_0 \mathrm{d}B_t\;\mbox{with}\;Y^{i,i}_0=\iota^i.
    \end{align*}

Let $(\widehat{\Qr}^n)_{n \ge 1}$ be the sequence defined by 
\begin{align*}
    \widehat{\Qr}^n
    :=
    \frac{1}{n} \sum_{i=1}^n \P \circ \left( Y^{i,i}, R^{i,n}, \mu^{i,n}, \widetilde{R}^{i,n}, W^i,B  \right)^{-1}.
\end{align*}
Since $\nu \in \Pc_{p'}(\R^d)$ with $p' >p$, $(b,\sigma,\sigma_0)$ is bounded and $\Er \x \Ir$ is compact, it is straightforward to check that the sequence  $(\widehat{\Qr}^n)_{n \ge 1}$ is relatively compact in $\Wc_p$. We recall that for simplicity we have chosen $(\Om,\H,\P)$ s.t. we can write all our random variables on this space without enlarging the space. Let $\widehat{\Qr}^\infty=\P \circ \left( Y,R,\mu,\widetilde{R},W,B \right)^{-1}$ be the limit of a convergent sub--sequence. It is easy to see that $(W,B)$ is an $\R^d \x \R^d$ Brownian motion and, by using the continuity of the coefficients and classical weak convergence characterization, we get
\begin{align} \label{eq:Y_controlled}
    \mathrm{d}Y_t
        =
        \langle \overline{b}(t,\cdot,\cdot), R_t \rangle
        +
        b\left(t,Y_{t},\mu,\beta (t,Y_t, \widetilde{R}) \right) \;\;\mathrm{d}t 
        +
        \sigma \big(t,Y_t \big) \mathrm{d}W_t
        +
        \sigma_0 \mathrm{d}B_t\;\mbox{with}\;\Lc^{\P}(Y_0)=\Lc^{\P}(\iota^1).
\end{align}

Notice that by uniqueness of this equation, we can find a measurable map $\Upsilon:\R^d \x \Cc \x \Cc \x \Cc_{\Wc} \x \M \x \M \to \Cc$ s.t. $Y=\Upsilon\left(Y_0, W,B,\mu,R,\widetilde{R} \right)$. We will show that $\Lc\left( Y_0, W,B,\mu,R,\widetilde{R} \right)=\Pr \circ \left(X_0', W, B, \mu, R, R \right)^{-1}$. This will allow us to deduce that $\Lc\left( Y,R,\mu,\widetilde{R},W,B \right)=\Pr \circ \left(X',R,\mu,R, W, B\right)^{-1}$.

\medskip
$\boldsymbol{\underline{\rm Step\;2}}$
Next, if we set $R^n:=\delta_{\left( \overline{\varphi}^n_t[\alphab], \aleph^n(t,\Xbb)   \right)} (\mathrm{d}m,\mathrm{d}e)\mathrm{d}t$, we show that, for $F$ and $G$ two bounded continuous maps,
\begin{align*}
    \lim_{n \to \infty}\frac{1}{n} \sum_{i=1}^n \E^\P \left[ G(X^i_0,W^i) F \left(R^{i,n}, \mu^{i,n},\widetilde{R}^{i,n}, B \right)  \right]
    =
    \lim_{n \to \infty}\frac{1}{n} \sum_{i=1}^n \E^\P \left[ G(X^i_0,W^i) F \left(R^{n}, \varphi^n[\alphab],R^{n}, B \right) \right].
\end{align*}

\medskip
For this purpose, let us introduce
	\begin{align*}
	    Z^i_t
	    :=
	    \exp\bigg\{ \int_0^t \phi^i_r \mathrm{d}W^i_r - \frac{1}{2} \int_0^t |\phi^i_r|^2 \mathrm{d}r \bigg\}\;\mbox{for all}\;t \in [0,T],\; \mbox{and}\;\frac{\mathrm{d}\Q^i}{\mathrm{d}\P}:=Z^i_T
	\end{align*}
	with 
	\begin{align*}
\phi^i_t:=\sigma(t,X^{i}_t)^{-1}\Big(\overline{b}\left(t,\overline{\zeta}^{i,n,\beta}_{t},\aleph^n(t,\Xbb^{\alphab,\aleph}) \right)&+
b\left(t,X^{i}_{t},\varphi^{n}[\alphab],\beta^{\alphab,i}(t,\Xbb^{\alphab,\aleph}) \right)
\\&- \overline{b}\left(t,\overline{\varphi}^{n}_{t}[\alphab],\aleph^n(t,\Xbb^{\alphab,\aleph}) \right) - b\left(t,X^{i}_{t},\varphi^{n}[\alphab],\alpha^{i,n}(t,\Xbb^{\alphab,\aleph}) \right) \Big)
\end{align*}
	and
	\begin{align*}
	    \overline{\zeta}^{i,n,\beta}_t
	    :=
	    \frac{1}{n} \left(\sum_{k \neq i}^n \delta_{(X^{k}_{t} ,\;\alpha^{k,n}(t,\Xbb^{\alphab,\aleph}) )}(\mathrm{d} x, \mathrm{d} a) + \delta_{(X^{i}_{t} ,\;\beta^{\alphab,i}(t,\Xbb^{\alphab,\aleph})) }(\mathrm{d} x, \mathrm{d} a) \right).
	\end{align*}
    By uniqueness in distribution, $\Lc^{\Q^i} (\Xbb^{\alpha},B)=\Lc^{\P}(\Ybb^i,B)$ for each $i.$ Also, since the difference between $\mub^{i,n}_t$ and $\mut^{i,n}_t$ is only $\beta^{\alphab,i}$, it is easy to see that  $\Lim_{n \to \infty} \Wc_p \left(\mub^{i,n}_t,\mut^{i,n}_t \right)=0$ $\mathrm{d}t \otimes \mathrm{d}\P$--a.e. Then, we can notice that 
\begin{align*}
    &\lim_{n \to \infty}\frac{1}{n} \sum_{i=1}^n\E^\P \left[ G(X^i_0,W^i) F \left(R^{i,n}, \mu^{i,n},\widetilde{R}^{i,n}, B \right)  \right]
    \\
    &=\lim_{n \to \infty}\frac{1}{n} \sum_{i=1}^n\E^\P \left[ G(X^i_0,W^i) F \left(\widetilde{R}^{i,n}, \mu^{i,n},\widetilde{R}^{i,n}, B \right)  \right]
    =
    \lim_{n \to \infty}\frac{1}{n} \sum_{i=1}^n\E^\P \left[ Z^i_T G\left(X^i_0,W^i_\cdot-\int_0^\cdot \phi^i_t \mathrm{d}t \right) F \left(R^{n}, \varphi^n[\alphab^n],R^{n}, B \right)  \right]
    \\
    &=
    \lim_{n \to \infty}\E^\P \left[ \E^{\hat\mu^n} \left[  Z_T G\left(X_0,W_\cdot-\int_0^\cdot \int_{\Vc} v \;\Lambda_t(\mathrm{d}v) \mathrm{d}t \right) \right] F \left(R^{n}, \varphi^n[\alphab^n],R^{n}, B \right)  \right]
\end{align*}
where $(X_0,Z,W,\Lambda)$ is the canonical variable of $\R^d \x \Cc \x \Cc \x \M(\Vc)$ with $\Vc \subset \R^d$ an appropirate compact and
\begin{align*}
    \muh^n
    :=
    \frac{1}{n} \sum_{i=1}^n \delta_{\left(X^i_0, Z^i, W^i, \Lambda^i \right)}\;\;\mbox{where}\;\;\Lambda^i:=\delta_{\phi^i_t}(\mathrm{d}v)\mathrm{d}t.
\end{align*}
It is easy to see that the sequence $\left(\P \circ \left(\muh^n,R^n,\varphi^n[\alphab^n],B \right)^{-1} \right)_{n \ge 1}$ is relatively compact in $\Wc_p$. Let us take a convergent sub--sequence. For simplification, we use the same notation for the sequence and the sub--sequence.  So, we are considering $\lim_{n \to \infty} \P \circ \left(\muh^n,R^n,\varphi^n[\alphab^n],B \right)^{-1}=\P \circ (\muh, R, \mu, B)^{-1}$. By observing that $\mathrm{d}Z^i_t=Z^i_t\phi^i_t\mathrm{d}W^i_t$ and $Z^i_0=1$ for each $i$, then by using similar techniques as in \eqref{eq:Y_controlled} (see also \cite[Proof of Proposition 4.17.]{djete2019general}), we can verify that $\P$--a.e. $\muh = \muh \circ (X_0,Z,W,\Lambda)^{-1}$, $\muh$--a.e.
\begin{align*}
    \mathrm{d}Z_t = Z_t \int_{\Vc} v\; M(\mathrm{d}v,\mathrm{d}t)\mbox{   and   }W_t=M(\Vc \x [0,t])\mbox{  for all }t \in [0,T]
\end{align*}
where $Z_0=1$ and $M$ is a $\muh$--martingale measure with quadratic variation $\Lambda$ (see \cite{el1990martingale} for an overview on martingale measure). We define $\frac{\mathrm{d}\P^{\hat \mu}}{\mathrm{d}\hat \mu}:=Z_T$. By Girsanov Theorem, we can see that under $\P^{\hat \mu}$,  $W_\cdot-\int_0^\cdot \int_{\Vc} v \;\Lambda_t(\mathrm{d}v) \mathrm{d}t$ is a Brownian motion. Therefore
\begin{align} \label{eq:change_of_proba_equality}
    &\E^{\Pr} \left[ G(Y_0,W) F \left(R, \mu,\widetilde{R}, B \right)  \right]
    =
    \lim_{n \to \infty}\frac{1}{n} \sum_{i=1}^n\E^\P \left[ G(X^i_0,W^i) F \left({R}^{i,n}, \mu^{i,n},\widetilde{R}^{i,n}, B \right) \right] \nonumber
    \\
    &=
    \E^\P \left[\E^{\P^{\hat \mu}} \left[ G\left(X_0,W_\cdot-\int_0^\cdot \int_{\Vc} v \;\Lambda_t(\mathrm{d}v) \mathrm{d}t \right)  \right] F (R, \mu,R, B)  \right] \nonumber
    \\
    &=
    \E^\P \left[ G(X^1_0,W)  \right]\E^\P \left[ F (R, \mu,R, B)  \right] \nonumber
    \\
    &=
    \lim_{n \to \infty}\frac{1}{n} \sum_{i=1}^n \E^\P \left[ G(X^i_0,W^i) F \left({R}^{n}, \varphi^n[\alphab],{R}^{n}, B \right)  \right]=\E^{\Pr} \left[ G(X_0',W) F \left(R, \mu,R, B \right)  \right]
\end{align}
where we used the law of large number with the i.i.d sequence $(W^i,\iota^i)_{i \ge 1}$. This result remains true for any convergent sub--sequence of $\left(\P \circ \left(\muh^n,R^n,\varphi^n[\alphab^n],B \right)^{-1} \right)_{n \ge 1}$ and any bounded maps  $(G,F)$. It is enough to conclude that $\Lc\left( Y_0, W,B,\mu,R,\widetilde{R} \right)=\Pr \circ \left(X_0', W, B, \mu, R, R \right)^{-1}$, so
$$
    \lim_{n \to \infty} \widehat{\Qr}^n=\Lc^{\P}\left( Y,R,\mu,\widetilde{R},W,B \right)=\Pr \circ \left(X',R,\mu,R, W, B\right)^{-1}.
$$

\medskip
$\boldsymbol{\underline{\rm Step\;3}}$
Let $(\Qr^n)_{n \ge 1}$ be defined by $\Qr^n:= \P \circ \left( \delta_{\phi^n(x)}(\mathrm{d}i) \eta^n_T(\mathrm{d}x), \eta^n \right)^{-1}$ where $\eta^n_t:=\frac{1}{n} \sum_{i=1}^n \delta_{Y^{i,i}_t}(\mathrm{d}x)$. Since we know that $\lim_{n \to \infty}\delta_{\phi^n(x)}(\mathrm{d}i)\mathrm{G}(x)\mathrm{d}x=\Phi(x)(\mathrm{d}i) \mathrm{G}(x)\mathrm{d}x$, by \Cref{prop:charac-convergence}, the sequence $(\Qr^n)_{n \ge 1}$ is relatively compact in $\Wc_p$, and any limit point $\Qr=\P \circ (\kappa,\eta)^{-1}$ satisfies: $\P$--a.e. $\kappa=\Phi(x)(\mathrm{d}i)\eta_T(\mathrm{d}x)$. Notice that, 
$$
    \E^{\P} \left[ \eta_T(\mathrm{d}x) \right]
    =
    \lim_{n \to \infty} \frac{1}{n} \sum_{i=1}^n \Lc(Y^{i,i}_T)=\Pr \circ (X_T')^{-1}.
$$
Therefore,
\begin{align*}
    \lim_{n \to \infty} \frac{1}{n} \sum_{i=1}^n \E^{\P} \left[ g \left( Y^{i,i}_T, \phi^n(Y^{i,i}_T) \right) \right]
    =
    \lim_{n \to \infty}  \E^{\P} \left[ \int_{\Ir \x \R} g \left( x, i \right) \delta_{\phi^n(x)}(\mathrm{d}i)\eta^n_T(\mathrm{d}x) \right]
    &=
    \E^{\P} \left[ \int_{\Ir \x \R} g \left( x, i \right) \Phi(x)(\mathrm{d}i)\eta_T(\mathrm{d}x) \right]
    \\
    &=
    \E^{\P} \left[ \int_{\Ir} g \left( X_T', i \right) \Phi(X_T')(\mathrm{d}i) \right].
\end{align*}

\medskip
Also, since $\mu_T$ is a Borel map of $R$, by the same techniques used in \eqref{eq:change_of_proba_equality} through change of probability, we have
\begin{align*}
    \lim_{n \to \infty} \frac{1}{n} \sum_{i=1}^n \E^{\P} \left[ \overline{g} \left( \mu^{i,n}, \xi^n(\Ybb^i) \right) \right]
    =
    \lim_{n \to \infty} \frac{1}{n} \sum_{i=1}^n \E^{\P} \left[ Z^i_T\overline{g} \left( \varphi^n[\alphab^n], \xi^n(\Xbb) \right) \right]
    &=
    \lim_{n \to \infty} \E^{\P} \left[ \overline{g} \left( \varphi^n[\alphab^n], \xi^n(\Xbb) \right) \right]
    \\
    &=\E^{\Pr} \left[ \int_{\Er}\overline{g} \left( \mu, e \right) \Zf(R)(\mathrm{d}e) \right].
\end{align*}

\medskip    
    By combining all the results, we find
	\begin{align*}
	    &\lim_{n \to \infty} \frac{1}{n} \sum_{i=1}^n J_{n,i}^{\Cf^n} \big( \alpha^{1,n}, \cdots, \alpha^{i-1,n}, \beta^{\alphab,i}, \alpha^{i+1,n},\cdots, \alpha^{n,n} \big)
        \\
	    &=
	    \lim_{n \to \infty} \frac{1}{n} \sum_{i=1}^n\E^{\P} \left[
        \int_0^T \overline{L}\left(t,\mub^{i,n}_{t}, \aleph^n(t,\Ybb^i)\right)
        +
        L\left(t,Y^{i,i}_t,\mu^{i,n} ,\beta^{\alphab,i}(t,\Ybb^i) \right)\; \mathrm{d}t 
        +
        \overline{g}\left(\mu^{i,n},\xi^n(\Ybb^i) \right)
        + 
        g \big( Y^{i,i}_T, \phi^n(Y^{i,i}_T) \big)
        \right]
        \\
        &=
        \E^{\Pr}\left[{\rm J}^{\Phi,\Zf}_A(\Gamma',\mu',R,\mu) \right].
	\end{align*}

\end{proof}

\medskip
For each $n \ge 1$, we consider the controls $\alphab^n:=(\alpha^{1,n},\cdots,\alpha^{n,n}) \in (\Ac_n)^n$. Let us introduce
\begin{align*}
    \Pr^n
    :=
    \frac{1}{n} \sum_{i=1}^n \P \circ \left( X^i, W^i, B, R^n, \varphi^n[\alphab^n] \right)^{-1}\;\;
\end{align*}
$\mbox{where}\;R^n:=\delta_{\left( \overline{\varphi}^n_t[\alphab^n], \aleph^n(t,\Xbb_t)\right)}(\mathrm{d}m,\mathrm{d}e)\mathrm{d}t.$ We also consider the sequence of non--negative numbers $(\varepsilon_n)_{n \ge 1}$ satisfying $\Lim_{n \to \infty} \varepsilon_n=0.$
\begin{proposition} \label{prop:convegence_n-player}
    Under {\rm \Cref{assum:main1}}, the sequence $\left(\Pr^n \right)_{n \ge 1}$ is relatively compact in $\Wc_p$ with $p'>p$, and any limit point is a control rule in the sense of {\rm \Cref{def:RelaxedCcontrol}}.
    Moreover, if for each $n \ge 1$, $\alphab^n$ is an $\varepsilon_n$--Nash equilibrium then any limit point $\Pr$ of a convergent sub--sequence $\left(\Pr^{n_k} \right)_{k \ge 1}$ is a MFG solution i.e. belongs to $\Pcb_{\rm mfg}$ associated to $(\Phi,\Zf)$ which verifies: $\lim_{k \to \infty} \Lc^\P \left( \xi^{n_k}(\Xbb),\; R^{n_k} \right) 
        =
        \Zf(r)(\mathrm{d}e) \Lc^{\Pr} \left( R \right)(\mathrm{d}r)$ and
    \begin{align*}
        \lim_{k \to \infty} \Lc^\P \left( \varphi^{n_k}[\alphab^{n_k}],\; \delta_{\phi^{n_k}(x)}(\mathrm{d}i)\varphi^{n_k}_T[\alphab^{n_k}](\mathrm{d}x)\right) 
        =
        \Lc^{\Pr} \left( \mu,\; \Phi(x)(\mathrm{d}i)\mu_T(\mathrm{d}x) \right)\;\mbox{in}\;\Wc_p.
    \end{align*}  
\end{proposition}

\begin{proof}
    We set $\Lambda^i:=\delta_{\alpha^{i,n}(t,\Xbb)}(\mathrm{d}a)\mathrm{d}t$ and we define the sequence $(\Prt^n )_{n \ge 1}$ by 
    \begin{align*}
        \Prt^n
        :=
        \frac{1}{n} \sum_{i=1}^n \P \circ \left(\Lambda^i, X^i, W^i, B, R^n, \varphi^n[\alphab^n] \right)^{-1}.
    \end{align*}
    Since the initial distribution $\Lc^{\P}(X^1_0) \in \Wc_{p'}$, under {\rm \Cref{assum:main1}}, it is straightforward to check that $(\Prt^n )_{n \ge 1}$ is relatively compact in $\Wc_p$ with $p' >p$. Let $\Prt$ be the limit of a convergent sub--sequence of $(\Prt^n )_{n \ge 1}$. For simplification, we use the same notation for the sequence and its sub--sequence. 
    Notice that, we can rewrite the dynamics of $X^i$ as
    \begin{align*}
        \mathrm{d}X^i_t
        =
        \langle \overline{b}\left(t, \cdot, \cdot\right), R^n_t \rangle
        +
        \int_{A} b\left(t,X^i_{t}, \varphi^{n}[\alphab^n] ,a \right)\Lambda^i_t(\mathrm{d}a)\; \mathrm{d}t 
        +
        \sigma(t,X^i_t) \mathrm{d}W^i_t
        +
        \sigma_0 \mathrm{d}B_t.
    \end{align*}
    By using some classical martingale problem (see for instance \cite{djete2019general}), we can check that $\Prt=\P \circ \left( \Lambda, X, W,B, R, \mu \right)^{-1}$  where $W$ and $B$ are independent Brownian motions and $X$ satisfies: $\P$--a.e.,  $X_0=\iota$ and
    \begin{align*}
        \mathrm{d}X_t
        =
        \langle \overline{b}\left(t, \cdot, \cdot\right), R_t \rangle
        +
        \int_{A} b\left(t,X_{t}, \mu,a \right)\Lambda_t(\mathrm{d}a)\; \mathrm{d}t 
        +
        \sigma(t,X_t) \mathrm{d}W_t
        +
        \sigma_0 \mathrm{d}B_t.
    \end{align*}

    Let $F$, $V$ and $G$ be two bounded continuous functions
    \begin{align*}
        \E^{\P} \left[ V\left(X_0 \right) F\left( W \right) G \left(B, R,\mu \right) \right]
        &=
        \lim_{n \to \infty} \frac{1}{n} \sum_{i=1}^n\E^{\P} \left[ V\left(X^i_0 \right)F\left( W^i \right) G \left(B, R^n,\varphi^n[\alphab^n] \right) \right]
        \\
        &=
        \E^{\P} \left[ V\left(X_0 \right) \right]\E^{\P} \left[ F\left( W \right) \right]\E^{\P} \left[ G \left(B, R,\mu \right) \right]
    \end{align*}
    where we use the fact that $(X^i_0,W^i)_{i \ge 1}$ is an i.i.d. sequence combined with the law of large number. This is true for any maps $F$, $V$ and $G$, we can deduce that $(B,R,\mu)$, $W$ and $X_0$ are independent. By the definition of $R^n$ and $\varphi^n[\alphab^n]$, it is easy to check that $\mathrm{d}\P \otimes \mathrm{d}t$--a.e. $R_t\left(\{ (m,e):\;m(\mathrm{d}x,A)=\mu_t \} \right)=1$ (see also the proof of \cite[Proposition 4.4]{MFD-2020}). We can then see $\mu_t$ as a Borel map of $R_t$.  Next, we verify that $\mu_t=\Lc^\P\left(X_t| \Gc_T \right)$, $\P$--a.e. for all $t \in [0,T]$ where $\Gc_t:=\sigma\{B_{t \wedge \cdot}, R_{t \wedge \cdot}\}$. For this purpose, let $t \in [0,T]$ and, $f$ and $G$ be two bounded continuous maps, we have
    \begin{align*}
        \E^\P \left[f(X_t) G \left(B,R \right)\right]
        =
        \lim_{n \to \infty} \frac{1}{n} \sum_{i=1}^n \E^\P \left[f(X^i_t) G \left(B,R^n \right)\right]
        =
        \lim_{n \to \infty} \E^\P \left[\langle f, \varphi^n_t[\alphab^n] \rangle G \left(B,R^n \right)\right]
        =
        \E^\P \left[\langle f, \mu_t \rangle G \left(B,R \right)\right].
    \end{align*}
    This is true for any $(f,G)$. This is enough to conclude that $\mu_t=\Lc^\P\left(X_t| \Gc_T \right)$, $\P$--a.e. for all $t \in [0,T]$. Now, for any bounded continuous maps $(h,\varphi,G)$, we find 
    \begin{align*}
        &\E^{\P}\left[ \int_0^T  \int_A h(t)\varphi(X_t) b\left(t,X_t,\mu,a \right)\Lambda_t(\mathrm{d}a) \mathrm{d}t \; G\left(B, R \right)  \right]
        \\
        &=
        \lim_{n\to \infty}\frac{1}{n}\sum_{i=1}^n\E^{\P}\left[ \int_0^T  \int_A h(t)\varphi(X^i_t) b\left(t,X^i_t,\varphi^n[\alphab^n],a \right) \Lambda^i_t(\mathrm{d}a) \mathrm{d}t \;G\left(B, R^n \right)  \right]
        \\
        &=
        \lim_{n\to \infty}\E^{\P}\left[ \int_{[0,T] \x \Pc_A}  \int_{\R^d \x A} h(t)\varphi(x) b\left(t,x,\varphi^n[\alphab^n],a \right) m(\mathrm{d}x,\mathrm{d}a)R^n_t(\mathrm{d}m,\Er) \mathrm{d}t \; G\left(B, R^n\right)  \right]
        \\
        &=
        \E^{\P}\left[ \int_{[0,T] \x \Pc_A}  \int_{\R^d \x A}  h(t) \varphi(x) b\left(t,x,\mu,a \right) m(\mathrm{d}x,\mathrm{d}a) R_t(\mathrm{d}m,\Er) \mathrm{d}t \; G\left(B, R \right)  \right].
    \end{align*}
    This is true for any $(h,\varphi,G)$. We deduce that $\mathrm{d}t \otimes \mathrm{d}\P$--a.e. 
    \begin{align*}
         \E^{\P}\left[  \int_A\left(\varphi b \right)\left(t,X_t,\mu,a \right) \Lambda_t(\mathrm{d}a) |\Gc_T \right]
            =
            \int_{\Pc_A} \int_{\R^d \x A} \left(\varphi b\right)\left(t,x, \mu,a \right) m(\mathrm{d}x,\mathrm{d}a) R_t(\mathrm{d}m,\Er).
    \end{align*}
    By similar arguments, we get $ \E^{\P}\left[  \int_A L\left(t,X_t,\mu,a \right) \Lambda_t(\mathrm{d}a) |\Gc_T \right]
            =
            \int_{\Pc_A} \int_{\R^d \x A} L\left(t,x, \mu,a \right) m(\mathrm{d}x,\mathrm{d}a) R_t(\mathrm{d}m,\Er)$, $\mathrm{d}t \otimes \mathrm{d}\P$--a.e.  We set $\Lambda(t,x,\rr)(\mathrm{d}a)\mathrm{d}t:=\int_{\Pc_A} m(x)(\mathrm{d}a)\rr(t)(\mathrm{d}m,\Er)\mathrm{d}t$ where for $m \in \Pc_A$, $\R \ni x \mapsto m(x) \in \Pc(A)$ is a Borel map satisfying $m(x)(\mathrm{d}a)m(\mathrm{d}x,A)=m$. The previous result combined with \Cref{lemma:projection} allows to verify that $\mu_t=\Lc^\P \left(S_t|\Gc_T \right)$, $\P$--a.e. where
    \begin{align*}
        \mathrm{d}S_t
        =
        \langle \overline{b}\left(t, \cdot, \cdot\right), R_t \rangle
        +
        \int_{A} b\left(t,S_{t}, \mu ,a \right)\Lambda(t,S_t,R)(\mathrm{d}a)\; \mathrm{d}t 
        +
        \sigma(t,S_t) \mathrm{d}W_t
        +
        \sigma_0 \mathrm{d}B_t,\;\P\mbox{--a.e.}
    \end{align*}
    By combining all the results, we find that $\lim_{n \to \infty} \Pr^n= \P \circ \left(X,W,B,R,\mu \right)^{-1}$. Consequently, $\Pr=\Prt \circ \left( X,W,B,R,\mu \right)^{-1}$ is a control rule (see \Cref{rm:proba_equivalent}) . This is true for any convergent sub--sequence.

    \medskip
    We now prove the second part of the Proposition i.e. $\Pr \in \Pcb_{\rm mfg}$. Notice that the sequences $\left( \delta_{\phi^n(x)}(\mathrm{d}i) \Gr(x)\mathrm{d}x \right)_{n \ge 1}$ and $\left(\P \circ \left( \xi^n(\Xbb^n), R^n \right)^{-1} \right)_{n \ge 1}$ are relatively compact in $\Wc_p$. Let $\Phi(x)(\mathrm{d}i)\Gr(x)\mathrm{d}x$ and $\Zf(r)(\mathrm{d}e)\Lc^\P \left(R \right)(\mathrm{d}r)$ be the limits of a sub--sequence. We choose the sub--sequence $(n_k)_{k \ge 1}$ s.t.
    \begin{align*}
        (\Pr^{n_k})_{k \ge 1},  \;\left( \delta_{\phi^{n_k}(x)}(\mathrm{d}i) \Gr(x)\mathrm{d}x \right)_{k \ge 1} \mbox{ and } \left(\P \circ \left(\xi^{n_k}(\Xbb^{n_k}), R^{n_k} \right)^{-1} \right)_{k \ge 1}
    \end{align*} 
    are convergent. Again, we use the same notation for the sequence and the sub--sequence. It is straightforward that $\Lc^{\P}(R)=\Lc^{\Pr}(R)$ and by using \Cref{prop:charac-convergence}, we find that 
    \begin{align*}
        \lim_{n \to \infty} \P \circ \left( \delta_{\phi^n(x)}(\mathrm{d}i) \varphi^n_T[\alphab^n](\mathrm{d}x), \varphi^n[\alphab^n], R^n \right)^{-1}
        =
        \Pr \circ \left( \Phi(x)(\mathrm{d}i) \mu_T(\mathrm{d}x), \mu, R \right)^{-1}.
    \end{align*}

    We now check the optimality property. Let $\Qr$ be a control rule s.t. $\Lc^{\Pr}(B,R)=\Lc^{\Qr}(B,R)$. Thanks to \Cref{lemma:conv_adm-control}, we can assume that $\mathrm{d}\Qr \otimes \mathrm{d}t$--a.e. $\Lambda'(t,X_t',R)(\mathrm{d}a)=\delta_{\beta(t,X_t',R)}(\mathrm{d}a)$ where $\beta$ is a Lipschitz map in $(x,\rr)$ uniformly in $t$. 
    Next, we can use \Cref{lemma:deviating-player} and find
    \begin{align*}
        \E^{\Qr} \big[{\rm J}^{\Phi,\Zf}_A(\Gamma',\mu',R,\mu) \big]
        &=
        \Lim_{n \to \infty}
        \frac{1}{n} \sum_{i=1}^n
        J_{n,i}^{\Cf^n} \big( \alpha^{1,n}, \cdots, \alpha^{i-1,n}, \beta^{\alphab^n,i}, \alpha^{i+1,n},\cdots, \alpha^{n,n} \big)
        \\
        &\le 
        \Lim_{n \to \infty}
        \frac{1}{n} \sum_{i=1}^n
        J_{n,i}^{\Cf^n} \left( \alphab^n \right) + \varepsilon_n
        =
        \E^{\Pr} \big[{\rm J}^{\Phi,\Zf}_A(\Gamma',\mu',R,\mu) \big].
    \end{align*}
    We can therefore conclude the proof of the Proposition.
\end{proof}

\medskip
We provide here a result similar to \Cref{lemma:deviating-player} for approximate MFG solutions. We take a control rule $\Pr \in \Pcb$ satisfying $\Lambda'(t,X_t',R)(\mathrm{d}a)\mathrm{d}t=\delta_{\beta(t,X_t',R)}(\mathrm{d}a)\mathrm{d}t$ $\Pr$--a.e. for some progressively Borel map $[0,T] \x \R \x \M \ni (t,x,\rr)\mapsto \beta(t,x,\rr) \in A$ Lipschitz in $(x,\rr)$ uniformly in $t$.

\begin{lemma} \label{lemma:deviating MFGstrong}
    Let {\rm \Cref{assum:main1}} hold true. We consider a sequence of contracts $\left(\Cf^\ell=(\phi^\ell,\xi^\ell,\aleph^\ell) \right)_{\ell \ge 1}$ and a sequence $(\alpha^\ell, \mub^\ell)_{\ell \ge 1}$ s.t. for each $\ell \ge 1,$  $(\alpha^\ell,\mub^\ell) \in {\rm MFG}[\Cf^\ell,\varepsilon_\ell]$,
    \begin{align*}
        \Lim_{\ell \to \infty} \P \circ \left(B, \mu^\ell, \delta_{\left( \mub^\ell_t,\aleph^\ell_t \right)} (\mathrm{d}m,\mathrm{d}e) \mathrm{d}t \right)^{-1}
        =
        \Pr \circ \left(B, \mu,R \right)^{-1}\;\;\mbox{in}\;\Wc_p
    \end{align*}
    and
    \begin{align*}
        \lim_{\ell \to \infty}\delta_{\phi^\ell(x)}(\mathrm{d}i)\mathrm{G}(x)\mathrm{d}x=\Phi(x)(\mathrm{d}i) \mathrm{G}(x)\mathrm{d}x\;\;\;\mbox{and}\;\;\;\lim_{\ell \to \infty} \Lc^\P \left( \xi^\ell(\mub^\ell,\aleph^\ell), \delta_{\left( \mub^\ell_t,\aleph^\ell_t \right)} (\mathrm{d}m,\mathrm{d}e) \mathrm{d}t\right)
        =
         \Zf(r)(\mathrm{d}e)\Lc^{\Pr} \left( R \right)(\mathrm{d}r).
    \end{align*}
    Then, there exists a sequence $(\beta^\ell)_{\ell \ge 1} \subset \Ac$ s.t.
\begin{align*}
    &\Lim_{\ell \to \infty}
    J_{A,\mub^\ell}^{\Cf^\ell}(\beta^\ell)
        =
        \E^{\Pr}\left[{\rm J}^{\Phi,\Zf}_A(\Gamma',\mu',R,\mu) \right].
\end{align*}
\end{lemma}

\begin{proof}
    Recall $X'$ satisfies 
     \begin{align*}
        \mathrm{d}X'_t
        =
        \int_{\Pc_A \x \Er} \overline{b}(t,m,e)R_t(\mathrm{d}m,\mathrm{d}e) +b(t,X'_t,\mu_t,\beta(t,X'_t,R))\;\; \mathrm{d}t
        +
        \sigma(t,X'_t) \mathrm{d}W_t
        + \sigma_0 \mathrm{d}B_t\;\mbox{with}\;\Lc(X'_0)=\nu.
    \end{align*}
    On the space $(\Om, \H, \P)$, we introduce $X'^\ell$ the solution of: $\Lc(X'^\ell_0)=\nu$,
    \begin{align*}
        \mathrm{d}X'^\ell_t
        =
        \overline{b}(t,\mub^\ell_t,\aleph^\ell_t) +b(t,X'^\ell_t,\mu^\ell,\beta(t,X'^\ell_t,R^\ell))\;\; \mathrm{d}t
        +
        \sigma(t,X'^\ell_t) \mathrm{d}W_t
        + \sigma_0 \mathrm{d}B_t\;\mbox{with}\;\;R^\ell:=\delta_{\left( \mub^\ell_t, \aleph^\ell_t \right)}(\mathrm{d}m,\mathrm{d}e)\mathrm{d}t.
    \end{align*}
    Since $\Lim_{\ell \to \infty} \P \circ \left(B, \mu^\ell, \delta_{\left( \mub^\ell_t,\aleph^\ell_t \right)} (\mathrm{d}m,\mathrm{d}e) \mathrm{d}t \right)^{-1}
        =
        \Pr \circ \left(B, \mu,R \right)^{-1}$ and $\left(B, \mu^\ell, \delta_{\left( \mub^\ell_t,\aleph^\ell_t \right)} (\mathrm{d}m,\mathrm{d}e) \mathrm{d}t \right)$, $X_0$ and $W$ are independent for each $\ell \ge 1$, by using martingale problem combined with similar arguments to  \Cref{lemma:deviating-player}, we can deduce that
        \begin{align*}
            \Lim_{\ell \to \infty} \P \circ \left(X'^\ell, W,B, \mu^\ell, \delta_{\left( \mub^\ell_t,\xi^\ell_t \right)} (\mathrm{d}m,\mathrm{d}e) \mathrm{d}t \right)^{-1}
        =
        \Pr \circ \left(X',W,B, \mu,R \right)^{-1}.
        \end{align*}
        Again, by the similar techniques used in \Cref{lemma:deviating-player} (without change of probability), we deduce that $\Lim_{\ell \to \infty}
    J_{A,\mub^\ell}^{\Cf^\ell}(\beta^\ell)
        =
        \E^{\Pr}\left[{\rm J}^{\Phi,\Zf}_A(\Gamma',\mu',R,\mu) \right]$.
\end{proof}

\medskip
With the sequence $(\alpha^{\ell},\mub^\ell)_{\ell \ge 1}$ given in \Cref{lemma:deviating MFGstrong} i.e. for each $\ell \ge 1$, $(\alpha^\ell,\mub^\ell) \in {\rm MFG}[\Cf^\ell,\varepsilon_\ell]$, we define
\begin{align*}
    \Pr^\ell
    :=
    \P \circ \left( X^{\alpha^\ell}, W, B, R^\ell, \mu^\ell \right)^{-1}\;\;
\end{align*}
$\mbox{where}\;\;R^\ell:=\delta_{\left( \mub^\ell_t, \aleph^\ell_t\right)}(\mathrm{d}m,\mathrm{d}e)\mathrm{d}t$ and $(\varepsilon_\ell)_{\ell \ge 1}$ is a sequence of non--negative numbers satisfying $\Lim_{\ell \to \infty} \varepsilon_\ell=0.$
\begin{proposition} \label{prop_convergence:limitcontract}
    Under {\rm \Cref{assum:main1}}, the sequence $\left(\Pr^\ell \right)_{\ell \ge 1}$ is relatively compact in $\Wc_p$ with $p'>p$, and any limit point $\Pr$ of a convergent sub--sequence $\left(\Pr^{\ell_k} \right)_{k \ge 1}$ belongs to $\Pcb_{\rm mfg}$ associated to $(\Phi,\Zf)$ that satisfies: $\lim_{k \to \infty} \Lc^\P \left( \xi^{\ell_k}(\mub^{\ell_k},\aleph^{\ell_k}),\; R^{\ell_k} \right) 
        =
         \Zf(r)(\mathrm{d}e)\Lc^{\Pr} \left(R \right)(\mathrm{d}e)$ and 
    \begin{align*}
        \lim_{k \to \infty} \Lc^\P \left( \mu^{\ell_k},\; \delta_{\phi^{n_k}(x)}(\mathrm{d}i)\mu^{\ell_k}_T(\mathrm{d}x) \right) 
        =
        \Lc^{\Pr} \left( \mu,\; \Phi(x)(\mathrm{d}i)\mu_T(\mathrm{d}x) \right)\;\mbox{in}\;\Wc_p.
    \end{align*}  
\end{proposition}

\begin{proof}
    The proof is quite similar to the proof of \Cref{prop:convegence_n-player}. For any $\ell \ge 1$, we  can rewrite $X^{\alpha^\ell}$ as follows
    \begin{align*}
        \mathrm{d}X^{\alpha^\ell}_t
            =
            \langle\overline{b}\left(t,\cdot,\cdot\right), R^\ell_t \rangle
            +
            \int_A b\left(t,X^{\alpha^\ell}_t,\mu^\ell,a\right)\Lambda^\ell_t(\mathrm{d}a)\;\;\mathrm{d}t + \sigma(t,X^{\alpha^\ell}_t)\mathrm{d}W_t + \sigma_0\mathrm{d}B_t\;\;\mbox{with}\;\;\Lambda^\ell:=\delta_{\alpha^\ell(t,X^{\alpha^\ell}_t,R^\ell)}(\mathrm{d}a)\mathrm{d}t.
    \end{align*}
    We define
    \begin{align*}
        \Prt^\ell
        :=
        \P \circ \left( \Lambda^\ell, X^{\alpha^\ell}, W, B, R^\ell, \mu^\ell \right)^{-1}.
    \end{align*}
    Usual techniques allow to say that the sequence $(\Prt^\ell)_{\ell \ge 1}$ is relatively compact $\Wc_p$. Let $\Prt=\P \circ \left( \Lambda,X,W,B,R,\mu \right)^{-1}$ be the limit of a convergent sub--sequence. Similar to the previous proof, we use the same notation for the sequence and the sub--sequence. By standard techniques, the probability $\Prt=\P \circ \left( \Lambda,X,W,B,R,\mu \right)^{-1}$ is s.t. $(W,B)$ is an $\R^d \x \R^d$--valued Brownian motion, $X$ verifies
    \begin{align*}
        \mathrm{d}X_t
            =
            \langle\overline{b}\left(t,\cdot,\cdot\right), R_t \rangle
            +
            \int_A b\left(t,X_t,\mu_t,a\right)\Lambda_t(\mathrm{d}a)\;\;\mathrm{d}t + \sigma(t,X_t)\mathrm{d}W_t + \sigma_0\mathrm{d}B_t.
    \end{align*}
    We set $\G=(\Gc_t)_{t \in [0,T]}:=\left( \sigma\{ B_{t \wedge \cdot}, R_{t \wedge \cdot}\} \right)_{t \in [0,T]}$. By using the similar approach used in the proof of \Cref{prop:convegence_n-player}, we easily check that $\mathrm{d}t \otimes \mathrm{d}\P$--a.e., $R_t \left( \{(m,e):\;\mu_t=m(\mathrm{d}x,A)\} \right)=1$, $\mu_t=\Lc(X_t|\Gc_T)$ and
    \begin{align*}
         \E^{\P}\left[  \int_A\left(\varphi b \right)\left(t,X_t,\mu,a \right) \Lambda_t(\mathrm{d}a) |\Gc_T \right]
            =
            \int_{\Pc_A} \int_{\R^d \x A} \left(\varphi b\right)\left(t,x, \mu,a \right) m(\mathrm{d}x,\mathrm{d}a) R_t(\mathrm{d}m,\Er)\;\;\mbox{for any }\varphi \in C_c,
    \end{align*}
    and,  $ \E^{\P}\left[  \int_A L\left(t,X_t,\mu,a \right) \Lambda_t(\mathrm{d}a) |\Gc_T \right]
            =
            \int_{\Pc_A} \int_{\R^d \x A} L\left(t,x, \mu,a \right) m(\mathrm{d}x,\mathrm{d}a) R_t(\mathrm{d}m,\Er)$, $\mathrm{d}t \otimes \mathrm{d}\P$--a.e. We can find a Borel map $K$ verifying $K(t,x,\rr)=m(x)(\mathrm{d}u) \rr(t)(\mathrm{d}m,\Er)$ where $(m(x))_{x \in \R} $ satisfies $m=m(x)(\mathrm{d}a)m(\mathrm{d}x,A)$. It follows that (see \Cref{lemma:projection}) $\mu_t=\Lc(S_t|\Gc_T)$ where $\Lc(S_0)=\nu$ and
    \begin{align*}
        \mathrm{d}S_t
        =
        \langle \overline{b}\left(t, \cdot, \cdot\right), R_t \rangle
        +
        \int_{A} b\left(t,S_{t}, \mu ,a \right)K(t,S_t,R)(\mathrm{d}a)\; \mathrm{d}t 
        +
        \sigma(t,S_t) \mathrm{d}W_t
        +
        \sigma_0 \mathrm{d}B_t,\;\P\mbox{--a.e.}
    \end{align*}
    We deduce that $\Pr:=\P \circ \left( X,W,B,R,\mu \right)^{-1} \in \Pcb$ (see \Cref{rm:proba_equivalent}). The probability $\Pr$ is in fact in $\Pcb_{\rm mfg}.$ We consider the unique sub--sequence $(\ell_k)_{k \ge 1}$ s.t. $(\Pr^{\ell_k})_{k \ge 1}$, $\left( \delta_{\phi^{\ell_k}(x)}(\mathrm{d}i) \Gr(x)\mathrm{d}x \right)_{k \ge 1}$ and $\left(\P \circ \left( \xi^{\ell_k}, R^{\ell_k} \right)^{-1} \right)_{k \ge 1}$ are convergent sequences. Let $\Phi(x)(\mathrm{d}i)\Gr(x)\mathrm{d}x$ and $ \Zf(r)(\mathrm{d}e)\Lc^{\Pr} \left( R \right)(\mathrm{d}r)$ be the limit of $\left( \delta_{\phi^{\ell_k}(x)}(\mathrm{d}i) \Gr(x)\mathrm{d}x \right)_{k \ge 1}$ and $\left(\P \circ \left( \xi^{\ell_k}, R^{\ell_k} \right)^{-1} \right)_{k \ge 1}$. By \Cref{lemma:deviating MFGstrong} and similar techniques to the proof of \Cref{prop:convegence_n-player}, we deduce that $\Pr \in \Pcb_{\rm mfg}$ associated to $(\Phi,\Zf)$, with $(\Phi,\Zf)$ satisfying the convergence mentioned. 
\end{proof}

\subsection{From MFG solutions to approximate Nash equilibria} \label{sec:MFG_to_Nash}

We now deal in this section with the construction of approximate MFG solutions and approximate Nash equilibria from MFG solutions.

\medskip
On the probability space $(\Om,\H,\P)$, we consider the processes $(X,W,B,R,\mu)$ s.t. $(W,B)$ is an $\R^{d} \x \R^d$--Brownian motion, $R$ is an $\Pc_A \x \Er$--valued $\H$--predictable process, $(X,\mu)$ is an $\R^d \x \Pc(\R^d)$--valued $\H$--adapted continuous process, $\mu_t=\Lc^\P(X_t|\Gc_T)$, $R_t\left(\{(m,e):\;\;m(\mathrm{d}x,A)=\mu_t\} \right)=1$ $\mathrm{d}\P \otimes \mathrm{d}t$--a.e. and
\begin{align*}
    \mathrm{d}X_t
            =
            \int_{\Pc_A \x \Er}\overline{b}\left(t,m,e\right) R_t(\mathrm{d}m,\mathrm{d}e)
            +
            \int_{\Pc_A}\int_A b\left(t,X_t,\mu,a\right)m(X_t)(\mathrm{d}a)R_t(\mathrm{d}m,\Er)\mathrm{d}t + \sigma(t,X_t)\mathrm{d}W_t + \sigma_0\mathrm{d}B_t,\;\;\P\mbox{--a.e.}
\end{align*}
where $\G=\left(\Gc_t \right)_{t \in [0,T]}:= \left(\sigma\{ R_{t \wedge \cdot}, B_{t \wedge \cdot}\} \right)_{t \in [0,T]}$ and for each $m \in \Pc_A$, the Borel map $\R^d \ni x \mapsto m(x) \in \Pc(A)$ satisfies $m=m(x)(\mathrm{d}a)m(\mathrm{d}x,A).$ Besides, $\Lc^\P(X_0)=\nu$ and, $(R,B)$, $X_0$ and $W$ are independent. 

\begin{lemma} \label{lemma:first_appro}
    There exists a sequence $(\mub^\ell,\aleph^\ell)_{\ell \ge 1}$ s.t. for each $\ell \ge 1$, $\mub^\ell_t=\Lc^{\P} \left( X^\ell_t, \alpha^\ell_t | \Gc_t \right)=\Lc^{\P} \left( X^\ell_t, \alpha^\ell_t | \Gc_T \right)$ $\mathrm{d}t \otimes \mathrm{d}\P$--a.e. where
    \begin{align*}
        \mathrm{d}X^\ell_t
            =
            \overline{b}\left(t,\mub^\ell_t,\aleph^\ell_t\right)
            +
            b\left(t,X^\ell_t,\mu^\ell_t,\alpha^\ell_t\right)\mathrm{d}t + \sigma(t,X^\ell_t)\mathrm{d}W_t + \sigma_0\mathrm{d}B_t,\;\;\mu^\ell_t= \left( X^\ell_t| \Gc_t \right),\;\;\P\mbox{--a.e.},
    \end{align*}
    $(\alpha^\ell, \aleph^\ell)$ is an $A \x \Er$--valued $\F$--predictable process and
    \begin{align*}
        \lim_{\ell \to \infty} \left( \mu^\ell, \delta_{\left( \mub^\ell_t, \aleph^\ell_t \right)} (\mathrm{d}m,\mathrm{d}e)\mathrm{d}t \right)= \left( \mu,R \right),\;\mbox{in }\Wc_p,\;\P\mbox{--a.e.}
    \end{align*}
    In addition, for any bounded Borel map $\Bc:\R^d \to \R$,
    \begin{align*}
        \lim_{\ell} \E^\P \left[ \left| \int_{\R^d} \Bc(x) \mu^\ell_T(\mathrm{d}x)
        -
        \int_{\R^d} \Bc(x) \mu_T(\mathrm{d}x) \right| \right]=0.
    \end{align*}
\end{lemma}

\begin{proof}

\medskip
    We introduce the following quantities: for each $(m,\pi) $ and ${\rm b} \in \R^d$, $m[\br](\mathrm{d}y,\mathrm{d}a):=\int_{\R^d} \delta_{x+\sigma_0 \br}(\mathrm{d}y)m(\mathrm{d}x,\mathrm{d}a)$ and $\pi(t)[\br](\mathrm{d}y):=\int_{\R^d} \delta_{x+\sigma_0 \br}(\mathrm{d}y)\pi(t)(\mathrm{d}x)$. We set $\overline{m}_t:=m[-B_t]$, $\mut_t:=\mu_t[-B_t]$ and $Y_t:=X_t-\sigma_0 B_t$ for all $t \in [0,T]$. The processes $\left(\mut_t,Y_t \right)_{t \in [0,T]}$ satisfies $\;\P\mbox{--a.e.}$
    \begin{align*}
        \mathrm{d}Y_t
        =
            \int_{\Pc_A \x \Er}\overline{b}\left(t,\overline{m}_t[B_t],e\right) R_t(\mathrm{d}m,\mathrm{d}e)
           +
            \int_{\Pc_A}\int_A b\left(t,Y_t + \sigma_0 B_t,\mut_t[B_t],a\right)m(Y_t+ \sigma_0& B_t)(\mathrm{d}a)R_t(\mathrm{d}m,\Er)\;\mathrm{d}t 
            \\
            &+ \sigma(t,Y_t + \sigma_0 B_t)\mathrm{d}W_t.
    \end{align*}
    Let $\Theta_t(\mathrm{d}m,\mathrm{d}e):=\int_{\Pc_A}\delta_{{m}'[-B_t]}(\mathrm{d}m)R_t(\mathrm{d}m',\mathrm{d}e)$. Notice that $\Theta_t(\{(m,e):\;\;m(\mathrm{d}x,A)=\mut_t\})=1$ $\mathrm{d}\P \otimes \mathrm{d}t$--a.e. 
    We define the maps 
    \begin{align*}
        \overline{h}\left(t,\br,m,e \right)
        :=
        \overline{b}\left( t, m[\br],e\right),\;\;h\left(t,\br,y,\pi,a \right)
        :=
        b\left(t,y+\sigma_0 \br, \pi[\br], a \right)\;\;\mbox{and}\;\;v(t,\br,y):=\sigma(t,y+\sigma_0 \br).
    \end{align*}
    Under \Cref{assum:main1}, it is straightforward to check that 
    $$
        [0,T] \x \R^d \x \Pc_A \x \Er \x \R^d \x \Cc_{\Wc} \x A \ni (t,\br,m,e,y,\pi,a) \mapsto \left( \overline{h}\left(t,\br,m,e \right), h\left(t,\br,y,\pi,a \right), v(t,\br,y)  \right) \in \R^d \x \R^d \x \S^d
    $$
    is continuous in $(\br,m,e,y,\pi,a)$ for each $t$ and Lipschitz in $(m,y,\pi)$ uniformly in $(t,\br,e,a)$. We can rewrite $Y$ as
    \begin{align*}
        \mathrm{d}Y_t
            =
            \int_{\Pc_A \x \Er}\overline{h}\left(t,B_t,m,e\right) \Theta_t(\mathrm{d}m,\mathrm{d}e)
            +
            \int_{\Pc_A}\int_A h\left(t,B_t,Y_t ,\mut,a\right)m^{Y_t}(\mathrm{d}a) \Theta_t(\mathrm{d}m)\mathrm{d}t + v(t,B_t,Y_t)\mathrm{d}W_t\;\;\Pr\mbox{--a.e.}
    \end{align*}
    By applying It\^o formula and taking the conditional expectation given the filtration $\Gc_T$, we find: for all $\varphi$
    \begin{align*}
        \mathrm{d}\langle \varphi, \mut_t\rangle
        =
        \int_{\Pc_A \x \Er}\int_{\R^d \x A} \varphi'(y) \;\;\overline{h}(t,B_t,m,e) + h \left(t,B_t,y ,\mut,a\right)\;\;m(\mathrm{d}y,\mathrm{d}a)\Theta_t(\mathrm{d}m,\mathrm{d}e) \mathrm{d}t + \frac{1}{2} \langle \varphi''(\cdot) v^2(t,B_t,\cdot), \mut_t \rangle \mathrm{d}t.
    \end{align*}
    We can find a sequence of an $\Pc_A \x \Er$--valued $\G$--predictable processes $(\widetilde{m}^\ell,\aleph^\ell)_{\ell \ge 1}$ s.t. $\lim_{\ell \to \infty} \delta_{(\tilde{m}^\ell_t,\aleph^\ell_t)}(\mathrm{d}m,\mathrm{d}e)\mathrm{d}t=\Theta$ $\P$--a.e. By using \cite[Proposition 5.8]{MFD-2020}, there exists a sequence of $A$--valued $\F$--predictable processes $(\alpha^\ell)_{\ell \ge 1}$ s.t. if $Y^\ell$ satisfies: $\P$--a.e. 
    \begin{align*}
        \mathrm{d}Y^\ell_t
            =
            \overline{h}\left(t,B_t,\overline{m}^\ell_t,\aleph^\ell_t\right)
            + h\left(t,B_t,Y^\ell_t ,\mut^\ell,\alpha^\ell_t\right)\mathrm{d}t + v(t,B_t,Y^\ell_t)\mathrm{d}W_t,\;\overline{m}^\ell_t:=\Lc(Y^\ell_t,\alpha^\ell_t|\Gc_t)\;\mbox{and}\;\mut^\ell_t:=\Lc(Y^\ell_t|\Gc_t)
    \end{align*}
    we have 
    \begin{align*}
        \lim_{\ell \to \infty}\E^\P \left[ \int_0^T \Wc_p \left( \overline{m}^\ell_t , \widetilde{m}^\ell_t  \right) \mathrm{d}t + \sup_{t \in [0,T]} \Wc_p \left(\mut_t^\ell, \mut_t \right) \right]
        =0.
    \end{align*}
    It is easy to check that $\lim_{\ell \to \infty}\E^\P \left[ \int_0^T \Wc_p \left( \overline{m}^\ell_t[B_t], \widetilde{m}^\ell_t[B_t]  \right) \mathrm{d}t + \sup_{t \in [0,T]} \Wc_p \left(\mut_t^\ell[B_t], \mut_t[B_t] \right) \right]
        =0$. Notice that, $\int_{\Pc_A}\delta_{m'[B_t]}(\mathrm{d}m) \Theta_t(\mathrm{d}m',\mathrm{d}e)\mathrm{d}t=R_t(\mathrm{d}m,\mathrm{d}e)\mathrm{d}t$, $\mut_t[B_t]=\mu_t$ and if we set $\mub^\ell_t:=\overline{m}^\ell_t[B_t]$, $\mu_t^\ell:=\mut^\ell_t[B_t]$, $X^\ell_\cdot:=Y^\ell_\cdot + \sigma_0 B_\cdot$, we get
        \begin{align*}
        \mathrm{d}X^\ell_t
            =
            \overline{b}\left(t,\mub^\ell_t,\aleph^\ell_t\right)
            +
            b\left(t,X^\ell_t,\mu^\ell_t,\alpha^\ell_t\right)\mathrm{d}t + \sigma(t,X^\ell_t)\mathrm{d}W_t + \sigma_0\mathrm{d}B_t
    \end{align*}
    and, up to a sub--sequence, $\P$--a.e., $\lim_{\ell \to \infty} \delta_{(\mub^\ell_t,\aleph^\ell_t)}(\mathrm{d}m,\mathrm{d}e)\mathrm{d}t=R$ and $\lim_{\ell \to \infty} \mu^\ell=\mu$.
    This is enough to deduce the proof of the first part of the Proposition. The second part is concluded by (an easy adaptation of) \Cref{prop:charac-convergence}.

\end{proof}

\medskip
Let $\alpha$ be an $A$--valued $\F$--predictable process and $\aleph$ be an $\Er$--valued $\F$--predictable process. We consider $X$ satisfying: $\Lc^\P(X_0)=\nu$,
\begin{align*}
        \mathrm{d}X_t
            =
            \overline{b}\left(t,\mub_t,\aleph_t\right)
            +
            b\left(t,X_t,\mu_t,\alpha_t\right)\mathrm{d}t + \sigma(t,X_t)\mathrm{d}W_t + \sigma_0\mathrm{d}B_t,\;\mub_t=\Lc(X_t,\alpha_t|\Gc_T),\;\mu_t= \left( X_t| \Gc_T \right),\;\;\P\mbox{--a.e.},
    \end{align*}
where $\Gc_t:=\sigma\{ R_{t \wedge \cdot},B_{t \wedge \cdot}\}$ with $R:=\delta_{\left(\mub_t,\aleph_t \right)}(\mathrm{d}m,\mathrm{d}e)\mathrm{d}t.$ The variables $(R,B)$, $X_0$ and $W$ are independent.

\medskip
Let $(\varepsilon_k)_{k \ge 1}$ be a sequence of positive numbers verifying $\lim_{k \to \infty} \varepsilon_k=0$. And, we set $W^k_\cdot:=W_{\cdot \vee \varepsilon_k}-W_{\varepsilon_k}$ and $B^k_\cdot:=B_{\cdot \vee \varepsilon_k}-B_{\varepsilon_k}$. 
\begin{lemma} \label{lemma_piece-wiseAppr}
    There exist a sequence of piece--wise constant $A$--valued $\F$--predictable processes $(\alpha^k)_{k \ge 1}$ and a sequence of piece--wise constant $\Er$--valued $\G$--predictable processes $(\aleph^k)_{k \ge 1}$ s.t.
    \begin{align*}
        \lim_{k \to \infty}\E^\P \left[ \sup_{t \in [0,T]} \left| X^k_t-X_t \right|^p + \int_0^T \rho_A \left( \alpha_t, \alpha^k_t \right)^p \mathrm{d}t
        +
        \int_0^T \rho_{\Er}\left( \aleph_t, \aleph^k_t \right)^p \mathrm{d}t\right]
        =
        0
    \end{align*}
    where $X^k$ is satisfying: $X^k_t=X_0$ for $t \in [0,\varepsilon_k]$ and $\mbox{for }t \in [\varepsilon_k,T],\;$
\begin{align} \label{eq:piece-wiseAppr}
        \mathrm{d}X^k_t
            =
            \overline{b}\left(t,\mub^k_t,\aleph^k_t\right)
            +
            b\left(t,X^k_t,\mu^k_t,\alpha^k_t\right)\mathrm{d}t + \sigma(t,X^k_t)\mathrm{d}W^k_t + \sigma_0\mathrm{d}B^k_t,\;\mub^k_t=\Lc(X^k_t,\alpha^k_t|\Gc_T),\;\mu^k_t= \Lc\left( X^k_t| \Gc_T \right),\;\;\P\mbox{--a.e.}
    \end{align}
    
\end{lemma}

\begin{proof}
    This result comes from an application of \cite[Lemma 4.3]{djete2019general}. To justify the measurability of the processes $(\alpha^k)_{k \ge 1}$ and $(\aleph^k)_{k \ge 1}$, it is enough to see that these processes are constructed by considering first a sub--division $0=t^k_0<t^k_1:=\varepsilon^k< \cdots<t^k_k=T$ verifying $\lim_{k \to \infty} \sup_{0 \le l \le  k-1}|t^k_{l+1}-t^k_{l}|$. And, second by taking 
    \begin{align*}
        \alpha^k_t
        :=
        \frac{1}{\varepsilon_k} \int_{\left([t]^k-\varepsilon_k\right) \vee 0}^{[t]^k} \alpha_s \mathrm{d}s\;\mbox{and}\;\aleph^k_t
        :=
        \frac{1}{\varepsilon_k} \int_{\left([t]^k-\varepsilon_k\right) \vee 0}^{[t]^k} \aleph_s \mathrm{d}s\;\mbox{where}\;[t]^k:=t^k_l\;\mbox{for}\;t\in [t^k_l,t^k_{l+1})\;\mbox{with}\;[T]^k:=T.
    \end{align*}
\end{proof}

\begin{lemma} \label{lemma:law_equality}
    Let us stay in the context of the previous {\rm \Cref{lemma_piece-wiseAppr}} with $k \ge 1$.

\medskip
    There exist an $A$--valued $\left( \sigma\{X_0,W_{t \wedge \cdot}, B_{t \wedge \cdot} \} \right)_{t \in [0,T]}$--predictable  process $\alphat^k$, and an $\Er$--valued $\left( \sigma\{B_{t \wedge \cdot} \} \right)_{t \in [0,T]}$--predictable process $\widetilde{\aleph}^k$ s.t.
    \begin{align} \label{eq:dist_equality}
        \Lc\left( X^k, W^k, B^k, \Lambda^k, R^k \right)
        =
        \Lc\left( \Xt^k, W^k, B^k, \Lambdat^k, \widetilde{R}^k \right)
    \end{align}
    where $X^k$ is given in {\rm\Cref{eq:piece-wiseAppr}} with $\Lambda^k:=\delta_{\alpha^k_t}(\mathrm{d}a)\mathrm{d}t$, $R^k:=\delta_{\left(\mub^k_t,\aleph^k_t\right)}(\mathrm{d}m,\mathrm{d}e)\mathrm{d}t$, $\Xt^k_t=X_0$ for $t \in [0,\varepsilon_k]$ and $\mbox{for }t \in [\varepsilon_k,T],\;$
    \begin{align*}
        \mathrm{d}\Xt^k_t
            =
            \overline{b}\left(t,\zetab^k_t,\alepht^k_t\right)
            +
            b\left(t,\Xt^k_t,\zeta^k_t,\alphat^k_t\right)\mathrm{d}t + \sigma(t,\Xt^k_t)\mathrm{d}W^k_t + \sigma_0\mathrm{d}B^k_t,\;\zetab^k_t=\Lc(\Xt^k_t,\alphat^k_t|B),\;\zeta^k_t= \Lc( \Xt^k_t| B),\;\P\mbox{--a.e.}
    \end{align*}
    with $\Lambdat^k:=\delta_{\tilde\alpha^k_t}(\mathrm{d}a)\mathrm{d}t$ and $\Rt^k:=\delta_{\left( \zetab^k_t, \tilde\aleph^k_t \right)}(\mathrm{d}m,\mathrm{d}e)\mathrm{d}t$.
\end{lemma}

\begin{proof}
    It is essentially an application of \cite[Lemma 4.4]{djete2019general}, we just need to check the measurability of the different process $\alphat^k$ and $\alepht^k$. Indeed, if we apply \cite[Lemma 4.4]{djete2019general}, we find that the processes $\alphat^k$ and $\alepht^k$ given in the statement of our Lemma has to be $\left( \sigma\{A^k,X_0,W_{t \wedge \cdot}, B^k_{t \wedge \cdot} \} \right)_{t \in [0,T]}$--predictable with $A^k$ an $[0,1]$--valued uniform random variable independent of $(X_0,W,B^k)$. But, since for each $k$, $(\aleph^k,\mub^k)$, $X_0$ and $W$ are independent, to get the equality in distribution given in \Cref{eq:dist_equality} (similar to \cite[Equation 4.7 of Lemma 4.4]{djete2019general}), we see that the construction of $\alepht^k$ can indeed be done without needing the variables $(X_0,W)$. Therefore $\alepht^k$ can be taken $\left( \sigma\{A^k,B^k_{t \wedge \cdot} \} \right)_{t \in [0,T]}$--predictable. In addition, we know that $\varepsilon_k >0$ and $B$ is an $\F$--Brownian motion, so its increments are independent and we can take $A^k=F(B_{\varepsilon_k \wedge \cdot})$ for an appropriate Borel map $F$. Consequently, $\alepht^k$ is $\left( \sigma\{B_{t \wedge \cdot} \} \right)_{t \in [0,T]}$--predictable and $\alphat^k$ is $\left( \sigma\{X_0,W_{t \wedge \cdot},B_{t \wedge \cdot} \} \right)_{t \in [0,T]}$--predictable.
    
\end{proof}

With the processes $\alphat^k$ and $\alepht^k$ given in \Cref{lemma:law_equality}. We introduce the process $S^k$ satisfying: $S^k_0=X_0$, and  for $t \in [0,T]$,
\begin{align*}
    \mathrm{d}S^k_t
            =
            \overline{b}\left(t,\etab^k_t,\alepht^k_t\right)
            +
            b\left(t,S^k_t,\eta^k_t,\alphat^k_t\right)\mathrm{d}t + \sigma(t,S^k_t)\mathrm{d}W_t + \sigma_0\mathrm{d}B_t,\;\etab^k_t=\Lc(S^k_t,\alphat^k_t|B),\;\eta^k_t= \Lc\left( S^k_t| B \right),\;\;\P\mbox{--a.e.}
\end{align*}
    We also define $\Psi^k:=\delta_{\left( \etab^k_t, \tilde\aleph^k_t\right)}(\mathrm{d}m,\mathrm{d}e)\mathrm{d}t$, $D^k:=\delta_{\tilde\alpha^k_t}(\mathrm{d}a)\mathrm{d}t$ and $\Lambda:=\delta_{\alpha_t}(\mathrm{d}a)\mathrm{d}t$.
\begin{lemma}
    We have $\lim_{k \to \infty}\E \left[ \sup_{t \in [0,T]} |S^k_t-\Xt^k_t|^p \right]=0$. Consequently,
    \begin{align*}
        \lim_{k \to \infty} \Lc \left(S^k, W,B, D^k, \Psi^k \right)
        =
        \Lc \left(X, W,B, \Lambda, R \right)\;\mbox{in}\;\Wc_p
    \end{align*}
\end{lemma}
\begin{proof}
    The main difference between $S^k$ and $\Xt^k$ is the shifting of the Brownian motions $W$ and $B$. By seeing that $W_{\varepsilon^k \wedge \cdot}$ and $B_{\varepsilon^k \wedge \cdot}$ converge to $0$ when $k \to \infty$, this is enough to deduce the $\L^p$--convergence by using some classical arguments since the map $(b,\sigma)$ is Lipschitz in $(x,\pi,m)$. For the Wasserstein convergence, it is a combination of this $\L^p$--convergence combined with the equality in distribution in \Cref{lemma:law_equality} and the $\L^p$--convergence in \Cref{lemma_piece-wiseAppr}.
\end{proof}

\medskip
Now, let $\alpha$ be an $A$--valued $\left(\sigma\{ X_0,W_{t \wedge \cdot}, B_{t \cdot}\} \right)_{t \in [0,T]}$--predictable process and $\aleph$ be an $\Er$--valued $\left(\sigma\{ B_{t \wedge \cdot}\} \right)_{t \in [0,T]}$--predictable process. We consider $X$ satisfying: $\Lc^\P(X_0)=\nu$,
\begin{align*}
        \mathrm{d}X_t
            =
            \overline{b}\left(t,\mub_t,\aleph_t\right)
            +
            b\left(t,X_t,\mu_t,\alpha_t\right)\mathrm{d}t + \sigma(t,X_t)\mathrm{d}W_t + \sigma_0\mathrm{d}B_t,\;\mub_t=\Lc(X_t,\alpha_t|B),\;\mu_t= \Lc\left( X_t| B \right),\;\;\P\mbox{--a.e.}
    \end{align*}

\begin{lemma}
    There exists a sequence of progressively Borel measurable maps $(\aleph^j,\alpha^j)_{j \ge 1}$ s.t. for each $j \ge 1$, $[0,T] \x \Cc_{\Wc} \ni (t,\pi) \mapsto \aleph^j(t,\pi_{t \wedge \cdot}) \in \Er$ and $[0,T] \x \R \x \Cc_{\Wc} \ni (t,x,\pi) \mapsto \alpha^j(t,x,\pi_{t \wedge \cdot}) \in A$ are Lipschitz maps in $(x,\pi)$ uniformly in $t$ and, if we define $X^j$ the solution of: $\aleph^j_t:=\aleph^j(t,\mu^j)$, $\alpha^j_t=\alpha^j(t,X^j_t,\mu^j)$,
    \begin{align*}
        \mathrm{d}X^j_t
            =
            \overline{b}\left(t,\mub^j_t,\aleph^j_t\right)
            +
            b\left(t,X^j_t,\mu^j_t,\alpha^j_t\right)\mathrm{d}t + \sigma(t,X^j_t)\mathrm{d}W_t + \sigma_0\mathrm{d}B_t,\;\mub^j_t=\Lc(X^j_t,\alpha^j_t| B_{t \wedge \cdot}),\;\mu^j_t= ( X^j_t| B_{t \wedge \cdot} ),\;\;\P\mbox{--a.e.},
    \end{align*}
    then we have 
    \begin{align*}
        \lim_{j \to \infty}\left( \mu^j, \delta_{\left( \mub^j_t, \aleph^j_t \right)}(\mathrm{d}m,\mathrm{d}e) \mathrm{d}t \right)
        =
         \left( \mu, \delta_{\left( \mub_t, \aleph_t \right)}(\mathrm{d}m,\mathrm{d}e) \mathrm{d}t \right)\;\;\mbox{in }\Wc_p,\;\;\P\mbox{--a.e.}
    \end{align*}
\end{lemma}

\begin{proof}
    
    
    $\boldsymbol{{\rm Step\;1}}$  Since $\alpha$ is an $\left(\sigma\{ X_0, W_{t \wedge \cdot}, B_{t \wedge \cdot}\} \right)_{t \in [0,T]}$--predictable process and $\aleph$ is an $\left(\sigma\{ B_{t \wedge \cdot}\} \right)_{t \in [0,T]}$--predictable process, by abusing the notations, we find progressively measurable Borel maps $\alpha:[0,T] \x \R \x \Cc \x \Cc \to A$ and $\aleph:[0,T] \x \Cc \to \Er$ s.t. $\alpha_t=\alpha(t,X_0,W,B)$ and $\aleph_t=\aleph(t,B)$ $\mathrm{d}t \otimes \mathrm{d}\P$--a.e. Let $\br \in \Cc$. We set $\alpha^{\br}_t:=\alpha(t,X_0,W,\br)$ and $\aleph^{\br}_t:=\aleph(t,\br)$. Let $X^{\br}$ be the solution of
    \begin{align*}
        \mathrm{d}\left(X^{\br}_t - \sigma_0 \br(t) \right)
            =
            \overline{b}\left(t,\mub^{\br}_t,\aleph^{\br}_t\right)
            +
            b\left(t,X^{\br}_t,\mu^{\br}_t,\alpha^{\br}_t\right)\mathrm{d}t + \sigma(t,X^{\br}_t)\mathrm{d}W_t,\;\;\P\mbox{--a.e.}
    \end{align*}
    with $\;\mub^{\br}_t=\Lc(X^{\br}_t,\alpha^{\br}_t),\;\mu^{\br}_t= \Lc( X^{\br}_t).$ By uniqueness in distribution, we check that $\Lc(X,W|B=\br)=\Lc(X^{\br},W)$. 
    Similarly to the previous proof of \Cref{lemma:first_appro}, we use some shifting. We introduce the following quantities: for each $(m,\pi)$, $m[\br(t)](\mathrm{d}y,\mathrm{d}a):=\int_{\R^d} \delta_{x+\sigma_0 \br(t)}(\mathrm{d}y)m(\mathrm{d}x,\mathrm{d}a)$ and $\pi(t)[\br](\mathrm{d}y):=\int_{\R^d} \delta_{x+\sigma_0 \br(t)}(\mathrm{d}y)\pi(t)(\mathrm{d}x)$. We set $\overline{m}^{\br}_t:=\mub^{\br}_t[-\br(t)]$, $\mut^{\br}_t:=\mu^{\br}_t[-\br]$ and $Y^{\br}_t=X^{\br}_t-\sigma_0 \br(t)$ for all $t \in [0,T]$. The processes $\left(\overline{m}^{\br}_t,\mut^{\br}_t,Y^{\br}_t \right)_{t \in [0,T]}$ satisfies
    \begin{align*}
        \mathrm{d}Y^{\br}_t
            =
            \overline{b}\left(t,\overline{m}^{\br}_t[\br(t)],\aleph^{\br}_t\right)
            +
            b\left(t,Y^{\br}_t + \sigma_0 \br(t),\mut^{\br}[\br],\alpha^{\br}_t\right)\mathrm{d}t + \sigma(t,Y^{\br}_t + \sigma_0 \br(t))\mathrm{d}W_t \;\;\P\mbox{--a.e.}
    \end{align*}
    with $\overline{m}^{\br}_t(\mathrm{d}y,A)=\mut^{\br}_t$. We define the maps 
    \begin{align*}
        \overline{h}\left(t,\br,m,e \right)
        :=
        \overline{b}\left( t, m[\br(t)],e\right),\;\;h\left(t,\br,y,\pi,a \right)
        :=
        b\left(t,y+\sigma_0 \br(t), \pi[\br], a \right)\;\;\mbox{and}\;\;v(t,\br,y):=\sigma(t,y+\sigma_0 \br(t)).
    \end{align*}
    Under \Cref{assum:main1}, it is straightforward to check that 
    $$
        [0,T] \x \Cc \x \Pc_A \x \Er \x \R^d \x \Cc_{\Wc} \x A \ni (t,\br,m,e,y,\pi,a) \mapsto \left( \overline{h}\left(t,\br,m,e \right), h\left(t,\br,y,\pi,a \right), v(t,\br,y)  \right)
    $$
    is continuous in $(\br,m,e,y,\pi,a)$ for each $t$ and Lipschitz in $(m,y,\pi)$ uniformly in $(t,\br,e,a)$.
    We can rewrite $Y^{\br}$ as
    \begin{align*}
        \mathrm{d}Y^{\br}_t
            =
            \overline{h}\left(t,\br,\overline{m}^{\br}_t,\aleph^{\br}_t\right)
            +
            h\left(t,\br,Y^{\br}_t ,\mut^{\br},\alpha^{\br}_t\right)\mathrm{d}t + v(t,\br,Y^{\br}_t)\mathrm{d}W_t\;\;\P\mbox{--a.e.}
    \end{align*}
    By applying It\^o formula and taking the expectation, we find: for all $\varphi$
    \begin{align*}
        \mathrm{d}\langle \varphi, \mut^{\br}_t\rangle
        =
        \int_{\R^d \x A} \varphi'(y) \;\;\overline{h}(t,\br,\overline{m}^{\br}_t,\aleph^{\br}_t) + h \left(t,\br,y ,\mut^{\br},a\right)\;\;\overline{m}^{\br}_t(\mathrm{d}y,\mathrm{d}a) \mathrm{d}t + \frac{1}{2} \langle \varphi''(\cdot) v^2(t,\br,\cdot), \mut^{\br}_t \rangle \mathrm{d}t.
    \end{align*}
    Let us consider the Borel map $\R^d \ni x \mapsto \overline{m}^{\br}_t(x) \in \Pc(A)$ verifying $\overline{m}^{\br}_t= \overline{m}^{\br}_t(x)(\mathrm{d}a)\mut^{\br}_t(\mathrm{d}x)$. Notice that, we can choose $\overline{m}^{\br}_t(x)$ s.t. the map $[0,T] \x \R^d \x \Cc \ni(t,x,\br) \mapsto \overline{m}^{\br}_t(x) \in \Pc(A)$ is Borel measurable . 
    Since $\Gr(x)\mathrm{d}x$ is non--atomic, there exists a sequence of progressively Borel maps $\left( \alphat^{l} \right)_{l \ge 1}$ s.t. $[0,T] \x \R^d \x \Cc \ni (t,x,\br) \mapsto \alphat^{l} (t,x,\br) \in A$ and 
    \begin{align} \label{eq:converg_control}
        \lim_{l \to \infty} \delta_{ \tilde\alpha^{l} \left(t,x, \br \right)}(\mathrm{d}a){\rm G}(x)\mathrm{d}x=\overline{m}^{\br}_t(x)(\mathrm{d}a){\rm G}(x)\mathrm{d}x\;\mbox{in}\;\Wc_p
    \end{align}
    $\;\mbox{where }{\rm G}\mbox{ is a continuous density}$ $\Gr(x):=\frac{1}{1+|x|^p} \left(\int_{\R} \frac{1}{1+|y|^p} \mathrm{d}y \right)^{-1}.$
    Now, let $Y^{\br,l}$ be the solution of
    \begin{align*}
        \mathrm{d}Y^{\br,l}_t
            =
            \overline{h}\left(t,\br,\overline{m}^{\br,l}_t,\aleph^{\br}_t\right)
            +
            h\left(t,\br,Y^{\br,l}_t ,\mut^{\br,l}_t,\alphat^{l}(t,Y^{\br,l}_t,\br)\right)\mathrm{d}t + v(t,\br,Y^{\br,l}_t)\mathrm{d}W_t\;\;\P\mbox{--a.e.}
    \end{align*}
    with $\overline{m}^{\br,l}_t:=\Lc\left(Y^{\br,l}_t, \alphat^{l}(t,Y^{\br,l}_t,\br) \right)$ and $\mut^{\br,l}_t:=\Lc(Y^{\br,l}_t )$. Combined \eqref{eq:converg_control} and (an easy extension of) \cite[Proposition A.4.]{closed-loop-MFG_MDF}, we have
    \begin{align*}
        \lim_{l \to \infty}\left( \mut^{\br,l}, \delta_{ \overline{m}^{\br,l}_t}(\mathrm{d}m)\mathrm{d}t \right)
        =
        \left( \mut^{\br}, \delta_{ \overline{m}^{\br}_t}(\mathrm{d}m)\mathrm{d}t \right)\;\mbox{in}\;\Wc_p.
    \end{align*}
    We can also check that $ \lim_{l \to \infty}\left( \mut^{\br,l}[\br], \delta_{ \overline{m}^{\br,l}_t[\br]}(\mathrm{d}m)\mathrm{d}t \right)
        =
        \left( \mut^{\br}[\br], \delta_{ \overline{m}^{\br}_t[\br]}(\mathrm{d}m)\mathrm{d}t \right)$. Now, let us consider $X^{l}$ the process satisfying
        \begin{align*}
            \mathrm{d}X^{l}_t
            =
            \overline{b}\left(t,\mub^{l}_t,\aleph_t\right)
            +
            b\left(t,X^{l}_t,\mu^{l}_t,\alpha^{l}\left(t, X^{l}_t, B \right)\right)\mathrm{d}t + \sigma(t,X^{l}_t)\mathrm{d}W_t + \sigma_0\mathrm{d}B_t,\;\;\P\mbox{--a.e.}
        \end{align*}
        with        $\mub^{l}_t=\Lc\left(X^{l}_t,\alpha^{l}\left(t, X^{l}_t , B \right)|B_{t \wedge \cdot} \right),\;\mu^{l}_t= \left( X^{l}_t| B_{t \wedge \cdot} \right)$ and $\alpha^{l}(t,x,\br)=\alphat^{l}(t,x-\sigma_0 \br(t),\br)$. By uniqueness, we can check that $\Lc(Y^{\br,l}_\cdot,W_\cdot)=\Lc(X^{l}_\cdot - \sigma_0 B_\cdot,W_\cdot|B=\br)$. We can deduce that
        \begin{align*}
            \lim_{l \to \infty}\left( \mu^{l}, \delta_{ \mub^{l}_t}(\mathrm{d}m)\mathrm{d}t \right)
        =
        \left( \mu, \delta_{ \mub_t}(\mathrm{d}m)\mathrm{d}t \right)\;\mbox{in}\;\Wc_p\;\P\mbox{--a.e.}
        \end{align*}
        By using \cite[Proposition A.7.]{closed-loop-MFG_MDF}, up to another approximation, we can consider that the map $[0,T] \x \R^d \x \Cc \ni(t,x,b) \mapsto \left(\alpha^l(t,x,b),\aleph(t,b) \right) \in A \x \Er$  are Lipschitz in $(x,b)$ uniformly in $t$. 

        \medskip
        $\boldsymbol{{\rm Step\;2}}$ Let $l \ge 1$. We know that $\sigma_0$ is invertible. Then, we can use \cite[Proposition A.10.]{closed-loop-MFG_MDF} and find a sequence of Borel progressively measurable maps 
        $$
            \left( [0,T] \x \R \x \Cc_{\Wc} \ni (t,x,\pi) \mapsto \left( \alpha^{l,j}(t,x,\pi), \aleph^{j}(t,\pi), \phi^j(t,\pi) \right) \in A \x \Er \x \R^d  \right)_{ j \ge 1}
        $$
        Lipschitz in $(x,\pi)$ uniformly in $t$ s.t. $\left( \alpha^{l,j}(t,x,\pi), \aleph^{j}(t,\pi) \right)=\left( \alpha^{l}(t,x,\phi^j(t,\pi)), \aleph(t,\phi^j(t,\pi)) \right)$ and if we let  $X^{l,j}$ be the process satisfying
        \begin{align*}
            \mathrm{d}X^{l,j}_t
            =
            \overline{b}\left(t,\mub^{l,j}_t,\aleph^{l}(t,\mu^{l,j}_t)\right)
            +
            b\left(t,X^{l,j}_t,\mu^{l,j}_t,\alpha^{l,j}\left(t, X^{l,j}_t, \mu^{l,j} \right)\right)\mathrm{d}t + \sigma(t,X^{l,j}_t)\mathrm{d}W_t + \sigma_0\mathrm{d}B_t,\;\;\P\mbox{--a.e.}
        \end{align*}
        with        $\mub^{l,j}_t=\Lc\left(X^{l,j}_t,\alpha^{l,j}\left(t, X^{l,j}_t, \mu^{l,j} \right)|B_{t \wedge \cdot} \right),\;\mu^{l,j}_t= \left( X^{l,j}_t| B_{t \wedge \cdot} \right)$, we have $\P$--a.e. for all $t \in [0,T]$ and each $l,j \ge 1$, $B_t=\phi^k(t,\mu^{l,j})$,
        \begin{align*}
            \lim_{j \to \infty}\E \left[ \int_0^T \| \mu^{l,j}_t-\mu^{l}_t\|_{\rm TV}  \mathrm{d}t\right]=0\;\;\mbox{and}\;\;\lim_{j \to \infty}  \delta_{\left(\mub^{l,j}_t, \aleph^{l}(t,\mu^{l,j}_t) \right)}(\mathrm{d}m,\mathrm{d}e)\mathrm{d}t = \delta_{\left(\mub^{l}_t, \aleph^l_t \right)}(\mathrm{d}m,\mathrm{d}e)\mathrm{d}t \;\mbox{in}\;\Wc_p,\;\P\mbox{--a.e.}
        \end{align*}
    where $\|\|_{\rm TV}$ is the total variation distance. 
    By combining the previous 2 steps, we can deduce the proof of the Lemma.
        
\end{proof}

\medskip
Let us consider the random variables $(\mu,R)$ and the filtration $\G$ given in the (preamble of )\Cref{lemma:first_appro}.   By combining all the previous results, we are able to prove that: 
\begin{proposition} \label{prop:converg}
    There exist:
    \begin{itemize}
        \item a sequence  $(\mu^{k}, \mub^{k}, \aleph^{k})_{k \ge 1}$ s.t. $\mu^{k}$ is an $\Pc(\R^d)$--valued $\G$--adapted continuous process, $\mub^{k}$ is an $\Pc_A$--valued $\G$--predictable process and $\alepht^{k}$ is an $\Er$--valued $\G$--predictable process;

        \item a sequence $(\zeta^{k}, \zetab^{k}, \alepht^{k})_{k \ge 1}$ s.t. $\zeta^{k}$ is an $\Pc(\R^d)$--valued $(\sigma\{B_{t \wedge \cdot}\})_{t \in [0,T]}$--adapted continuous process, $\zetab^{k}$ is an $\Pc_A$--valued $(\sigma\{B_{t \wedge \cdot}\})_{t \in [0,T]}$--predictable process and $\aleph^{k}$ is an $\Er$--valued $(\sigma\{B_{t \wedge \cdot}\})_{t \in [0,T]}$--predictable process;

        \item a sequence of Borel progressively measurable maps $(\aleph^{j,k},\alpha^{j,k,\ell})$ s.t.  $[0,T] \x \Cc_{\Wc} \ni (t,\pi) \to \aleph^{j,k}(t,\pi_{t \wedge \cdot}) \in \Er$ and $[0,T] \x \R \x \Cc_{\Wc} \ni (t,x,\pi) \to \alpha^{j,k}(t,x,\pi_{t \wedge \cdot}) \in A$ are Lipschitz maps in $(x,\pi)$ uniformly in $t$ and, if we define $X^{j,k}:=X$ the solution of: $\aleph^{j,k}_t:=\aleph^{j,k}(t,\mu^{j,k})$, $\alpha^{j,k}_t=\alpha^{j,k}(t,X^{j,k}_t,\mu^{j,k})$, $\mub^{j,k}_t=\Lc(X_t,\alpha^{j,k}_t| B_{t \wedge \cdot}),\;\mu^{j,k}_t= ( X_t| B_{t \wedge \cdot} )$ and
    \begin{align*}
        \mathrm{d}X_t
            =
            \overline{b}\left(t,\mub^{j,k}_t,\aleph^{j,k}_t\right)
            +
            b\left(t,X_t,\mu^{j,k}_t,\alpha^{j,k}_t\right)\mathrm{d}t + \sigma(t,X_t)\mathrm{d}W_t + \sigma_0\mathrm{d}B_t,\;\;\;\P\mbox{--a.e.},
    \end{align*}
    \end{itemize}
    all these sequences satisfy: first, 
    \begin{align*}
        \lim_{k \to \infty}\lim_{j \to \infty} \Wc_p \left( \left( \mu^{j,k}, \delta_{\left( \mub^{j,k}_t, \aleph^{j,k}_t \right)}(\mathrm{d}m,\mathrm{d}e) \mathrm{d}t \right)
        ,
         \left( \zeta^{k}, \delta_{\left( \zetab^{k}_t, \tilde\aleph^{k}_t \right)}(\mathrm{d}m,\mathrm{d}e) \mathrm{d}t \right) \right)=0,\;\;\P\mbox{--a.e. }, 
    \end{align*}
    second, for each $k \ge 1$,
    \begin{align*}
        \Lc^\P \left( \zeta^{k}, \delta_{\left( \zetab^{k}_t, \tilde\aleph^{k}_t \right)}(\mathrm{d}m,\mathrm{d}e) \mathrm{d}t, B^k,W^k,X_0 \right)
        =
        \Lc^\P \left( \mu^{k}, \delta_{\left( \mu^{k}_t, \aleph^{k}_t \right)}(\mathrm{d}m,\mathrm{d}e) \mathrm{d}t, B^k,W^k,X_0 \right)
    \end{align*}
    and finally
    \begin{align*}
         \lim_{k \to \infty} \Wc_p \left(\left( \mu^{k}, \delta_{\left( \mu^{k}_t, \aleph^{k}_t \right)}(\mathrm{d}m,\mathrm{d}e) \mathrm{d}t \right)
         ,
         \left( \mu, R_t(\mathrm{d}m,\mathrm{d}e) \mathrm{d}t \right) \right)=0,\;\;\P\mbox{--a.e. }
    \end{align*}

\end{proposition}

\medskip
All the previous results have been a preparation to provide the following Proposition which states the approximation of MFG solution by a sequence of approximate strong MFG solutions.
\begin{proposition} \label{prop:approximation_by_strong}
    Let $\varepsilon \ge 0$ and $\Pr \in \Pcb_{\rm \varepsilon\mbox{-}mfg}$ associated to $(\Phi,\Zf)$. 
    There exist a sequence of non--negative numbers $(\varepsilon_\ell)_{\ell \ge 1}$ and a sequence $(\phi^\ell,\xi^\ell,\aleph^\ell,\alpha^\ell,\mub^\ell)_{\ell \ge 1}$ s.t. for each $\ell \ge 1$, $(\alpha^\ell,\xi^\ell,\aleph^\ell) \in \Cc_s$, $(\alpha^\ell,\mub^\ell) \in {\rm MFG}_S\left[\Cf^\ell,\varepsilon_\ell \right]$ where $\Cf^\ell:=(\phi^\ell,\xi^\ell,\aleph^\ell)$, 
    \begin{align*}
        \limsup_{\ell} \varepsilon_\ell \in [0,\varepsilon]\;\;\mbox{and}\;\;\lim_{\ell \to \infty}\Lc^\P \left( \mu^\ell, \delta_{\left( \mub^\ell_t, \aleph^\ell_t \right)}(\mathrm{d}m,\mathrm{d}e) \mathrm{d}t \right)
        =
        \Lc^{\Pr} \left( \mu, R \right)\;\;\mbox{in }\Wc_p.
    \end{align*}
    Besides,
    \begin{align*}
        \lim_{\ell \to \infty} J_{A,\mub^\ell}^{\Cf^\ell}(\alpha^\ell)
        =
        \E^{\mathrm{P}} \big[{\rm J}^{\Phi,\Zf}_A\left(\Gamma',\mu',R,\mu \right) \big]\;\;\mbox{and}\;\;\lim_{\ell \to \infty} J_{\Pr}^{\alpha^\ell,\mub^\ell}(\Cf^\ell)
        =
        \E^{{\Pr}}\big[{\rm J}^{R,\mu}_{\Pr}(\Phi,\Zf) \big].
    \end{align*}
\end{proposition}

\begin{proof} 
For simplification, we can assume that $\Pr=\P \circ \left( X,W,B,R,\mu \right)^{-1}$ with the random variables $(X,W,B,R,\mu)$ defined on $(\Om,\H,\P)$. Therefore, we will only use the probability space $(\Om,\H,\P)$.

\medskip
$\boldsymbol{{\rm Step\;1}}$
    First, by combining the approximation of measurable function by continuous functions in \cite[Proposition C.1]{Lacker_carmona_delarue_CN} and the approximation of continuous function by Lipschitz functions in \cite[Theorem 6.4.1]{Cobza2019}, there exists a sequence of Lipschitz functions $\left(\xi^\ell: \M \to \Er \right)_{\ell \ge 1}$ and $\left(\phi^\ell: \R^d \to \Ir \right)_{\ell \ge 1}$ s.t.
    \begin{align*}
        \lim_{\ell \to \infty} \delta_{\xi^\ell(\rr)}\Lc^{\P}(R)(\mathrm{d}r)
        =\Zf(\rr)(\mathrm{d}e)\Lc^{\P}(R)(\mathrm{d}r)\;\mbox{and}\;\lim_{\ell \to \infty} \delta_{\phi^\ell(x)}(\mathrm{d}i)\Gr(x)\mathrm{d}x=\Phi(x)(\mathrm{d}i)\Gr(x)\mathrm{d}x.
    \end{align*}
    
    We are then using \Cref{prop:converg}, we consider the sequence $(\alpha^{j,k}, \mub^{j,k}, \aleph^{j,k})_{j,k \ge 1}$ given in the Proposition and we have 
    \begin{align*}
        \lim_{k \to \infty}\lim_{j \to \infty} \Lc^\P \left( \mu^{j,k}, \delta_{\left( \mub^{j,k}_t, \aleph^{j,k}_t \right)}(\mathrm{d}m,\mathrm{d}e) \mathrm{d}t \right)
        =
        \Lc^\P \left( \mu, R \right)\;\;\mbox{in }\Wc_p.
    \end{align*} 
    We can check that: in $\Wc_p$,
    \begin{align*}
        \lim_{\ell \to \infty} \lim_{k \to \infty} \lim_{j \to \infty} &\Lc^\P \left(  \xi^\ell(\mu^{j,k}), \delta_{\left( \mub^{j,k}_t,\; \aleph^{j,k}_t \right)}(\mathrm{d}m,\mathrm{d}e) \mathrm{d}t \right)
        =
        \Lc^\P \left( \xi, R \right)
    \end{align*}
and with the help of \Cref{prop:charac-convergence},
    \begin{align*}
        \lim_{\ell \to \infty} \lim_{k \to \infty} \lim_{j \to \infty} &\Lc^\P \left( \mu^{j,k},\; \delta_{\phi^\ell(x)}(\mathrm{d}i)\mu^{j,k}_T(\mathrm{d}x),\; \delta_{\left( \mub^{j,k}_t,\; \aleph^{j,k}_t \right)}(\mathrm{d}m,\mathrm{d}e) \mathrm{d}t \right)
        =
        \Lc^\P \left( \mu,\; \Phi(x)(\mathrm{d}i)\mu_T(\mathrm{d}x),\; R \right).
    \end{align*}
    By setting $\Cf^{j,k,\ell}:=(\phi^\ell,\xi^{\ell},\aleph^{j,k})$, this allows us to verify that
    \begin{align*}
        \lim_{\ell \to \infty} \lim_{k \to \infty} \lim_{j \to \infty} J_{A,\mub^{j,k}}^{\Cf^{j,k,\ell}}(\alpha^{j,k})
        =
        \E^{\mathrm{P}} \big[{\rm J}^{\Phi,\Zf}_A\left(\Gamma',\mu',R,\mu \right) \big]\;\;\mbox{and}\;\;\lim_{\ell \to \infty} \lim_{k \to \infty} \lim_{j \to \infty} J_{\Pr}^{\alpha^{j,k},\mub^{j,k}}(\Cf^{j,k,\ell})
        =
        \E^{{\Pr}}\big[{\rm J}^{R,\mu}_{\Pr}(\Phi,\Zf) \big].
    \end{align*}

    \medskip
    To complete the proof, 
    we need to check that $(\alpha^{j,k},\mub^{j,k}) \in {\rm MFG}_S[\Cf^{j,k,\ell},\varepsilon_{j,k,\ell}]$ for some sequence of non--negative numbers $(\varepsilon_{j,k,\ell})_{j,k,\ell \ge 1}$ verifying $\limsup_{\ell \to \infty}\limsup_{k \to \infty} \limsup_{j \to \infty}\varepsilon_{j,k,\ell} \in [0,\varepsilon]$.

    \medskip
    $\boldsymbol{{\rm Step\;2}}$ Let us recall that  $(\varepsilon^k,W^k,B^k)_{k \ge 1}$ is given in \Cref{lemma_piece-wiseAppr} and $(\mu^k,\mub^k,\aleph^k,\zeta^k,\zetab^k,\alepht^k)_{k \ge 1}$ is given in \Cref{prop:converg}. We set $\Rt^k:=\delta_{(\zetab^k_t,\tilde \aleph^k_t)}(\mathrm{d}m,\mathrm{d}e)\mathrm{d}t$ and $R^k:=\delta_{(\mub^k_t,\aleph^k_t)}(\mathrm{d}m,\mathrm{d}e)\mathrm{d}t$.  
    By using the equality in distribution given in \Cref{prop:converg}, it is straightforward that for each $\beta \in \Ac$,
    \begin{align*}
        \widetilde{\Qr}^{\beta,k}:=\Lc\left(\Xt^{\beta,k}, W^k, B^k, \Rt^k, \zeta^k \right)
        =
        \Lc\left(\Xh^{\beta,k}, W^k, B^k, R^k,\mu^k \right)=:\widehat{\Qr}^{\beta,k}
    \end{align*}
    with $(\Xt^{\beta,k},\Xh^{\beta,k})$ satisfying: $\Xt^{\beta,k}_t=\Xh^{\beta,k}_t=X_0$ for $t \in [0,\varepsilon^k]$, for $t \in [\varepsilon^k,T]$,
    \begin{align*}
        \mathrm{d}\Xt^{\beta,k}_t
            =
            \overline{b}\left(t,\zetab^k_t,\alepht^k_t\right)
            +
            b\left(t,\Xt^{\beta,k}_t,\zeta^k_t,\beta\left(t, \Xt^{\beta,k}_t, \Rt^k \right)\right)\mathrm{d}t + \sigma(t,\Xt^{\beta,k}_t)\mathrm{d}W^k_t + \sigma_0\mathrm{d}B^k_t,\;\;\P\mbox{--a.e.},
    \end{align*}
    and 
    \begin{align*}
        \mathrm{d}\Xh^{\beta,k}_t
            =
            \overline{b}\left(t,\mub^k_t,\aleph^k_t\right)
            +
            b\left(t,\Xh^{\beta,k}_t,\mu^k_t,\beta\left(t, \Xh^{\beta,k}_t, R^k \right)\right)\mathrm{d}t + \sigma(t,\Xh^{\beta,k}_t)\mathrm{d}W^k_t + \sigma_0\mathrm{d}B^k_t,\;\;\P\mbox{--a.e.}
    \end{align*}
    Consequently, if we consider $\Phi^\ell(\cdot):=\delta_{\phi^\ell(\cdot)}(\mathrm{d}i)$ and $\Zf^\ell(\cdot):=\delta_{\xi^\ell(\cdot)}(\mathrm{d}e)$, we find 
    \begin{align} \label{eq:first_equality}
        \sup_{\beta \in \Ac}\E^{\widetilde{\Qr}^{\beta,k}} \big[{\rm J}^{\Phi^\ell,\Zf^\ell}_A(\Gamma',\mu',R,\mu) \big]
        =
        \sup_{\beta \in \Ac}\E^{\widehat{\Qr}^{\beta,k}} \big[{\rm J}^{\Phi^\ell,\Zf^\ell}_A(\Gamma',\mu',R,\mu) \big].
    \end{align}
    Notice that, the canonical variable $\Gamma'$ is not well defined in the previous equality since $\widehat{\Qr}^{\beta,k}$ and $\widetilde{\Qr}^{\beta,k}$ do not belong to $\Pcb$. Indeed, the processes $\Xt^{\beta,k}$ and $\Xh^{\beta,k}$ are frozen on $[0,\varepsilon^k]$. We just froze $\Gamma'$ on $[0,\varepsilon^k] \x \R^d \x A$ and abuse the notation by considering that the distributions belong to $\Pcb$.

\medskip
    By using the almost surely convergence in \Cref{prop:converg} and the regularity of the coefficients, we can check that 
    \begin{align*}
        \lim_{k \to \infty}\left|\sup_{\beta \in \Ac}\E^{\widehat{\Qr}^{\beta,k}} \big[{\rm J}^{\Phi^\ell,\Zf^\ell}_A(\Gamma',\mu',R,\mu) \big]
        -
        \sup_{\beta \in \Ac}\E^{\widehat{\Pr}^{\beta,k}} \big[{\rm J}^{\Phi^\ell,\Zf^\ell}_A(\Gamma',\mu',R,\mu) \big] \right|=0
    \end{align*}
    where 
    \begin{align*}
        \widehat{\Pr}^{\beta,k}:=\Lc\left(X^{\beta,k},  W, B, R, \mu \right)
    \end{align*}
    with $X^{\beta,k}$ verifying
    \begin{align*}
        \mathrm{d}X^{\beta,k}_t
            =
            \int_{\Pc_A \x \Er}\overline{b}\left(t,m,e\right) R_t(\mathrm{d}m,\mathrm{d}e)
            +
            b\left(t,X^{\beta,k}_t,\mu,\beta\left(t, \Xh^{\beta,k}_t, R^k \right)\right)\mathrm{d}t + \sigma(t,X^{\beta,k}_t)\mathrm{d}W_t + \sigma_0\mathrm{d}B_t,\;\;\P\mbox{--a.e.}
    \end{align*}
    Since $(R^k_t)_{t \in [0,T]}$ is an $\G$--predictable process for each $k$, let us point out the fact that $(R^k_{t \wedge \cdot},B_{t \wedge \cdot})$ is a Borel map of $R_{t \wedge \cdot}$ (see also \Cref{rm:measurable_B}). Then, a projection argument like in \Cref{lemma:projection}  allows us to find a progressively Borel map $\Lambda^{\beta,k}:[0,T] \x \R^d \x \M \to \Pc(A)$ s.t. $\Lc(X^{\beta,k}_t|R_{t \wedge \cdot})=\Lc(X^{\beta,k}_t|R)=\Lc(Z^{\beta,k}_t|R)$, $\P$--a.e. with
    \begin{align*}
        \mathrm{d}Z^{\beta,k}_t
            =
            \int_{\Pc_A \x \Er}\overline{b}\left(t,m,e\right) R_t(\mathrm{d}m,\mathrm{d}e)
            +
            \int_A b\left(t,Z^{\beta,k}_t,\mu,a\right) \Lambda^{\beta,k}\left(t, Z^{\beta}_t, R \right)(\mathrm{d}a)\;\;\mathrm{d}t + \sigma(t,Z^{\beta,k}_t)\mathrm{d}W_t + \sigma_0\mathrm{d}B_t,\;\;\P\mbox{--a.e.}
    \end{align*}
    This leads to say that $\widehat{\Pr}^{\beta,k} \in \Pcb$ with $\widehat{\Pr}^{\beta,k} \circ (B,R)^{-1}=\Pr \circ (B,R)^{-1}$. As a result, we obtain that:,
    \begin{align} \label{eq:second_equality}
        &\limsup_{\ell \to \infty}\limsup_{k \to \infty}\sup_{\beta \in \Ac}\E^{\widehat{\Pr}^{\beta,k}} \big[{\rm J}^{\Phi^\ell,\Zf^\ell}_A(\Gamma',\mu',R,\mu) \big] \nonumber
        \\
        &=\limsup_{\ell \to \infty}\limsup_{k \to \infty}\bigg(\sup_{\beta \in \Ac}\E^{\widehat{\Pr}^{\beta,k}} \left[\int_0^T  \langle \overline{L}\left(t,\cdot,\cdot \right),R_t \rangle
        +
        \langle L\left(t, \cdot,\mu,\cdot \right), \Gamma_t' \rangle\;\;\mathrm{d}t +
         \int_{\R^d \x \Ir} \overline{g}(x,i) \Phi^\ell(x)(\mathrm{d}i)\mu_T'(\mathrm{d}x)
         \right] \nonumber
         \\
         &~~~~~~~~~~~~~~~~~~~~~~~~~~~~~~~~+ \E^{\P} \left[\int_{\Er}g(\mu,e)\mathfrak{Z}^\ell(R)(\mathrm{d}e)
         \right] \bigg) \nonumber
         \\
        &\le \limsup_{k \to \infty} \sup_{\beta \in \Ac}\E^{\widehat{\Pr}^{\beta,k}} \big[{\rm J}^{\Phi,\Zf}_A(\Gamma',\mu',R,\mu) \big]~\le 
        \E^{\Pr} \big[{\rm J}^{\Phi,\Zf}_A(\Gamma',\mu',R,\mu) \big] + \varepsilon.
    \end{align}

    Next, for any $\beta \in \Ac$, given $M^{j,k}:=\delta_{(\mub^{j,k}_t,\aleph^{j,k}_t)}(\mathrm{d}m,\mathrm{d}e)\mathrm{d}t$, we consider $S^{\beta,j,k}$ satisfying: $S^{\beta,j,k}=X_0$,
    \begin{align*}
        \mathrm{d}S^{\beta,j,k}_t
            =
            \overline{b}\left(t,\mub^{j,k}_t,\aleph^{j,k}_t\right)
            +
            b\left(t,S^{\beta,j,k}_t,\mu^{j,k},\beta\left(t, S^{\beta,j,k}_t, M^{j,k} \right)\right)\mathrm{d}t + \sigma(t,S^{\beta,j,k}_t)\mathrm{d}W_t + \sigma_0\mathrm{d}B_t,\;\;\P\mbox{--a.e.}
    \end{align*}
    Let us now considering $\St^{\beta,j,k}$ verifying: $\St^{\beta,j,k}_t=X_0$ for $t \in [0 , \varepsilon_k]$, and for $t \in [\varepsilon_k,T]$,
    \begin{align*}
        \mathrm{d}\St^{\beta,j,k}_t
            =
            \overline{b}\left(t,\zetab^{k}_t,\alepht^{k}_t\right)
            +
            b\left(t,\St^{\beta,j,k}_t,\zeta^{k},\beta\left(t, S^{\beta,j,k}_t, M^{j,k} \right)\right)\mathrm{d}t + \sigma(t,\St^{\beta,j,k}_t)\mathrm{d}W^k_t + \sigma_0\mathrm{d}B^k_t,\;\;\P\mbox{--a.e.}
    \end{align*} 
    By using the almost surely convergence in \Cref{prop:converg}, it is easy to verify that 
    $$
        \lim_{k \to \infty} \lim_{j \to \infty}\E\left[ \sup_{t \in [0,T]} |\St^{\beta,j,k}_t-S^{\beta,j,k}_t|^p\right]=0.
    $$
    Notice that $(M^{j,k}_t)_{t  \in [0,T]}$ is $\left(\sigma\{ B_{t \wedge \cdot}\} \right)_{t \in [0,T]}$--predictable. Therefore, we can find $\betat^{j,k}$ s.t. $\beta\left(t, S^{\beta,j,k}_t, M^{j,k} \right)=\betat^{j,k}\left(t, S^{\beta,j,k}_t, B \right)$. Also, by similar techniques mentioned \Cref{rm:measurable_B}, we can see that $B^k_{t \wedge \cdot}$ is a Borel map of $\Rt^k_{t \wedge \cdot}$. Thanks to some projection arguments (see \Cref{lemma:projection}), we can find a progressively Borel map $\Lambdat^{\beta,j,k}:[0,T] \x \R^d \x \M \to \Pc(A)$ s.t. we get that $\Lc(\St^{\beta,j,k}_t|B^k_{t \wedge \cdot})=\Lc(\St^{\beta,j,k}_t|B^k)=\Lc(\Zt^{\beta,j,k}_t|\Rt^k)$ where: $\Zt^{\beta,j,k}_t=X_0$ for $t \in [0,\varepsilon_k]$, and for $t \in [\varepsilon_k,T],$
    \begin{align*}
        \mathrm{d}\Zt^{\beta,j,k}_t
            =
            \overline{b}\left(t,\zetab^{k}_t,\alepht^{k}_t\right)
            +
            \int_A b\left(t,\Zt^{\beta,j,k}_t,\zeta^{k},a\right) \Lambdat^{\beta,j,k}\left(t, \Zt^{\beta,j,k}_t, \Rt^k \right)(\mathrm{d}a)\;\;\mathrm{d}t + \sigma(t,\Zt^{\beta,j,k}_t)\mathrm{d}W^k_t + \sigma_0\mathrm{d}B^k_t,\;\;\P\mbox{--a.e.}
    \end{align*}
    Consequently, by combining the previous result, if we set 
    \begin{align*}
        \widetilde{\Pr}^{\beta,j,k}:=\Lc\left(\Zt^{\beta,j,k}, W, B, \Rt^{k}, \zeta^{k} \right),
    \end{align*}
    by \Cref{eq:first_equality} and \Cref{eq:second_equality} we find that,
    \begin{align*}
        \limsup_{\ell \to \infty}\limsup_{k \to \infty}\limsup_{j \to \infty}\sup_{\beta \in \Ac}J^{\Cf^{j,k,\ell}}_{A,\mub^{j,k}}(\beta)
        &=
        \limsup_{\ell \to \infty}\limsup_{k \to \infty} \limsup_{j \to \infty} \sup_{\beta \in \Ac}\E^{\widetilde{\Pr}^{\beta,j,k}} \big[{\rm J}^{\Phi^\ell,\Zf^\ell}_A(\Gamma',\mu',R,\mu) \big]
        \\
        &= \limsup_{\ell \to \infty}\limsup_{k \to \infty} \sup_{\beta \in \Ac}\E^{\widetilde{\Qr}^{\beta,k}} \big[{\rm J}^{\Phi^\ell,\Zf^\ell}_A(\Gamma',\mu',R,\mu) \big]
        \\
        &=
        \limsup_{\ell \to \infty}\limsup_{k \to \infty} \sup_{\beta \in \Ac}\E^{\widehat{\Qr}^{\beta,k}} \big[{\rm J}^{\Phi^\ell,\Zf^\ell}_A(\Gamma',\mu',R,\mu) \big]
        \\
        &=
        \limsup_{\ell \to \infty}\limsup_{k \to \infty}\sup_{\beta \in \Ac}\E^{\widehat{\Pr}^{\beta,k}} \big[{\rm J}^{\Phi^\ell,\Zf^\ell}_A(\Gamma',\mu',R,\mu) \big]
        \le \E^{\Pr} \big[{\rm J}^{\Phi,\Zf}_A(\Gamma',\mu',R,\mu) \big] + \varepsilon.
    \end{align*}

    \medskip
    $\boldsymbol{{\rm Step\;3}}$
    For each $j,k, \ell \ge 1$, we set
    \begin{align*}
        \varepsilon_{j,k,\ell}
        :=
        \sup_{\beta \in \Ac} J^{\Cf^{j,k,\ell}}_{A,\mub^{j,k}}(\beta)
        -
        J^{\Cf^{j,k,\ell}}_{A,\mub^{j,k}}(\alpha^{j,k}).
    \end{align*}
    By definition, we have $\varepsilon_{j,k,\ell} \ge 0$ for each $j,k, \ell \ge 1$. Also, by using the previous parts, we obtain
    \begin{align*}
        \limsup_{\ell \to \infty}\limsup_{k \to \infty}\limsup_{j \to \infty}\varepsilon_{j,k,\ell}
        &=
        \limsup_{\ell \to \infty}\limsup_{k \to \infty}\limsup_{j \to \infty}\sup_{\beta \in \Ac} J^{\Cf^{j,k,\ell}}_{A,\mub^{j,k}}(\beta)
        -
        \limsup_{\ell \to \infty}\limsup_{k \to \infty}\limsup_{j \to \infty}J^{\Cf^{j,k,\ell}}_{A,\mub^{j,k}}(\alpha^{j,k})
        \\
        & =\limsup_{\ell \to \infty}\limsup_{k \to \infty}\limsup_{j \to \infty}\sup_{\beta \in \Ac} J^{\Cf^{j,k,\ell}}_{A,\mub^{j,k}}(\beta)
        -
        \E^{\Pr} \big[{\rm J}^{\Phi,\Zf}_A(\Gamma',\mu',R,\mu) \big]
        \le \varepsilon.
    \end{align*}
    This is enough to conclude the result.

\end{proof}

\medskip Let $\varepsilon \ge 0$, $\Cf=(\phi,\xi,\aleph)$ be a contract and $(\alpha,\mub) \in {\rm MFG}_S[\Cf,\varepsilon]$. We define the contract $\Cf^n:=(\phi,\xi^n,\aleph^n)$ and the  maps $\alphab^n:=\left( \alpha^{1,n},\cdots, \alpha^{n,n} \right) \in (\Ac_n)^n$ by
\begin{align*}
    \xi^{n}(\boldsymbol{x})
    :=
    \xi\left(R[\pi^n] \right),\;\;\aleph^{n}(t,\boldsymbol{x})
    :=
    \aleph(t,\pi^n)\;\;\mbox{and}\;\;\alpha^{i,n}(t,\boldsymbol{x})
    :=
    \alpha(t,x^i(t),\pi^n),\;1 \le i \le n,\;\mbox{with}\;\boldsymbol{x}:=\left(x^1,\cdots,x^n \right) \in (\Cc)^n,
\end{align*}
$\;\pi^n(t):=\frac{1}{n}\sum_{i=1}^n \delta_{x^i(t)}$, $\overline{\pi}^n(t):=\frac{1}{n}\sum_{i=1}^n \delta_{(x^i(t),\alpha^{i,n}(t,\boldsymbol{x}))}$ and $R[\pi^n]:= \delta_{ \left( \overline{\pi}(t), \aleph^n(t,\boldsymbol{x}) \right)}(\mathrm{d}m,\mathrm{d}e)\mathrm{d}t$. We now give the approximation by approximate Nash equilibria.

\begin{proposition} \label{prop:conv_n-player-strong}
    There exists a sequence of non--negative numbers $(\varepsilon_n)_{n \ge 1}$ s.t. $\limsup_{n} \varepsilon_n \in [0,\varepsilon]$ and for each $n \ge 1$, $\alphab^n$ belongs to ${\rm NE}[\Cf^n,\varepsilon_n]$, 
    \begin{align*}
        \lim_{n \to \infty}\Lc^\P \left( \varphi^n[\alphab^n], R[ \varphi^n[\alphab^n] ]
        \right)
        =
        \Lc^\P \left( \mu, \delta_{\left( \mub_t, \aleph_t \right)}(\mathrm{d}m,\mathrm{d}e) \mathrm{d}t \right)\;\;\mbox{in }\Wc_p
    \end{align*}
    and
    \begin{align*}
        \lim_{n \to \infty}\frac{1}{n}\sum_{i=1}^{n}J^{\Cf^n}_{n,i}(\alphab^{n})=
            J^{\Cf}_{A,\mub}({\alpha})\;\;\mbox{and}\;\;\lim_{n \to \infty} J_{n,\Pr}^{{\color{black}\alphab^n}}(\Cf^n)=J_{\Pr}^{\alpha,\mub}(\Cf).
    \end{align*}
\end{proposition}

\begin{proof}
    This is just an application of \cite[Proposition 3.15]{closed-loop-MFG_MDF}. Indeed, given the map $(\xi,\aleph)$, let us observe that if we define the map $(\widehat{B}, \widehat{L},\widehat{L}_{\Pr}):[0,T] \x \R^d \x \M \x \Cc_{\Wc} \x A \to \R^d \x \R \x \R$ and $\left(\widehat{g},\widehat{g}_{\Pr} \right):\Cc_{\Wc} \to \R$ by
    \begin{align*}
        ( \widehat{B}, \widehat{L}, \widehat{L}_{\Pr} )(t,x,m,\pi,a):=
        \left( \overline{b}, \overline{L}, \overline{L}_{\Pr} \right)(t,m,\aleph(t,\pi)) + 
        \left( b, L, {\rm 0} \right)(t,x,\pi,a),\;\;\left(\widehat{g}, \widehat{g}_{\Pr} \right) (\pi):=\left(\overline{g}, \overline{g}_{\Pr} \right)(\pi,\xi(R[\pi]))
    \end{align*}
    the maps $(\widehat{B}, \widehat{L},\widehat{L}_{\Pr},\sigma,\sigma_0,\widehat{g}, \widehat{g}_{\Pr})$ satisfy the assumptions of \cite[Proposition 3.15]{closed-loop-MFG_MDF}. Therefore, we just apply \cite[Proposition 3.15]{closed-loop-MFG_MDF} to conclude the proof of the Proposition.
\end{proof}

\subsection{Existence of MFG} \label{sec:existence_MFG}
This part is dedicated to showing that the set ${\rm MFG}[\Cf]$ is non--empty for any contract $\Cf$. Let us first considering the case without common noise. We introduce the following quantities: for each $t \in [0,T]$, $(m,\pi)$ and $c \in \R^d$, $m[c](\mathrm{d}y,\mathrm{d}a):=\int_{\R^d} \delta_{x+\sigma_0 c}(\mathrm{d}y)m(\mathrm{d}x,\mathrm{d}a)$ and $\pi(t)[c](\mathrm{d}y):=\int_{\R^d} \delta_{x+\sigma_0 c }(\mathrm{d}y)\pi(t)(\mathrm{d}x)$. We define the maps

\begin{align*}
    \overline{h}\left(t,c,m,e \right)
        :=
        \overline{b}\left( t, m[c],e\right),\;\;\left(h,F\right)\left(t,c,y,\pi,a \right)
        :=
        \left(b,L\right)\left(t,y+\sigma_0 c, \pi[c], a \right)\;\;\mbox{and}\;\;v(t,c,y):=\sigma(t,y+\sigma_0 c).
    \end{align*}
Notice that for each $c$, the maps $(\overline{h},\overline{F})$, $(h,F)$ and $v$ are continuous in $(c,m,e,y,\pi,a)$ for each $t$ and Lipschitz in $(m,y,\pi)$ uniformly in $(t,c,e,a)$ and $v(\cdot,c,\cdot)$. Also, $\inf_{(t,c,y)} v v^\top >0.$

\medskip
We say $({\rm q,\pi })$ belongs to $\Pcb_{\rm no}$ if $\mathrm{d}t$--a.e. $\qr(t) \left(\{m:\;m(\mathrm{d}y,A)=\pi(t)\} \right)=1$ and $\pi(0)=\nu$.
Given $(\br,e)$, for $({\rm q,\pi }) \in \Pcb_{\rm no}$, we will say $ (\qrt,\pirt) \in  \Pcb^{\br,e}_{\rm no}[\qr,\pi]$ if $(\qrt,\pirt) \in \Pcb_{\rm no}$ and, for all $(t,f)$,
\begin{align*}
    0=\langle f, \pirt(t) \rangle
    &- \bigg[
    \langle f, \nu \rangle
    +
    \int_0^t \int_{\Pc_A} \int_{\R^d} f'(y)\overline{h}\left(r,\br(r),m,e(r) \right) \pirt(r)(\mathrm{d}y) \;{\rm q}(r)(\mathrm{d}m)\mathrm{d}r
    \\
    &~~~~+\int_0^t \int_{\Pc_A} \int_{\R^d \x A} f'(y) h\left(r,\br(r),y,\pi,a \right) m (\mathrm{d}y,\mathrm{d}a)\; \qrt(r)(\mathrm{d}m)\mathrm{d}r 
    +
    \int_0^t\frac{1}{2}\int_{\R} f''(y) v(r,\br(r),y)^2 \pirt(r)(\mathrm{d}y)\mathrm{d}r \bigg]
    \\
    &=:N^{b,e}_{\pi,\qr}[\pirt,\qrt](t,f).
\end{align*}
We also set 
\begin{align*}
    \overline{J}_{q,\pi}^{\phi,\br}[\pirt,\qrt]
    :=&
    \int_{[0,T] \x \Pc_A} \int_{\R^d \x A} F\left(t,\br(t),y,\pi,a \right)m(\mathrm{d}y,\mathrm{d}a)\qrt(t)(\mathrm{d}m)\mathrm{d}t
    +
    \int_{\R}g\left( y +\sigma_0\br(T), \phi\left(y+\sigma_0 \br(T) \right) \right) \pirt(T)(\mathrm{d}x).
\end{align*}

We define the set
\begin{align*}
    \Rc^{\phi,e,\br}_{\pi,\qr} 
    :=
    \arg \max_{(\pirt,\qrt) \in \Pcb^{\br,e}_{\rm no}[\qr,\pi] }\overline{J}_{q,\pi}^{\phi,\br}[\pirt,\qrt] 
\end{align*}

\begin{proposition}
    Under {\rm \Cref{assum:main1}}, for any $(\br,e)$ and $\phi$, there exists $(q,\pi)$ belonging to $\Rc^{\phi,e,\br}_{\pi,\qr} $. 
\end{proposition}

\begin{proof}
    This proof is essentially a simple extension of \cite[Section 4]{open-loop-MFG_MDF}. Let $(\phi^k)_{k \ge 1}$ be a regularization of $\phi$ i.e. for each $k \ge 1$, $\phi^k \in C^\infty_c(\R)$ and $\lim_{k \to \infty} \phi^k=\phi$ a.e. Then, for each $k \ge 1$, by \cite[Theorem 4.6]{open-loop-MFG_MDF}, we can find $(\pi^k,q^k) \in \Rc^{\phi^k,e,\br}_{\pi,\qr}$. Since the coefficients are bounded and $\nu \in \Pc_{p'}(\R^d)$ with $p' >p$, we know that the sequence $(\pi^k,q^k)_{k \ge 1}$ is relatively compact in $\Wc_p$ (see also \cite[Lemma 4.2]{open-loop-MFG_MDF}). Let $(\pi,q)$ be the limit of a specific convergent sub--sequence $(\pi^{k_j},q^{k_j})_{j \ge 1}$. Also, we can check that $(\pi,q) \in \Pcb^{\br,e}_{\rm no}[q,\pi]$, and by using similar techniques to \Cref{prop:charac-convergence}, we have 
    $$
        \lim_{j \to \infty} \overline{J}_{q^{k_j},\pi^{k_j}}^{\phi^{k_j},\br}[\pi^{k_j},q^{k_j}]
        =
        \overline{J}_{q,\pi}^{\phi,\br}[\pi,q].
    $$ 
    For any $ (\qrt,\pirt) \in  \Pcb^{\br,e}_{\rm no}[\qr,\pi]$, by techniques borrowed from \cite[Proposition 4.5.]{open-loop-MFG_MDF} , there exists $(\pirt^{j},\qrt^{j})_{j \ge 1}$ verifying: $(\pirt^{j},\qrt^{j}) \in \Pcb^{\br,e}_{\rm no}[\qr^{k_j},\pi^{k_j}]$ for each $j \ge 1$, $\lim_{j \to \infty} \left( \pirt^j,\qrt^j \right)=\left(\pirt,\qrt \right)$ in $\Wc_p$ and $\lim_{j \to \infty} \overline{J}_{q^{k_j},\pi^{k_j}}^{\phi^{k_j},\br}[\pirt^{j},\qrt^{j}]
        =
        \overline{J}_{q,\pi}^{\phi,\br}[\pirt,\qrt]$ . Consequently,
        \begin{align*}
            \overline{J}_{q,\pi}^{\phi,\br}[\pi,q]
            =
            \lim_{j \to \infty} \overline{J}_{q^{k_j},\pi^{k_j}}^{\phi^{k_j},\br}[\pi^{k_j},q^{k_j}]
            \ge 
            \lim_{j \to \infty} \overline{J}_{q^{k_j},\pi^{k_j}}^{\phi^{k_j},\br}[\pirt^{j},\qrt^{j}]
            =
            \overline{J}_{q,\pi}^{\phi,\br}[\pirt,\qrt].
        \end{align*}
\end{proof}
We recall that ${\rm L}^p([0,T];\Er)$ is the ${\rm L}^p$--space of Borel measurable map $e:[0,T] \to \Er$ s.t. $\int_0^T\rho_{\Er}(e(t),e_0)^p \mathrm{d}t < \infty$ for some $e_0 \in \Er$. With the distance $\Wc(e,e')^p:=\int_0^T\rho_{\Er}(e(t),e'(t))^p \mathrm{d}t,$ $\left({\rm L}^p([0,T];\Er), \Wc \right)$ is a Polish space.
\begin{proposition} \label{prop:mfg_no}
    There exists a universally measurable map $\Psi: \Cc \x {\rm L}^p([0,T];\Er) \to \Cc_{\Wc} \x \M$ s.t. for each $(\br,e,z)$, $\Psi(\br,e) \in \Rc^{\phi,e,\br}_{\Psi(\br,e)}$.
\end{proposition}

\begin{proof}
    This proof is largely inspired by the proof of \cite[Lemma 4.7.]{djete2019mckean}. We will apply a selection measurable theorem. We introduce 
    \begin{align*}
        K^1_{t}[f, \widetilde{v}, \widetilde{h}]
        :=
        \left\{
            \left(\br,e,\pi,\qr,\pirt,\qrt \right):\;N^{\br,e}_{\pi,\qr}[\pirt,\qrt](t,f)=0,\;\; \int_{[0,T] \x \Pc_A} \widetilde{h}(t) \left( \langle \widetilde{v}, m(\mathrm{d}x,A) \rangle - \langle \widetilde{v} , \pirt(t) \rangle \right) \qrt(t)(\mathrm{d}m)\mathrm{d}t=0
        \right\}
    \end{align*}
    and
    \begin{align*}
        K^2[v,h]
        :=
        \left\{
            \left(\br,e,\pi,\qr,\pirt,\qrt \right):\; \int_{[0,T] \x \Pc_A} h(t) \left( \langle v, m(\mathrm{d}x,A) \rangle - \langle v , \pi(t) \rangle \right) \qr(t)(\mathrm{d}m)\mathrm{d}t=0
        \right\}.
    \end{align*}
    Let us observe that for any $(t,f,v,h,\widetilde{v},\widetilde{h})$, the sets $K^1_{t}[f, \widetilde{v}, \widetilde{h}]$ and $K^2[v,h]$ are Borel sets. We get that
    \begin{align*}
        \left\{ \left(\br,e,\pi,q,\pirt,\qrt \right):\;\;(\pirt,\qrt) \in \Pcb^{\br,e}_{\rm no}[\qr,\pi],\;(\pi,\qr) \in \Pcb_{\rm no}
        \right\}
        =
        \bigcap_{(t,f,v,h,\widetilde{v},\widetilde{h}) \in \X} K^1_{t}[f, \widetilde{v}, \widetilde{h}] \cap K^2[v,h]
    \end{align*}
    where $\X$ is a suitable countable dense set of $[0,T] \x C^\infty_c \x C^\infty_c \x C([0,T]) \x C^\infty_c \x C([0,T])$. Therefore, since the map $\Cc \x {\rm L}^p([0,T];\Er) \x (\Cc_{\Wc} \x \M)^2 \ni (\br,e,\pi,q,\pirt,\qrt) \mapsto \overline{J}_{q,\pi}^{\phi,\br}[\pirt,\qrt] \in \R $ is Borel measurable, by \cite[Proposition 2.21.]{karoui2013capacities}, the map $\Cc \x {\rm L}^p([0,T];\Er)\x \Cc_{\Wc} \x \M \ni (\br,e,\pi,q) \mapsto \overline{V}^{\phi}[\br,e,\pi,q] \in \R $ is universally measurable where
    \begin{align*}
        \overline{V}^{\phi}[\br,e,\pi,q]
        :=
        \sup_{(\tilde\pi,\tilde q) \in \Pcb^{\br,e}_{\rm no}[\qr,\pi] }\overline{J}_{q,\pi}^{\phi,\br}[\pirt,\qrt].
    \end{align*}
    If we define
    \begin{align*}
        K^2_{t}[f,v,h]
        :=
        \left\{
            \left(\br,e,\pi,q\right):\;N^{\br,e}_{\pi,\qr}[\pi,q](t,f)=0,\;\;\int_{[0,T] \x \Pc_A} h(t) \left( \langle v, m(\mathrm{d}x,A) \rangle - \langle v , \pi(t) \rangle \right) \qr(t)(\mathrm{d}m)\mathrm{d}t=0
        \right\},
    \end{align*}
    we find that the set of mean field game solutions $\Sc^\star \subset \Cc \x {\rm L}^p([0,T;\Er]) \x \Cc_{\Wc} \x \M$ is a universally measurable set because
    \begin{align*}
        \Sc^\star:=\left\{ (\br,e,\pi,q):\;(\pi,q) \in \Rc^{\phi,e,\br}_{\pi,\qr}
        \right\}
        =
        \left\{ 
        \left(\br,e,\pi,q\right):\; \overline{V}^{\phi}[\br,e,\pi,q] \le \overline{J}_{q,\pi}^{\phi,\br}[\pi,q]
        \right\} \bigcap \left( \cap_{(t,f,v,h) \in \Z} K^2_{t}[f,v,h] \right)
    \end{align*}
    where $\Z$ is a suitable countable dense set of $[0,T] \x C^\infty_c \x C^\infty_c \x C([0,T])$.
    Consequently, by applying \cite[Theorem 2.20.]{karoui2013capacities}, there exists a universally measurable map $\Psi:\Cc \x {\rm L}^p([0,T];\Er) \to \Cc_{\Wc} \x \M$ verifying $\Psi(\br,e) \in \Rc^{\phi,e,\br}_{\Psi(\br,e)}$ for each $(\br,e)$.
\end{proof}

\begin{proposition}
    Under {\rm \Cref{assum:main1}} and {\rm \Cref{assum:convexity}}, for any contract $\Cf=(\phi,\xi,\aleph)$, the set ${\rm MFG}[\Cf]$ is non--empty.
\end{proposition}
\begin{proof}
    We know that $\aleph$ is an $\Er$--valued $\F$--predictable process 
    s.t. $(\aleph,B)$, $\iota$ and $W$ are independent. With the map $\Psi$ given in \Cref{prop:mfg_no}, we define $\Rt:=\widetilde{Q}_t(\mathrm{d}m)\delta_{\aleph_t}(\mathrm{d}e)\mathrm{d}t$ with
\begin{align*}
    (\widetilde{Q}, \widetilde{\eta} )
    :=
    \Psi \left(B, \aleph \right)\;\;\mbox{and we set}\;\left( \Gc_t:=\sigma \{ B_{t \wedge \cdot}, \Rt_{t \wedge \cdot}, \widetilde{\eta}_{t \wedge \cdot} \} \right)_{t \in [0,T]}.
\end{align*}
We know that $\P$--a.e. $N_{\tilde \eta,\tilde Q}[\widetilde{\eta},\widetilde{Q}](t,f)=0$ for all $(t,f)$. Then, we can see that  $\Lc(Y_t| B,\aleph)=\Lc(Y_t| \Gc_t)=\widetilde{\eta}_t$, $\P$--a.e. where $Y$ satisfies $\Lc(Y_0)=\nu$ and
\begin{align*}
        \mathrm{d}Y_t
            =
            \int_{\Pc_A}\overline{h}\left(t,B_t,m,\aleph_t\right) \widetilde{Q}_t(\mathrm{d}m)
            +
            \int_{\Pc_A}\int_A h\left(t,B_t,Y_t ,\widetilde{\eta},a\right)m^{Y_t}(\mathrm{d}a) \widetilde{Q}_t(\mathrm{d}m)\;\;\;\mathrm{d}t + v(t,B_t,Y_t)\mathrm{d}W_t\;\;\P\mbox{--a.e.}
    \end{align*}
    If we define $X_\cdot:=Y_\cdot+ \sigma_0 B_\cdot$. The process $X$ follows the dynamics
    \begin{align*}
    \mathrm{d}X_t=\int_{\Pc_A \x \Er}\overline{b}\left(t,m,e\right)K_t(\mathrm{d}m,\mathrm{d}e)
            +
            \int_{\Pc_A}\int_A b\left(t,X_t,\mu,a\right)m^{X_t}_t(\mathrm{d}a)K_t(\mathrm{d}m,\Er)\mathrm{d}t + \sigma(t,X_t)\mathrm{d}W_t + \sigma_0\mathrm{d}B_t\;\;\P\mbox{--a.e.}
    \end{align*}
    where $K_t(\mathrm{d}m,\mathrm{d}e)\mathrm{d}t:=\int_{\Pc_A} \delta_{m'[B_t]}(\mathrm{d}m)\widetilde{Q}_t(\mathrm{d}m')\delta_{\aleph_t}(\mathrm{d}e)\mathrm{d}t$ and $\mu_t:=\widetilde{\eta}_t[B_t]$ for each $t \in [0,T]$.
    We can find a Borel map $\Lambda:[0,T] \x \R^d \x \M \to \Pc(A)$ s.t. $\Lambda(t,X_t,K)(\mathrm{d}a):=\int_{\Pc_A}m^{X_t}(\mathrm{d}a) K_t(\mathrm{d}m,\Er)$. Notice that $K_t\left( \{ (m,e):\;m(\mathrm{d}x,A)=\mu_t  \} \right)=1$, $\mathrm{d}\P \otimes \mathrm{d}t$--a.e.
    
    \medskip
    Let $\overline{\Pr}:=\P \circ \left(X,W,B,K,\mu \right)^{-1} \in \Pcb$. Let us check the optimality condition of a MFG solution. 
    Let $\Pr \in \Pcb$ s.t. $\Lc^\Pr(B,R)=\Lc^{\overline{\Pr}}(B,R)$. We know that we can write $\Pr=\P \circ \left(X',W,B,R,\mu \right)^{-1}$ with 
    \begin{align*}
    \mathrm{d}X'_t=\int_{\Pc_A \x \Er}\overline{b}\left(t,m,e\right)R_t(\mathrm{d}m,\mathrm{d}e)
            +
            \int_A b\left(t,X'_t,\mu,a\right)\Lambda'(t,X'_t,R)(\mathrm{d}a)\mathrm{d}t + \sigma(t,X'_t)\mathrm{d}W_t + \sigma_0\mathrm{d}B_t\;\;\P\mbox{--a.e.}
    \end{align*}
    We define $\mut_t:=\Lc(X'_t - \sigma_0 B_t|\Gc_t)$ and $\widetilde{H}:=\delta_{\overline{v}_t}(\mathrm{d}m)\mathrm{d}t$ where $\overline{v}_t:=\E \left[ \delta_{X'_t- \sigma_0 B_t}(\mathrm{d}y) \Lambda'(t,X_t',R)(\mathrm{d}a) |\Gc_t\right]$. We then find that, $\P$--a.e. $\om \in \Om$, $(\widetilde{\mu},\widetilde{H})\in \Pcb_{\rm no}$, $N_{\tilde \eta,\tilde Q}[\widetilde{\mu},\widetilde{H}](t,f)=0$ for all $(t,f)$, and 
    \begin{align*}
         \overline{J}_{\tilde Q,\tilde \eta}^{\phi,\aleph,B}[\widetilde{\mu},\widetilde{H}]
        \le 
        \overline{J}_{\tilde Q,\tilde \eta}^{\phi,\aleph,B}[\widetilde Q,\widetilde \eta].
    \end{align*}
    We recall that $\xi:\M \to \Er$ is a Borel map, consequently, 
    \begin{align*}
        \E^{\Pr} \left[ \Jr_A^{\Phi,\Zf} (\Lambda',\mu',R,\mu) \right]
        &=
        \E^{\Pr} \left[ \overline{J}_{\tilde Q,\tilde \eta}^{\phi,\aleph,B}[\widetilde{\mu},\widetilde{H}] + \int_{[0,T] \x \Pc_A \x \Er} \overline{L}\left( t, m,e\right) R_t(\mathrm{d}m,\mathrm{d}e)\mathrm{d}t + \overline{g} \left( \mu, \xi(R) \right)\right]
        \\
        &\le 
        \E^{\Pr} \left[ \overline{J}_{\tilde Q,\tilde \eta}^{\phi,\aleph,B}[\widetilde Q,\widetilde \eta] + \int_{[0,T] \x \Pc_A \x \Er} \overline{L}\left( t, m,e\right) R_t(\mathrm{d}m,\mathrm{d}e)\mathrm{d}t + \overline{g} \left( \mu, \xi(R) \right)\right]
        \\
        &=
        \E^{\overline{\Pr}} \left[ \Jr_A^{\Phi,\Zf} (\Lambda',\mu',R,\mu) \right]
    \end{align*}

    \medskip
    We can deduce that $\overline{\Pr} \in \Pcb_{\rm mfg}$. And, with \Cref{prop:relaxed_to_weak}, we get $\Prt \in \Pcb_{\rm mfg}^W$ s.t. $\E^{\overline{\Pr}} \left[ \Jr_A^{\Phi,\Zf} (\Gamma',\mu',R,\mu) \right] \le \E^{\widetilde{\Pr}} \left[ \Jr_A^{\Phi,\Zf} (\Gamma',\mu',R,\mu) \right]$. 
    
\end{proof}

\subsection{Proof of the main results} \label{sec:proofs_main}

\subsubsection{Proof of \Cref{thm:convergence_limit-contract}}

\begin{itemize}
    \item We begin by proving the first point.
    This first point is essentially an application of \Cref{prop_convergence:limitcontract}. Indeed, let $(\alpha^\ell,\mub^\ell,\Cf^\ell,\varepsilon_\ell)_{\ell \ge 1}$ be the sequence satisfying the property stated in \Cref{thm:convergence_limit-contract}. By setting $\Pr^\ell:=\P \circ \left( X^{\alpha^\ell}, W, B, R^\ell, \mu^\ell \right)^{-1}$ where $R^\ell:=\delta_{(\mub^\ell_t,\aleph^\ell_t)}(\mathrm{d}m,\mathrm{d}e)\mathrm{d}t$ for each $\ell \ge 1$, by \Cref{prop_convergence:limitcontract}, we know that the sequence $(\Pr^\ell)_{\ell \ge 1}$ is relatively compact in $\Wc_p$ and each limit point $\Pr$ of a convergent sub--sequence $\left(\Pr^{\ell_k} \right)_{k \ge 1}$ belongs to $\Pcb_{\rm mfg}$ associated to $(\Phi,\Zf)$ that satisfies: $\lim_{k \to \infty} \Lc^\P \left( \xi^{\ell_k},\; R^{\ell_k} \right) 
        =
        \Zf(r)(\mathrm{d}e)\Lc^{\Pr} ( R)(\mathrm{d}r)$ and 
    \begin{align*}
        \lim_{k \to \infty} \Lc^\P \left( \mu^{\ell_k},\; \delta_{\phi^{n_k}(x)}(\mathrm{d}i)\mu^{\ell_k}_T(\mathrm{d}x),\; R^{\ell_k} \right) 
        =
        \Lc^{\Pr} \left( \mu,\; \Phi(x)(\mathrm{d}i)\mu_T(\mathrm{d}x),\; R \right)\;\mbox{in}\;\Wc_p.
    \end{align*} 
    Let $\Qr^{k}:=\Lc^\P \left( \mu^{\ell_k},\; \delta_{\phi^{n_k}(x)}(\mathrm{d}i)\mu^{\ell_k}_T(\mathrm{d}x),\;\xi^{\ell_k},\; R^{\ell_k} \right)$. It is straightforward that the sequence $(\Qr^k)_{k \ge 1}$ is relatively compact in $\Wc_p$. Let $\Qr=\P \circ \left( \mu,\kappa, \xi,R \right)^{-1}$ be the limit of a convergent sub--sequence. Let us keep the same notation for the sequence and the sub--sequence for simplification. By the convergence showed previously and the definition of the variables, we easily check that $\mu=H(R)$ where $H$ is a Borel map and $\kappa=\Phi(x)(\mathrm{d}i)\mu_T(\mathrm{d}x)$ $\P$--a.e.. Consequently,
    \begin{align*}
        &\lim_{k \to \infty} J^{\alpha^{\ell_k},\mub^{\ell_k}}_{\Pr}(\Cf^{\ell_k})
        \\
        &=
        \lim_{k \to \infty}\E \left[ U \left( \int_{\R^d} \Upsilon( x) \mu^{\ell_k}_T(\mathrm{d}x) - \int_{\R^d} g_{\Pr}\left(x,\phi^{\ell_k}(x) \right) \mu^{\ell_k}_T(\mathrm{d}x)
        -\overline{g}_{\Pr} \left(\mu^{\ell_k},\xi^{\ell_k} \right)
        -
        \langle \overline{L}_{\Pr} , R^{\ell_k} \rangle\right) \right]
        \\
        &=
        \E\left[ U \left( \int_{\R^d} \Upsilon( x) \mu_T(\mathrm{d}x) - \int_{\R^d} g_{\Pr}\left(x,i \right) \Phi(x)(\mathrm{d}i)\mu_T(\mathrm{d}x)
        -\overline{g}_{\Pr} \left(\mu,\xi \right)
        -
        \langle \overline{L}_{\Pr} , R \rangle\right) \right]
        \\
        &=
        \int_{\M \x \Er} U \left( \int_{\R^d} \Upsilon( x) H_T(\rr)(\mathrm{d}x) - \int_{\R^d} g_{\Pr}\left(x,i \right) \Phi(x)(\mathrm{d}i)H_T(\rr)(\mathrm{d}x)
        -\overline{g}_{\Pr} \left(H(\rr),\xi \right)
        -
        \langle \overline{L}_{\Pr} , \rr \rangle\right) \Zf(r)(\mathrm{d}e) \Lc(R)(\mathrm{d}r)
        \\
        &=
        \E^{{\Pr}}\big[{\rm J}^{R,\mu}_{\Pr}(\Phi,\Zf) \big].
    \end{align*}
    This is true for any sub--sequence of $(\Qr)_{k \ge 1}$, we can then deduce the convergence just proved for the all sequence.

    \medskip
    By combining \Cref{prop:relaxed_to_weak} and \Cref{prop:eq--measure--valuedMFG_control}, there exist a contract $\Cf=(\phi,\xi,\aleph)$ and $(\alpha, \mub) \in {\rm MFG}[\Cf]$ satisfying
    \begin{align*}
        \E^{{\Pr}}\big[{\rm J}^{\Phi,\Zf}_A(\Gamma',\mu',R,\mu) \big] \le J_{A,\mub}^{\Cf}(\alpha)\;\;\mbox{and}\;\;\E^{{\Pr}}\big[{\rm J}^{R,\mu}_{\Pr}(\Phi,\Zf) \big] \le J^{\alpha,\mub}_{\Pr}(\Cf).
    \end{align*}
    This is enough to conclude the proof of the first part of the Theorem.

    \item This is a direct application of \Cref{prop:approximation_by_strong}.
\end{itemize}

\subsubsection{Proof of \Cref{thm:existence_contract}}

Since $R_0$ is s.t. $\Xi$ is non--empty, let $(\alpha^\ell,\mub^\ell,\Cf^\ell)_{\ell \ge 1}$ be the sequence satisfying: for each $\ell \ge 1$, $\Cf^\ell=(\phi^\ell,\xi^\ell,\aleph^\ell)$ belongs to $\Xi$, $(\alpha^\ell,\mub^\ell) \in \underline{\rm MFG}[\Cf^\ell]$ and $V_{\Pr} \le J^{\alpha^\ell, \mub^\ell}_{\Pr}(\Cf^\ell) + \frac{1}{2^\ell}$. By considering $\varepsilon_\ell=0$ for each $\ell \ge 1$, we can apply \Cref{thm:convergence_limit-contract} and find a sub--sequence $(\ell_j)_{j \ge 1}$, a contract $\Cf^\star$ and $(\alpha^\star,\mub^\star)$ s.t. $(\alpha^\star,\mub^\star) \in \mbox{{\rm MFG}}[\Cf^\star]$, 
$$
    \lim_{j \to \infty}J^{\Cf^{\ell_j}}_{A,\mub^{\ell_j}}({\alpha^{\ell_j}})\le
            J^{\Cf^\star}_{A,\mub^\star}({\alpha^\star})~~~~~\mbox{\rm and}~~~~\lim_{j \to \infty} J^{\alpha^{\ell_j},\mub^{\ell_j}}_{\Pr}(\Cf^{\ell_j}) \le
            J^{\alpha^\star,\mub^\star}_{\Pr}(\Cf^\star).
$$
Consequently, we deduce that $V_{\Pr} \le J^{\Cf^\star}_{A,\mub^\star}({\alpha^\star})$. Since $(\alpha^\star,\mub^\star) \in {\rm MFG}[\Cf^\star]$ and $R_0 \le J^{\Cf^\star}_{A,\mub^\star}({\alpha^\star})$, in fact $\Cf^\star \in \Xi$ and $V^{\Pr} = J^{\Cf^\star}_{A,\mub^\star}({\alpha^\star})$. The shape of the control is deduced by \Cref{prop:relaxed_to_weak} and \Cref{prop:eq--measure--valuedMFG_control}.

\subsubsection{Proof of \Cref{cor:limit_strong_valuefunc}}

Notice that, for $(\alpha,\mub) \in \underline{\rm MFG}[\Cf]$, by the second point of \Cref{thm:convergence_limit-contract}, there exist a sequence of non--negative numbers $(\varepsilon_{\ell})_{\ell \ge 1}$ s.t. $\lim_{\ell \to \infty} \varepsilon_\ell=0,$ and a sequence $(\alpha^\ell,\mub^\ell,\phi^\ell, \xi^\ell,\aleph^\ell)_{\ell \ge 1}$ s.t. for each $\ell \ge 1$, $\Cf^\ell=(\phi^\ell,\xi^\ell,\aleph^\ell)$ is a contract, $(\alpha^\ell, \mub^\ell )\in \mbox{{\rm MFG}}_S[\Cf^\ell,\varepsilon_\ell]$ and $\lim_{\ell \to \infty}J^{\Cf^{\ell}}_{A,\mub^\ell}(\alpha^{\ell})=J^{\Cf}_{A,\mub}({\alpha}).$ Therefore, by setting $\varepsilon_{0,\ell}:=\left|J^{\Cf^{\ell}}_{A,\mub^\ell}(\alpha^{\ell})-J^{\Cf}_{A,\mub}(\alpha) \right|$,
$$
    J^{\Cf^{\ell}}_{A,\mub^\ell}(\alpha^{\ell}) = J^{\Cf}_{A,\mub}(\alpha) + J^{\Cf^{\ell}}_{A,\mub^\ell}(\alpha^{\ell})-J^{\Cf}_{A,\mub}(\alpha)
    \ge 
    R_0 + J^{\Cf^{\ell}}_{A,\mub^\ell}(\alpha^{\ell})-J^{\Cf}_{A,\mub}(\alpha)
    \ge 
    R_0 - \varepsilon_{0,\ell}.
$$
We have  $(\alpha^\ell, \mub^\ell )\in \underline{{\rm MFG}}_S[\Cf^\ell,\varepsilon_\ell,\varepsilon_{0,\ell}]$.  Since, $\lim_{\ell \to \infty} \varepsilon_{\ell}=\lim_{\ell \to \infty} \varepsilon_{0,\ell}=0$, we deduce that $\Xi_S[\varepsilon,\varepsilon_0]$ is non--empty for each $\varepsilon,\varepsilon_0>0$.

\medskip
Let $(\Cf^{\ell,\varepsilon,\varepsilon_0}, \alpha^{\ell,\varepsilon,\varepsilon_0}, \mub^{\ell,\varepsilon,\varepsilon_0})_{\ell \ge 1, \varepsilon >0, \varepsilon_0 >0}$ be a sequence s.t. $\Cf^{\ell,\varepsilon,\varepsilon_0} \in \Xi[\varepsilon,\varepsilon_0]$, $(\alpha^{\ell,\varepsilon,\varepsilon_0},\mub^{\ell,\varepsilon,\varepsilon_0}) \in \underline{\rm MFG}_S[\Cf,\varepsilon,\varepsilon_0]$ and for each $\ell \ge 1$, $\varepsilon, \varepsilon>0$,
\begin{align*}
    V^S_{\Pr}[\varepsilon,\varepsilon_0] \le J^{\alpha^{\ell,\varepsilon,\varepsilon_0},\mub^{\ell,\varepsilon,\varepsilon_0}}_{\Pr}(\Cf^{\ell,\varepsilon,\varepsilon_0}) + \frac{1}{2^\ell}.
\end{align*}
We apply the first point of \Cref{thm:convergence_limit-contract} to find a sub--sequence $(\ell_j,\varepsilon_j,\varepsilon_{0,j})_{j \ge 1}$, a contract $\Cf$ and $(\alpha,\mub) \in {\rm MFG}[\Cf]$ s.t. $\lim_{j \to \infty}J^{\alpha^{\ell_j,\varepsilon_j,\varepsilon_{0,j}},\mub^{\ell_j,\varepsilon_j,\varepsilon_{0,j}}}_{\Pr}(\Cf^{\ell_j,\varepsilon_j,\varepsilon_{0,j}}) \le J^{\alpha,\mub}_{\Pr}(\Cf)$. It is straightforward that $(\alpha,\mub) \in \underline{\rm MFG}[\Cf]$.  This allows us to deduce that
\begin{align*}
    \limsup_{(\varepsilon,\varepsilon_0) \to (0,0)} V^S_{\Pr}[\varepsilon,\varepsilon_0]  \le V_{\Pr}.
\end{align*}

\medskip
Now, let $\Cf$ be a contract and $(\alpha,\mub) \in \underline{\rm MFG}[\Cf]$. Next, we apply the second point of \Cref{thm:convergence_limit-contract} (combined with the previous point) and find a sequence of non--negative numbers $(\varepsilon_{\ell})_{\ell \ge 1}$ s.t. $\lim_{\ell \to \infty} \varepsilon_\ell=0,$ and a sequence $(\alpha^\ell,\mub^\ell,\phi^\ell, \xi^\ell,\aleph^\ell)_{\ell \ge 1}$ s.t. for each $\ell \ge 1$, $\Cf^\ell=(\phi^\ell,\xi^\ell,\aleph^\ell)$ is a contract, $(\alpha^\ell, \mub^\ell )\in \underline{{\rm MFG}}_S[\Cf^\ell,\varepsilon_\ell,\varepsilon_{0,\ell}]$ and $\lim_{\ell \to \infty}J^{\Cf^{\ell}}_{A,\mub^\ell}(\alpha^{\ell})=
            J^{\Cf}_{A,\mub}({\alpha})~~~~~\mbox{\rm and}~~~~\lim_{\ell \to \infty} J^{\alpha^{\ell},\mub^\ell}_{\Pr}(\Cf^\ell)=
            J^{\alpha,\mub}_{\Pr}(\Cf).$ 

\medskip            
For each $\delta, \delta_0 >0$, there exists $\ell(\delta,\delta_0)$ s.t. for each $\ell \ge \ell(\delta,\delta_0)$, $\varepsilon_\ell \le \delta$ and $\varepsilon_{0,\ell} \le \delta_0$. Therefore, for each $\ell \ge \ell(\delta,\delta_0)$,
\begin{align*}
    J^{\alpha^{\ell},\mub^\ell}_{\Pr}(\Cf^\ell) \le V_{\Pr}^S[\delta,\delta_0].
\end{align*}
Then, $J^{\alpha,\mub}_{\Pr}(\Cf) \le V^S_{\Pr}[\delta,\delta_0]$ for each $\delta, \delta_0 >0$. This leads to $V_{\Pr} \le \liminf_{(\varepsilon,\varepsilon_0) \to (0,0)} V^S_{\Pr}[\varepsilon,\varepsilon_0]$. By combining, all the results,
$$
    V_{\Pr} = \lim_{(\varepsilon,\varepsilon_0) \to (0,0)} V^S_{\Pr}[\varepsilon,\varepsilon_0]. 
$$
We can conclude the proof since the second point follows immediately.

\subsubsection{Proof of \Cref{thm:conv:n-player}}

\begin{itemize}
    \item We start by proving the first point. This proof is quite close to the proof of \Cref{thm:convergence_limit-contract}. Let $(\alphab^n,\phi^n,\xi^n,\aleph^n)_{n \ge 1}$ be a sequence given as in the Theorem. We set $\Pr^n
    :=
    \frac{1}{n} \sum_{i=1}^n \P \circ \left( X^i, W^i, B, R^n, \varphi^n[\alphab^n] \right)^{-1}$ 
$\mbox{where}\;R^n:=\delta_{\left( \overline{\varphi}^n_t[\alphab^n], \aleph^n(t,\Xbb_t)\right)}(\mathrm{d}m,\mathrm{d}e)\mathrm{d}t.$ By \Cref{prop:convegence_n-player}, the sequence $(\Pr^n)_{n \ge 1}$ is relatively compact in $\Wc_p$ and any limit point $\Pr$ of a convergent sub--sequence $\left(\Pr^{n_k} \right)_{k \ge 1}$ is a MFG solution i.e. belongs to $\Pcb_{\rm mfg}$ associated to $(\Phi,\Zf)$ which verifies: $\lim_{k \to \infty} \Lc^\P \left( \xi^{n_k}(\Xbb),\; R^{n_k} \right) 
        =
        \Zf(r)(\mathrm{d}e)\Lc^{\Pr} \left( R \right)(\mathrm{d}r)$ and 
    \begin{align*}
        \lim_{k \to \infty} \Lc^\P \left( \varphi^{n_k}[\alphab^{n_k}],\; \delta_{\phi^{n_k}(x)}(\mathrm{d}i)\varphi^{n_k}_T[\alphab^{n_k}](\mathrm{d}x),\; R^{n_k} \right) 
        =
        \Lc^{\Pr} \left( \mu,\; \Phi(x)(\mathrm{d}i)\mu_T(\mathrm{d}x),\; R \right)\;\mbox{in}\;\Wc_p.
    \end{align*} 
    Similarly to the proof of \Cref{thm:convergence_limit-contract}, we combine \Cref{prop:relaxed_to_weak} and \Cref{prop:eq--measure--valuedMFG_control} and find a contract $\Cf=(\phi,\xi,\aleph)$ and $(\alpha, \mub) \in {\rm MFG}[\Cf]$ satisfying $\E^{{\Pr}}\big[{\rm J}^{\Phi,\Zf}_A(\Gamma',\mu',R,\mu) \big] \le J_{A,\mub}^{\Cf}(\alpha)\;\;\mbox{and}\;\;\E^{{\Pr}}\big[{\rm J}^{R,\mu}_{\Pr}(\Phi,\Zf) \big] \le J^{\alpha,\mub}_{\Pr}(\Cf).$
    We can then conclude the proof of the first part of the Theorem.

    \item This second point is just an application of \Cref{prop:approximation_by_strong} combined with \Cref{prop:conv_n-player-strong}.
\end{itemize}

\subsubsection{Proof of \Cref{cor:conv_strong_limit_n-player}}

The proof of this Corollary follows the same idea as the proof of \Cref{cor:limit_strong_valuefunc}. By mimicking the proof of \Cref{cor:limit_strong_valuefunc} combined with the second part of \Cref{thm:conv:n-player}, we see that, for each $\varepsilon, \varepsilon_0 >0$, there exists $\overline{n}$ s.t. for each $n  \ge \overline{n}$, $\Xi^{\rm dist}_n[\varepsilon, \varepsilon_0]$ is non--empty. 

\medskip
Let $(\Cf^{n,\varepsilon,\varepsilon_0},\alphab^{n,\varepsilon,\varepsilon_0})_{n \ge 1, \varepsilon>0, \varepsilon_0 >0}$ be a sequence s.t. $\Cf^{n,\varepsilon,\varepsilon_0} \in \Xi_n[\varepsilon,\varepsilon_0]$, $\alphab^{n,\varepsilon,\varepsilon_0} \in \underline{\rm NE}[\Cf^{n,\varepsilon,\varepsilon_0},\varepsilon, \varepsilon_{0}]$ and for $n \ge 1$ 
$$
    V_{n,\Pr}[\varepsilon,\varepsilon_{0}] \le J^{\alphab^{n,\varepsilon,\varepsilon_0}}_{n,\Pr}(\Cf^{n,\varepsilon,\varepsilon_0}) + \frac{1}{2^n}.
$$
By \Cref{thm:conv:n-player}, there exist a sub--sequence $(n_j, \varepsilon_j,\varepsilon_{0,j})_{j \ge 1}$, a contract $\Cf$ and $(\alpha,\mub) \in {\rm MFG}[\Cf]$ satisfying 
$$
    \lim_{j \to \infty}J^{\alphab^{n_j,\varepsilon_j,\varepsilon_{0,j}}}_{n_j,\Pr}(\Cf^{n_j,\varepsilon_j,\varepsilon_{0,j}}) \le J^{\alpha,\mub}_{\Pr}(\Cf).
$$
It is straightforward that $(\alpha,\mub) \in \underline{\rm MFG}[\Cf]$.   This allows us to see that 
$$
    \limsup_{(\varepsilon,\varepsilon_0) \to (0,0)}\limsup_{n \to \infty} V_{n,\Pr}[\varepsilon, \varepsilon_{0}] \le V_{\Pr}.
$$ 
Next, let $\Cf$ be an admissible contract and $(\alpha,\mub) \in \underline{\rm MFG}[\Cf]$, by \Cref{thm:conv:n-player}, there exist  a sequence of non--negative numbers $(\varepsilon_n)_{n \ge 1}$ s.t. $\lim_{n \to \infty} \varepsilon_n=0,$ and a sequence $(\alphab^n,\Cf^n)_{n \ge 1}$ with $\Cf^n$ a $distributed$ contract and $\alphab^n \in {\rm NE}^{\rm dist}[\Cf^n,\varepsilon_n]$ for each $n \ge 1$, and we have $\lim_{n \to \infty} J_{n,\Pr}^{{\color{black}\alphab^n}}(\Cf^n)=J_{\Pr}^{\alpha,\mub}(\Cf)$. We set 
$$
    {\varepsilon}_{0,n}:=\left|J^{\Cf^n}_{n,i}(\alphab^n) - J^{\Cf}_{A,\mub}(\alpha) \right|
    =
    \left|J^{\Cf^n}_{n,1}(\alphab^n) - J^{\Cf}_{A,\mub}(\alpha) \right|.
$$
We have $\alphab^n \in \underline{\rm NE}^{\rm dist}[\Cf^n,\varepsilon_n, {\varepsilon}_{0,n}]$. Since $\lim_{n \to \infty} {\varepsilon}_{0,n}=\lim_{n \to \infty} \varepsilon_n=0$, let $\delta, \delta_0 >0$, there exists $n(\delta,\delta_0) \in \N^*$ s.t. for any $n \ge n(\delta,\delta_0)$, we have $\varepsilon_n \le \delta$ and $\varepsilon_{0,n} \le \delta_0$. Consequently, for any $n \ge n(\delta,\delta_0)$, 
\begin{align*}
    J_{n,\Pr}^{{\color{black}\alphab^{n}}}(\Cf^{n})
    \le 
    V_{n,\Pr}^{\rm dist}[\delta,\delta_0].
\end{align*}
Then $J_{\Pr}^{\alpha,\mub}(\Cf) \le \liminf_{n \to \infty} V_{n,\Pr}^{\rm dist}[\delta,\delta_0]$, for any $\delta, \delta_0 >0$. This leads to $V_{\Pr} \le \liminf_{(\delta,\delta_0) \to (0,0)}\liminf_{n \to \infty} V_{n,\Pr}^{\rm dist}[\delta,\delta_0].$
By summarizing, we find
$$
    V_{\Pr} \le \liminf_{(\delta,\delta_0) \to (0,0)}\liminf_{n \to \infty} V_{n,\Pr}^{\rm dist}[\delta,\delta_0] \le \limsup_{(\delta,\delta_0) \to (0,0)}\limsup_{n \to \infty} V^{\rm dist}_{n,\Pr}[\delta,\delta_0] \le V_{\Pr}
$$
and
$$
    V_{\Pr} \le \liminf_{(\delta,\delta_0) \to (0,0)}\liminf_{n \to \infty} V_{n,\Pr}[\delta,\delta_0] \le \limsup_{(\delta,\delta_0) \to (0,0)}\limsup_{n \to \infty} V_{n,\Pr}[\delta,\delta_0] \le V_{\Pr}.
$$
We deduce the proof of the Corollary. 

\bibliographystyle{plain}

\bibliography{MFG-PA-ArxivVersion1}

\begin{thebibliography}{56}
\providecommand{\natexlab}[1]{#1}
\providecommand{\url}[1]{\texttt{#1}}
\expandafter\ifx\csname urlstyle\endcsname\relax
  \providecommand{\doi}[1]{doi: #1}\else
  \providecommand{\doi}{doi: \begingroup \urlstyle{rm}\Url}\fi

\bibitem[Aronson and Serrin(1967)]{AronsonSerrin67}
D.~G. Aronson and J.~Serrin.
\newblock Local behavior of solutions of quasilinear parabolic equations.
\newblock \emph{Archive for Rational Mechanics and Analysis volume},
  25:\penalty0 81–122, 1967.
\newblock \doi{10.1137/S0363012996313549}.
\newblock URL \url{https://doi.org/10.1007/BF00281291}.

\bibitem[Aurell et~al.(2022)Aurell, Carmona, Dayanikli, and
  Lauri\`{e}re]{DayanikliOptimal2022}
A.~Aurell, R.~Carmona, G.~Dayanikli, and M.~Lauri\`{e}re.
\newblock Optimal incentives to mitigate epidemics: A stackelberg mean field
  game approach.
\newblock \emph{SIAM Journal on Control and Optimization}, 60\penalty0
  (2):\penalty0 S294--S322, 2022.
\newblock \doi{10.1137/20M1377862}.
\newblock URL \url{https://doi.org/10.1137/20M1377862}.

\bibitem[Bogachev et~al.(2015)Bogachev, Krylov, R\"{o}ckner, and
  Shaposhnikov]{FK-PL-equations}
V.~I. Bogachev, N.~V. Krylov, M.~R\"{o}ckner, and S.~V. Shaposhnikov.
\newblock \emph{{F}okker–{P}lanck–{K}olmogorov {E}quations}.
\newblock Mathematical Surveys and Monographs. American Mathematical Society,
  2015.

\bibitem[Carmona and Delarue(2018)]{carmona2018probabilisticI}
R.~Carmona and F.~Delarue.
\newblock \emph{Probabilistic theory of mean field games with applications
  {I}}, volume~83 of \emph{Probability theory and stochastic modelling}.
\newblock Springer International Publishing, 2018.

\bibitem[Carmona and Lacker(2015)]{LackerCarmona-Ext}
R.~Carmona and D.~Lacker.
\newblock A probabilistic weak formulation of mean field games and
  applications.
\newblock \emph{The Annals of Applied Probability}, 25\penalty0 (3):\penalty0
  1189--1231, 2015.

\bibitem[Carmona and Wang(2018)]{Carmona2018FiniteStateCT}
R.~Carmona and P.~Wang.
\newblock Finite--state contract theory with a principal and a field of agents.
\newblock \emph{Manag. Sci.}, 67:\penalty0 4725--4741, 2018.

\bibitem[Carmona and Zeng(2022)]{zengOptimal2022}
R.~Carmona and C.~Zeng.
\newblock Optimal execution with identity optionality.
\newblock \emph{Applied Mathematical Finance}, 29\penalty0 (4):\penalty0
  261--287, 2022.
\newblock \doi{10.1080/1350486X.2023.2193343}.
\newblock URL \url{https://doi.org/10.1080/1350486X.2023.2193343}.

\bibitem[Carmona et~al.(2016)Carmona, Delarue, and
  Lacker]{Lacker_carmona_delarue_CN}
R.~Carmona, F.~Delarue, and D.~Lacker.
\newblock {Mean field games with common noise}.
\newblock \emph{The Annals of Probability}, 44\penalty0 (6):\penalty0 3740 --
  3803, 2016.
\newblock \doi{10.1214/15-AOP1060}.
\newblock URL \url{https://doi.org/10.1214/15-AOP1060}.

\bibitem[Cobza{\c{s}} et~al.()Cobza{\c{s}}, Miculescu, and Nicolae]{Cobza2019}
{\c{S}}.~Cobza{\c{s}}, R.~Miculescu, and A.~Nicolae.
\newblock \emph{Approximations Involving Lipschitz Functions}, pages 317--334.
\newblock Springer International Publishing, Cham.
\newblock ISBN 978-3-030-16489-8.
\newblock \doi{10.1007/978-3-030-16489-8_6}.
\newblock URL \url{https://doi.org/10.1007/978-3-030-16489-8_6}.

\bibitem[Cvitani{\'c} et~al.(2017)Cvitani{\'c}, Possama{\"\i}, and
  Touzi]{cvitanic2014moral}
J.~Cvitani{\'c}, D.~Possama{\"\i}, and N.~Touzi.
\newblock Moral hazard in dynamic risk management.
\newblock \emph{Management Science}, 63\penalty0 (10):\penalty0 3328--3346,
  2017.

\bibitem[Cvitani{\'c} et~al.(2018)Cvitani{\'c}, Possama{\"\i}, and
  Touzi]{cvitanic2015dynamic}
J.~Cvitani{\'c}, D.~Possama{\"\i}, and N.~Touzi.
\newblock Dynamic programming approach to principal--agent problems.
\newblock \emph{Finance and Stochastics}, 22\penalty0 (1):\penalty0 1--37,
  2018.

\bibitem[Demski and Sappington(1984)]{demski1984optimal}
J.~Demski and D.~Sappington.
\newblock Optimal incentive contracts with multiple agents.
\newblock \emph{Journal of Economic Theory}, 33\penalty0 (1):\penalty0
  152--171, 1984.

\bibitem[Djete(2020)]{djete2020some}
M.~F. Djete.
\newblock \emph{Some results on the {M}c{K}ean--{V}lasov optimal control and
  mean field games: limit theorems, dynamic programming principle and numerical
  approximations}.
\newblock PhD thesis, Universit{\'e} Paris Dauphine PSL, 2020.

\bibitem[Djete(2022)]{MFD-2020}
M.~F. Djete.
\newblock {Extended mean field control problem: a propagation of chaos result}.
\newblock \emph{Electronic Journal of Probability}, 27\penalty0
  (none):\penalty0 1 -- 53, 2022.
\newblock \doi{10.1214/21-EJP726}.
\newblock URL \url{https://doi.org/10.1214/21-EJP726}.

\bibitem[Djete(2023{\natexlab{a}})]{closed-loop-MFG_MDF}
M.~F. Djete.
\newblock Large population games with interactions through controls and common
  noise: convergence results and equivalence between open-loop and closed-loop
  controls.
\newblock \emph{ESAIM: COCV}, 29:\penalty0 39, 2023{\natexlab{a}}.
\newblock \doi{10.1051/cocv/2023005}.
\newblock URL \url{https://doi.org/10.1051/cocv/2023005}.

\bibitem[Djete(2023{\natexlab{b}})]{open-loop-MFG_MDF}
M.~F. Djete.
\newblock {Mean field games of controls: On the convergence of Nash
  equilibria}.
\newblock \emph{The Annals of Applied Probability}, 33\penalty0 (4):\penalty0
  2824 -- 2862, 2023{\natexlab{b}}.
\newblock \doi{10.1214/22-AAP1879}.
\newblock URL \url{https://doi.org/10.1214/22-AAP1879}.

\bibitem[Djete et~al.(0)Djete, Possama{\"\i}, and Tan]{djete2019general}
M.~F. Djete, D.~Possama{\"\i}, and X.~Tan.
\newblock Mckean–vlasov optimal control: Limit theory and equivalence between
  different formulations.
\newblock \emph{Mathematics of Operations Research}, 0\penalty0 (0):\penalty0
  null, 0.
\newblock \doi{10.1287/moor.2021.1232}.
\newblock URL \url{https://doi.org/10.1287/moor.2021.1232}.

\bibitem[Djete et~al.(2022)Djete, Possama{\"\i}, and Tan]{djete2019mckean}
M.~F. Djete, D.~Possama{\"\i}, and X.~Tan.
\newblock {McKean–Vlasov optimal control: The dynamic programming principle}.
\newblock \emph{The Annals of Probability}, 50\penalty0 (2):\penalty0 791 --
  833, 2022.
\newblock \doi{10.1214/21-AOP1548}.
\newblock URL \url{https://doi.org/10.1214/21-AOP1548}.

\bibitem[El~Karoui and M{\'e}l{\'e}ard(1990)]{el1990martingale}
N.~El~Karoui and S.~M{\'e}l{\'e}ard.
\newblock Martingale measures and stochastic calculus.
\newblock \emph{Probability Theory and Related Fields}, 84\penalty0
  (1):\penalty0 83--101, 1990.

\bibitem[El~Karoui and Tan(2013)]{karoui2013capacities}
N.~El~Karoui and X.~Tan.
\newblock Capacities, measurable selection and dynamic programming part {I}:
  abstract framework.
\newblock \emph{arXiv preprint arXiv:1310.3363}, 2013.

\bibitem[El~Karoui et~al.(1987)El~Karoui, Huu~Nguyen, and
  Jeanblanc-Picqu{\'e}]{el1987compactification}
N.~El~Karoui, D.~Huu~Nguyen, and M.~Jeanblanc-Picqu{\'e}.
\newblock Compactification methods in the control of degenerate diffusions:
  existence of an optimal control.
\newblock \emph{Stochastics}, 20\penalty0 (3):\penalty0 169--219, 1987.

\bibitem[Elie and Possama\"{\i}(2019)]{possamai2019contracting}
R.~Elie and D.~Possama\"{\i}.
\newblock Contracting theory with competitive interacting agents.
\newblock \emph{SIAM Journal on Control and Optimization}, 57\penalty0
  (2):\penalty0 1157--1188, 2019.
\newblock \doi{10.1137/17M1121202}.
\newblock URL \url{https://doi.org/10.1137/17M1121202}.

\bibitem[Elie et~al.(2019)Elie, Mastrolia, and
  Possama\"{\i}]{mastroliaAtale2019}
R.~Elie, T.~Mastrolia, and D.~Possama\"{\i}.
\newblock A tale of a principal and many, many agents.
\newblock \emph{Mathematics of Operations Research}, 44\penalty0 (2):\penalty0
  440--467, 2019.
\newblock \doi{10.1287/moor.2018.0931}.
\newblock URL \url{https://doi.org/10.1287/moor.2018.0931}.

\bibitem[Elie et~al.(2021)Elie, Hubert, Mastrolia, and
  Possamaï]{HubertMean2021}
R.~Elie, E.~Hubert, T.~Mastrolia, and D.~Possamaï.
\newblock Mean--field moral hazard for optimal energy demand response
  management.
\newblock \emph{Mathematical Finance}, 31\penalty0 (1):\penalty0 399--473,
  2021.
\newblock \doi{https://doi.org/10.1111/mafi.12291}.
\newblock URL \url{https://onlinelibrary.wiley.com/doi/abs/10.1111/mafi.12291}.

\bibitem[Espinosa and Touzi(2015)]{espinosa2015optimal}
G.-E. Espinosa and N.~Touzi.
\newblock Optimal investment under relative performance concerns.
\newblock \emph{Mathematical Finance}, 25\penalty0 (2):\penalty0 221--257,
  2015.

\bibitem[Filippov(1962)]{Filippov1962}
A.~F. Filippov.
\newblock On certain questions in the theory of optimal control.
\newblock \emph{Journal of the Society for Industrial and Applied Mathematics
  Series A Control}, 1\penalty0 (1):\penalty0 76--84, 1962.
\newblock \doi{10.1137/0301006}.
\newblock URL \url{https://doi.org/10.1137/0301006}.

\bibitem[Frei and Reis(2011)]{Frei2011AFM}
C.~Frei and G.~D. Reis.
\newblock A financial market with interacting investors: does an equilibrium
  exist?
\newblock \emph{Mathematics and Financial Economics}, 4:\penalty0 161--182,
  2011.

\bibitem[Green and Stokey(1983)]{green1983comparison}
J.~Green and N.~Stokey.
\newblock A comparison of tournaments and contracts.
\newblock \emph{The Journal of Political Economy}, 91\penalty0 (3):\penalty0
  349--364, 1983.

\bibitem[Holmstr{\"o}m(1982)]{holmstrom1982moral}
B.~Holmstr{\"o}m.
\newblock Moral hazard in teams.
\newblock \emph{The Bell Journal of Economics}, 13\penalty0 (2):\penalty0
  324--340, 1982.

\bibitem[Holmstrom and Milgrom(1987)]{Milgrom1987}
B.~Holmstrom and P.~Milgrom.
\newblock Aggregation and linearity in the provision of intertemporal
  incentives.
\newblock \emph{Econometrica}, 55\penalty0 (2):\penalty0 303--328, 1987.
\newblock ISSN 00129682, 14680262.
\newblock URL \url{http://www.jstor.org/stable/1913238}.

\bibitem[Huang et~al.(2003)Huang, Caines, and Malham{\'e}]{huang2003individual}
M.~Huang, P.~Caines, and R.~Malham{\'e}.
\newblock Individual and mass behaviour in large population stochastic wireless
  power control problems: centralized and {N}ash equilibrium solutions.
\newblock In C.~Abdallah and F.~Lewis, editors, \emph{Proceedings of the 42nd
  IEEE conference on decision and control, 2003.}, pages 98--103. IEEE, 2003.

\bibitem[Huang et~al.(2006)Huang, Malham{\'e}, and Caines]{huang2006large}
M.~Huang, R.~Malham{\'e}, and P.~Caines.
\newblock Large population stochastic dynamic games: closed--loop
  {M}c{K}ean--{V}lasov systems and the {N}ash certainty equivalence principle.
\newblock \emph{Communications in Information \& Systems}, 6\penalty0
  (3):\penalty0 221--252, 2006.

\bibitem[Jaber et~al.(2023)Jaber, Neuman, and Voss]{JaberEquilibrium2023}
E.~A. Jaber, E.~Neuman, and M.~Voss.
\newblock Equilibrium in functional stochastic games with mean--field
  interaction.
\newblock \emph{arXiv preprint arXiv:2306.05433}, 2023.

\bibitem[Jackson and Lacker(2023)]{JacksonAppro2023}
J.~Jackson and D.~Lacker.
\newblock Approximately optimal distributed stochastic controls beyond the mean
  field setting.
\newblock \emph{arXiv preprint arXiv:2301.02901}, 2023.

\bibitem[Kobeissi(2022)]{kobeissiOnclassical2022}
Z.~Kobeissi.
\newblock On classical solutions to the mean field game system of controls.
\newblock \emph{Communications in Partial Differential Equations}, 47\penalty0
  (3):\penalty0 453--488, 2022.
\newblock \doi{10.1080/03605302.2021.1985518}.
\newblock URL \url{https://doi.org/10.1080/03605302.2021.1985518}.

\bibitem[Koo et~al.(2008)Koo, Shim, and Sung]{SungOptimal2008}
H.~K. Koo, G.~Shim, and J.~Sung.
\newblock {Optimal Multi‐-Agent Performance Measures For Team Contracts}.
\newblock \emph{Mathematical Finance}, 18\penalty0 (4):\penalty0 649--667,
  October 2008.
\newblock \doi{10.1111/j.1467-9965.2008.}

\bibitem[Lacker(2016)]{lacker2016general}
D.~Lacker.
\newblock A general characterization of the mean field limit for stochastic
  differential games.
\newblock \emph{Probability Theory and Related Fields}, 165\penalty0
  (3-4):\penalty0 581--648, 2016.

\bibitem[Lacker(2020)]{Lacker-closedloop}
D.~Lacker.
\newblock {On the convergence of closed-loop Nash equilibria to the mean field
  game limit}.
\newblock \emph{The Annals of Applied Probability}, 30\penalty0 (4):\penalty0
  1693 -- 1761, 2020.
\newblock \doi{10.1214/19-AAP1541}.
\newblock URL \url{https://doi.org/10.1214/19-AAP1541}.

\bibitem[Lacker and Flem(2023)]{leflemclosed2023}
D.~Lacker and L.~L. Flem.
\newblock {Closed-loop convergence for mean field games with common noise}.
\newblock \emph{The Annals of Applied Probability}, 33\penalty0 (4):\penalty0
  2681 -- 2733, 2023.
\newblock \doi{10.1214/22-AAP1876}.
\newblock URL \url{https://doi.org/10.1214/22-AAP1876}.

\bibitem[Lacker et~al.(2022)Lacker, Mukherjee, and Yeung]{Lackermeanfield2022}
D.~Lacker, S.~Mukherjee, and L.~C. Yeung.
\newblock Mean field approximations via log--concavity.
\newblock \emph{arXiv preprint arXiv:2206.01260}, 2022.

\bibitem[Laffont and Martimort(2002)]{LaffontMartimort2002}
J.-J. Laffont and D.~Martimort.
\newblock \emph{The Theory of Incentives: The Principal-Agent Model}.
\newblock Princeton University Press, 2002.
\newblock ISBN 9780691091846.
\newblock URL \url{http://www.jstor.org/stable/j.ctv7h0rwr}.

\bibitem[Laffont and Tirole(1993)]{laffont1993theory}
J.-J. Laffont and J.~Tirole.
\newblock \emph{A Theory of Incentives in Procurement and Regulation}.
\newblock MIT Press, 1993.
\newblock ISBN 9780585167978.

\bibitem[Lasry and Lions(2006)]{lasry2006jeux}
J.-M. Lasry and P.-L. Lions.
\newblock Jeux {\`a} champ moyen. {I}--{L}e cas stationnaire.
\newblock \emph{Comptes Rendus Math{\'e}matique}, 343\penalty0 (9):\penalty0
  619--625, 2006.

\bibitem[Lasry and Lions(2007)]{lasry2007mean}
J.-M. Lasry and P.-L. Lions.
\newblock Mean field games.
\newblock \emph{Japanese Journal of Mathematics}, 2\penalty0 (1):\penalty0
  229--260, 2007.

\bibitem[Mookherjee(1984)]{Mookherjee1984}
D.~Mookherjee.
\newblock Optimal incentive schemes with many agents.
\newblock \emph{The Review of Economic Studies}, 51\penalty0 (3):\penalty0
  433--446, 1984.
\newblock ISSN 00346527, 1467937X.
\newblock URL \url{http://www.jstor.org/stable/2297432}.

\bibitem[M{\"u}ller(1998)]{muller1998first}
H.~M{\"u}ller.
\newblock The first--best sharing rule in the continuous--time principal--agent
  problem with exponential utility.
\newblock \emph{Journal of Economic Theory}, 79\penalty0 (2):\penalty0
  276--280, 1998.
\newblock ISSN 0022-0531.
\newblock \doi{http://dx.doi.org/10.1006/jeth.1997.2381}.
\newblock URL
  \url{http://www.sciencedirect.com/science/article/pii/S0022053197923814}.

\bibitem[M{\"u}ller(2000)]{muller2000asymptotic}
H.~M{\"u}ller.
\newblock Asymptotic efficiency in dynamic principal--agent problems.
\newblock \emph{Journal of Economic Theory}, 91\penalty0 (2):\penalty0
  292--301, 2000.

\bibitem[Roxin(1962)]{Roxin1962}
E.~Roxin.
\newblock {The existence of optimal controls.}
\newblock \emph{Michigan Mathematical Journal}, 9\penalty0 (2):\penalty0 109 --
  119, 1962.
\newblock \doi{10.1307/mmj/1028998668}.
\newblock URL \url{https://doi.org/10.1307/mmj/1028998668}.

\bibitem[Salani{\'e}(1997)]{salanie1997economics}
B.~Salani{\'e}.
\newblock \emph{The Economics of Contracts: A Primer}.
\newblock MIT Press, 1997.
\newblock ISBN 9780262193863.

\bibitem[Sannikov(2008)]{sannikov2008continuous}
Y.~Sannikov.
\newblock A continuous--time version of the principal--agent problem.
\newblock \emph{The Review of Economic Studies}, 75\penalty0 (3):\penalty0
  957--984, 2008.

\bibitem[Sannikov(2013)]{sannikov2012contracts}
Y.~Sannikov.
\newblock Contracts: the theory of dynamic principal--agent relationships and
  the continuous--time approach.
\newblock In D.~Acemoglu, M.~Arellano, and E.~Dekel, editors, \emph{Advances in
  economics and econometrics, 10th world congress of the Econometric Society,
  volume 1, economic theory}, number~49 in Econometric Society Monographs,
  pages 89--124. Cambridge University Press, 2013.

\bibitem[Sch{\"a}ttler and Sung(1993)]{schattler1993first}
H.~Sch{\"a}ttler and J.~Sung.
\newblock The first--order approach to the continuous--time principal--agent
  problem with exponential utility.
\newblock \emph{Journal of Economic Theory}, 61\penalty0 (2):\penalty0
  331--371, 1993.

\bibitem[Sch{\"a}ttler and Sung(1997)]{schattler1997optimal}
H.~Sch{\"a}ttler and J.~Sung.
\newblock On optimal sharing rules in discrete--and continuous--time
  principal--agent problems with exponential utility.
\newblock \emph{Journal of Economic Dynamics and Control}, 21\penalty0
  (2):\penalty0 551--574, 1997.

\bibitem[Schmitz(2006)]{SchmitzReview}
P.~W. Schmitz.
\newblock \emph{Journal of Institutional and Theoretical Economics (JITE) /
  Zeitschrift für die gesamte Staatswissenschaft}, 162\penalty0 (3):\penalty0
  535--540, 2006.
\newblock ISSN 09324569.
\newblock URL \url{http://www.jstor.org/stable/40752600}.

\bibitem[Sung(1995)]{sung1995linearity}
J.~Sung.
\newblock Linearity with project selection and controllable diffusion rate in
  continuous--time principal--agent problems.
\newblock \emph{The RAND Journal of Economics}, 26\penalty0 (4):\penalty0
  720--743, 1995.

\bibitem[Villani(2008)]{villani2008optimal}
C.~Villani.
\newblock \emph{Optimal transport: old and new}, volume 338 of
  \emph{Grundlehren der Mathematischen Wissenschafte}.
\newblock Springer, 2008.

\end{thebibliography}

\begin{appendix}

\section{Technical results}


Let $(\Sc,\Delta)$ be a Polish space and $v$ be a progressively Borel map s.t. $[0,T] \x \M(\Sc) \x \R \ni (t,z,x) \to v(t,z,x) \in \R^{d \x d}$ is Lipschitz in $x$ uniformly in $(t,z)$ and $m \Ir_{d} \le v v^\top \le M \Ir_{d}$ for some $0 < m \le M$. 
On the probability space $\left(\Om,\H,\P \right)$, let $(b_t, Z_t)_{t \in [0,T]}$ be an $\R^d \x \Pc(\Sc)$--valued $\H$--predictable process with $(Z_t)_{t \in [0,T]}$ independent of $(W,\iota)$. The process $Z=(Z_t)_{t \in [0,T]}$ will be consider as an element of $\M(\Sc)$. We consider $X$ the process satisfying: $X_0:=\iota$, $\P$--a.e.
\begin{align*}
    \mathrm{d}X_t
    =b_t \mathrm{d}t + v(t,Z,X_t) \mathrm{d}W_t.
\end{align*}
We set $\mu=(\mu_t)_{t \in [0,T]}$ by $\mu_t=\Lc(X_t|Z)$ for all $t \in [0,T]$. Let $\widehat{X}$ be the process satisfying: $\widehat{X}_0=\iota$, $\P$--a.e.
\begin{align} \label{Eq:appendix_projection}
    \mathrm{d}\widehat{X}_t
    =\widehat{b}(t, \widehat{X}_t,\mu, Z) \mathrm{d}t + v(t,Z,\widehat{X}_t) \mathrm{d}W_t\;\;\mbox{where}\;\;\widehat{b}(t,x,\pi,z):=\E\Big[b_t \Big| X_t=x, \mu_{t \wedge \cdot}= \pi(t \wedge \cdot), Z_{t \wedge \cdot} = z(t \wedge \cdot)  \Big].
\end{align}

\begin{lemma} \label{lemma:projection}
    The process $\widehat{X}$ verifies:
    \begin{align*}
        \Lc(X_t| Z)
        =
        \Lc(\widehat{X}_t| Z)=\Lc(\widehat{X}_t| Z_{t \wedge \cdot}, \mu_{t \wedge \cdot}),\;\;\mbox{for all}\;t \in [0,T],\;\;\P\mbox{--a.e.}
    \end{align*}
\end{lemma}

\begin{proof}
    For simplification we only treat the case $d=1$. Since $Z$ is independent of $(W,\iota)$, 
    by applying It\^o Formula and taking the conditional expectation w.r.t. $\sigma\{Z\}$, we check that: for any twice differentiable bounded map $\varphi:\R \to \R$, and $s \le t$,
    \begin{align*}
        \langle \varphi, \mu_t \rangle -  \langle \varphi, \mu_s \rangle - \int_s^t \int_{\R} v^2 (t,Z,x)\varphi''(x) \mu_r(\mathrm{d}x) \mathrm{d}r
        =
        \int_s^t\int_{\R} \varphi'(x) \E[b_r|X_r=x,Z] \mu_r(\mathrm{d}x) \mathrm{d}r.
    \end{align*}
    Therefore, there exists a progressively Borel measurable map $H^\varphi: [0,T] \x C([0,T];\Pc(\R)) \x \M(\Sc) \to  \R$ s.t. 
    $$
        H^{\varphi}(t,\mu,Z)=\int_{\R} \varphi'(x) \E[b_t|X_t=x,Z] \mu_t(\mathrm{d}x),\;\;\mathrm{d}t \otimes \mathrm{d}\P\mbox{--a.e.}
    $$
    Consequently by conditioning w.r.t. $\sigma\{Z_{t \wedge \cdot}, \mu_{t \wedge \cdot} \}$, we obtain that $H^{\varphi}(t,\mu,Z)=\int_{\R} \varphi'(x) \E[b_t|X_t=x,\mu_{t \wedge \cdot},Z_{t \wedge \cdot}] \mu_t(\mathrm{d}x)$, $\mathrm{d}t \otimes \mathrm{d}\P$--a.e. This is true for any $\varphi$. This is enough to see that the process $\widehat{X}$ defined in \eqref{Eq:appendix_projection} satisfies the statement of the Lemma. We can conclude.
\end{proof}

\medskip
Now, we consider that the random variable $Z$ is independent of $(W^i,\iota^i)_{ i \ge 1}$ s.t. $\E[\Delta(Z,z_0)^{p'}]< \infty$ for some $z_0 \in \Sc$. For each $i \ge 1$, $X^{i,n}$ is the $\R^d$--valued $\F$--adapted process satisfying: $X^{i,n}_0=\iota^i$,
\begin{align*}
    \mathrm{d}X^{i,n}_t
    =
    b^{i,n}_t \mathrm{d}t + v(t,Z,X^{i,n}_t)\mathrm{d}W^i_t
\end{align*}
where for each $n \ge 1$, $(b^{i,n})_{ 1 \le i \le n}$ is a sequence of $\R^d$--valued $\F$--predictable processes, and the following condition is satisfies: 
\begin{align} \label{eq:unifor_bound_b}
    \sup_{n \ge 1} \frac{1}{n} \sum_{i=1}^n \E \left[\int_0^T |b^{i,n}_t|^{p'} \mathrm{d}t\right] < \infty.
\end{align}
Let $(\phi^n)_{n \ge 1}$ be a sequence of Borel map verifying $\phi^n:\R^d \to \Vc$ where $\Vc$ is some Polish space and
\begin{align*}
    \lim_{n \to \infty}\delta_{\phi^n(x)}(\mathrm{d}v)\Gr(x)\mathrm{d}x
    =
    \Rr(x)(\mathrm{d}v)\Gr(x)\mathrm{d}x\;\mbox{for the weak convergence}
\end{align*}
where $\Rr$ is a kernel i.e. $\R^d \ni x \mapsto \Rr(x) \in \Pc(\Vc)$ is a Borel map and $\Gr$ is the density $\Gr(x):=\frac{1}{1+|x|^p}\left(\int_{\R^d} \frac{1}{1+|y|^p} \mathrm{d}y \right)^{-1}$. Recall that $1 <p < p'$. We define the sequence $(\Pr^n)_{n \ge 1} \subset \Pc \left(\Pc(\Vc \x \R^d) \x \Cc_\Wc \x \Sc \right)$ by
\begin{align*}
    \Pr^n
    =
    \P \circ \left( \delta_{\phi^n(x)}(\mathrm{d}v) \varphi^n_T(\mathrm{d}x),\;\; \varphi^n,\;Z \right)^{-1}\;\;\mbox{where}\;\;\varphi^n:=(\varphi^n_t)_{t \in [0,T]}:=\left(\frac{1}{n} \sum_{i=1}^n \delta_{X^i_t}(\mathrm{d}x) \right)_{t \in [0,T]}.
\end{align*}

\begin{proposition} \label{prop:charac-convergence}
    The sequence $(\Pr^n)_{n \ge 1}$ is relatively compact in $\Wc_p$. In addition, each limit point $\Pr=\P \left( \kappa,\mu, Z\right)^{-1}$ verifies: $\P$--a.e.
    \begin{align*}
        \Rr(x)(\mathrm{d}v)\mu_T(\mathrm{d}x)=\kappa(\mathrm{d}v,\mathrm{d}x).
    \end{align*}
\end{proposition}

\begin{proof}
    By using condition \eqref{eq:unifor_bound_b} and $\E[\Delta(Z,z_0)^{p'}]< \infty$, the relative compactness of $(\Pr^n)_{n \ge 1}$ follows from classical techniques. We take $\Pr=\P \left( \kappa,\mu, Z\right)^{-1}$ the limit of a sub--sequence. For simplicity, we use the same notation for the sequence and its sub--sequence. Let us show the second point. For this purpose, we first show that: for any bounded continuous $(F^q)_{1 \le q \le Q}$ and $(G^k)_{1 \le k \le K}$, we have
    \begin{align} \label{eq:test_func}
        \E\left[ \prod_{k=1}^{K} \langle G^k, \mu_T \rangle \prod_{q=1}^Q \int_{\Vc \x \R^d} F^q(v,x)\kappa(\mathrm{d}v,\mathrm{d}x)  \right]
        =
        \E\left[ \prod_{k=1}^{K} \langle G^k, \mu_T \rangle \prod_{q=1}^Q \int_{\Vc \x \R^d} F^q(v,x)\Rr(x)(\mathrm{d}v)\mu_T(\mathrm{d}x)  \right].
    \end{align}
    We only prove this result for $Q=K=2$. The generalization is straightforward from this case. Notice that, for each $n \ge 1$,
    \begin{align*}
        &\E\left[ \prod_{k=1}^{2} \langle G^k, \varphi^n_T \rangle \prod_{q=1}^2 \frac{1}{n} \sum_{i=1}^n F^q\left(\phi^n(X^{i,n}_T),X^{i,n}_T \right)  \right]
        \\
        &=\frac{1}{n^4} \sum_{i_1,i_2,i_3,i_4} \E 
        \left[
        \int_{\R^4} G^1(x^1)G^2(x^2) F^1(\phi^n(x^3),x^3) F^2(\phi^n(x^4),x^4) \Lc\left(X^{i_1}_T,X^{i_2}_T, X^{i_3}_T, X^{i_4}_T |Z\right)(\mathrm{d}x^1,\mathrm{d}x^2,\mathrm{d}x^3,\mathrm{d}x^4)
        \right].
    \end{align*}
    When the indexes $\boldsymbol{i}:=(i_1,i_2,i_3,i_4)$ are all distinct, since $m \Ir_{d \x d} \le v v^\top \le M \Ir_{d \x d}$ and $Z$ is independent of $(W^i,\iota^i)_{i \ge 1}$, for any $z \in \Sc$, the probability $\Lc\left(X^{i_1}_T,X^{i_2}_T, X^{i_3}_T, X^{i_4}_T |Z=z\right)$ has a density $f^{n}_{\boldsymbol{i}}(z)(T,\xb)$  w.r.t. the Lebesgue measure, we noted $\xb:=(x^1,x^2,x^3,x^4)$. We can choose $f^n_{\ib}$ s.t. $[0,T] \x \Sc \x \R^4 \ni (t,z,\xb) \mapsto f^n_{\ib}(z)(t,\xb) \in \R$ is Borel measurable.  By \cite[Proposition A.1.]{closed-loop-MFG_MDF} (see also \cite[Theorem 4]{AronsonSerrin67} and \cite[Theorem 6.2.7]{FK-PL-equations}), for each compact $\Gamma \subset \R^4$, there exists $\delta \in (0,1)$ (independent of $z$) s.t.
    \begin{align*}
        \sup_{n \ge 1}\; \sup_{z \in \Sc} \sup_{\ib \mbox{ distinct}}\sup_{\xb \neq \xb', (\xb,\xb') \in \Gamma \x \Gamma} \frac{|f^n_{\ib}(z)(T,\xb)-f^n_{\ib}(z)(T,\xb')|}{|\xb-\xb'|^{\delta}} < \infty.
    \end{align*}
    We can then conclude, for any bounded map $h$, the sequence of maps 
    $$
        \left(\R^4 \ni \xb \mapsto \frac{1}{n^4} \sum_{\ib \mbox{ distinct}}\E\left[ h(Z) f^n_{\ib}(Z)(T,\xb) \right] \in \R \right)_{n \ge 1}
    $$
    is relatively compact in the set of continuous functions for the locally uniform topology. Using the weak convergence, we can show that the whole sequence converges and that the limit is $P(\xb)$ that verifies: for any continuous map $C$
    \begin{align*}
        \int_{\R^4}C(\xb)P(\xb) \mathrm{d}\xb
        =
        \E\left[ \int_{\R^4} C(\xb)h(Z) \mu_T(\mathrm{d}x^1) \mu_T(\mathrm{d}x^2)\mu_T(\mathrm{d}x^3) \mu_T(\mathrm{d}x^4) \right].
    \end{align*}
    Consequently, we can see that
    \begin{align} \label{eq:uniform_cong}
        \lim_{n \to \infty}\sup_{\xb \in \Gamma} \left| \frac{1}{n^4} \sum_{\ib \mbox{ distinct}}\E\left[ h(Z) f^n(Z)(T,\xb) \right]
        -
        \E\left[ h(Z) f(Z)(T,\xb) \right] \right|\;\mbox{for any compact }\Gamma \subset \R^d 
    \end{align}
    $\mbox{where}\;f(Z)(T,\xb)\;\mbox{is the density of}\;\E\left[ \mu_T(\mathrm{d}x^1) \mu_T(\mathrm{d}x^2)\mu_T(\mathrm{d}x^3) \mu_T(\mathrm{d}x^4)|Z \right] \in \Pc(\R^4)$. We can then rewrite 
    \begin{align*}
        &=
        \frac{1}{n^4} \sum_{\ib\;\mbox{distincts}} \E 
        \left[
        \int_{\R^4} G^1(x^1)G^2(x^2) F^1(\phi^n(x^3),x^3)F^2(\phi^n(x^4),x^4) \Lc\left(X^{i_1}_T,X^{i_2}_T, X^{i_3}_T, X^{i_4}_T |Z\right)(\mathrm{d}x^1,\mathrm{d}x^2,\mathrm{d}x^3,\mathrm{d}x^4)
        \right]
        \\
        &=
        \int_{\R^4} G^1(x^1)G^2(x^2)\Gr^{-1}(x^3)  \Gr^{-1}(x^4)\frac{1}{n^4} \sum_{\ib\;\mbox{distinct}}\E 
        \left[f^n_{\ib}(Z)(T,\xb))\right] \mathrm{d}x^1\mathrm{d}x^2\prod_{q=3}^4 \left(\int_{A}F^{q-2}(a,x^q) \delta_{\phi^n(x^q)}(\mathrm{d}a)\Gr(x^q)\mathrm{d}x^{q} \right).
    \end{align*}
    By combining the locally uniform convergence of the sequence in \eqref{eq:uniform_cong}, the weak convergence and the fact that $\frac{1}{n} \sum_{i=1}^n \E \left[ \sup_{t \in [0,T]} |X^{i,n}_t|^{p'} \right] < \infty$, we obtain 
    \begin{align*}
        &\lim_{n \to \infty}\int_{\R^4} G^1(x^1)G^2(x^2)\Gr^{-1}(x^3)  \Gr^{-1}(x^4)\frac{1}{n^4} \sum_{\ib\;\mbox{distinct}}\E 
        \left[f^n_{\ib}(Z)(T,\xb))\right] \mathrm{d}x^1\mathrm{d}x^2\prod_{q=3}^4 \int_{A}F^{q-2}(a,x^q) \delta_{\phi^n(x^q)}(\mathrm{d}a)\Gr(x^q)\mathrm{d}x^{q}
        \\
        &=
        \int_{\R^4} G^1(x^1)G^2(x^2)\Gr^{-1}(x^3)  \Gr^{-1}(x^4)\E 
        \left[f(Z)(T,\xb))\right] \mathrm{d}x^1\mathrm{d}x^2\prod_{q=3}^4 \int_{A}F^{q-2}(a,x^q) \Rr(x^q)(\mathrm{d}a)\Gr(x^q)\mathrm{d}x^{q}
        \\
        &=
        \int_{\R^4} G^1(x^1)G^2(x^2)\E 
        \left[f(Z)(T,\xb))\right] \mathrm{d}x^1\mathrm{d}x^2\prod_{q=3}^4 \int_{A}F^{q-2}(a,x^q) \Rr(x^q)(\mathrm{d}a)
        \mathrm{d}x^{q}
        \\
        &=
        \E\left[\int_{\R^4} G^1(x^1)G^2(x^2) \int_{A}F^{3}(a,x^3) \Rr(x^3)(\mathrm{d}a) \int_{A}F^{4}(a,x^4) \Rr(x^4)(\mathrm{d}a) \mu_T(\mathrm{d}x^1) \mu_T(\mathrm{d}x^2)\mu_T(\mathrm{d}x^3) \mu_T(\mathrm{d}x^4)   \right].
    \end{align*}
    By combining our results and the fact that $(G^1,G^2,F^1,F^2)$ are bounded, we get
    \begin{align*}
        &\E\left[ \prod_{k=1}^{2} \langle G^k, \mu_T \rangle \prod_{q=1}^2 \int_{\Ar \x \R^d} F^q(a,x)\kappa(\mathrm{d}a,\mathrm{d}x)  \right]
        \\
        &=\lim_{n \to \infty}\E\left[ \prod_{k=1}^{2} \langle G^k, \varphi^n_T \rangle \prod_{q=1}^2 \frac{1}{n} \sum_{i=1}^n F^q\left(\phi^n(X^{i,n}_T),X^{i,n}_T \right)  \right]
        \\
        &=\lim_{n \to \infty}\frac{1}{n^4} \sum_{i_1,i_2,i_3,i_4} \E 
        \left[
        \int_{\R^4} G^1(x^1)G^2(x^2) F^1(\phi^n(x^3),x^3) F^2(\phi^n(x^4),x^4) \Lc\left(X^{i_1}_T,X^{i_2}_T, X^{i_3}_T, X^{i_4}_T |Z\right)(\mathrm{d}x^1,\mathrm{d}x^2,\mathrm{d}x^3,\mathrm{d}x^4)
        \right]
        \\
        &=\lim_{n \to \infty}\frac{1}{n^4} \sum_{\ib\;\mbox{distinct}} \E 
        \left[
        \int_{\R^4} G^1(x^1)G^2(x^2) F^1(\phi^n(x^3),x^3) F^2(\phi^n(x^4),x^4) \Lc\left(X^{i_1}_T,X^{i_2}_T, X^{i_3}_T, X^{i_4}_T |Z\right)(\mathrm{d}x^1,\mathrm{d}x^2,\mathrm{d}x^3,\mathrm{d}x^4)
        \right]
        \\
        &=
        \E\left[\int_{\R^4} G^1(x^1)G^2(x^2) \int_{A}F^{3}(a,x^3) \Rr(x^3)(\mathrm{d}a) \int_{A}F^{4}(a,x^4) \Rr(x^4)(\mathrm{d}a) \mu_T(\mathrm{d}x^1) \mu_T(\mathrm{d}x^2)\mu_T(\mathrm{d}x^3) \mu_T(\mathrm{d}x^4)   \right]
        \\
        &=
        \E\left[ \prod_{k=1}^{2} \langle G^k, \mu_T \rangle \prod_{q=1}^2 \int_{\Ar \x \R^d} F^q(a,x)\Rr(x)(\mathrm{d}a)\mu_T(\mathrm{d}x)  \right].
    \end{align*}
    We conclude the result in \eqref{eq:test_func}. Since \eqref{eq:test_func} is true for any $(F^q)_{1 \le q \le Q}$ and $(G^k)_{1 \le k \le K}$, by \cite[Proposition A.3.]{djete2019general}, we can say that $\Lc^{\P}(\mu_T, \kappa)=\Lc^{\P}\left(\mu_T, \Rr(x)(\mathrm{d}v)\mu_T(\mathrm{d}x) \right)$. This leads to $\kappa=\Rr(x)(\mathrm{d}v)\mu_T(\mathrm{d}x)$, $\P$--a.e. Indeed, for a continuous bounded map $h$, by the equality in distribution we have
    \begin{align*}
        \E^{\P} \left[ \left| \langle h, \kappa \rangle - \int_{\Vc \x \R^d } h(v,x) \Rr(x)(\mathrm{d}v)\mu_T(\mathrm{d}x) \right| \right]
        =
        0.
    \end{align*}
\end{proof}

Let us consider a closed convex set $(U,\Delta^U)$, a Polish space $(\Zc,\Delta^{\Zc})$ and a Borel measurable map $(f,h):\Zc \x U \to \R^d \x \R$. We assume that for any $\eta \in \Zc$, the set 
$$
    \left\{ \left(f(\eta,u), z \right):\;z \le h(\eta,u),\;\;u \in U\right\}
$$
is a closed convex set.

\begin{proposition} \label{prop:projection}
    There exists a universally measurable map $\Uc: \Zc \x \R^d \x \R \to U$ s.t. for any $q \in \Pc(U)$, we have 
    \begin{align*}
        \int_{U} f(\eta,u')q(\mathrm{d}u')
        =
        f \left(\eta, \;\Uc \left(\eta,\; \int_{U} f(\eta,u')q(\mathrm{d}u'),\;\int_{U} h(\eta,u')q(\mathrm{d}u') \right) \right)
    \end{align*}
    and
    \begin{align*}
        \int_{U} h(\eta,u')q(\mathrm{d}u')
        \le
        h \left(\eta, \;\Uc \left(\eta,\; \int_{U} f(\eta,u')q(\mathrm{d}u'),\;\int_{U} h(\eta,u')q(\mathrm{d}u') \right) \right).
    \end{align*}
\end{proposition}

\begin{proof}
    We apply the measurable selection Theorem of \cite[Theorem 2.20.]{karoui2013capacities}. Indeed, let us first observe that the set $A:=\{ (\eta,y,z,u):\;f(\eta,u)=y,\;z \le h(\eta,u) \}$ is a Borel set. Then, by \cite[Theorem 2.20.]{karoui2013capacities}, there exists a universally measurable map $\Uc: \Zc \x \R^d \x \R \to U$ s.t. 
    \begin{align*}
        \left\{
            \left(\eta,y,z, \Uc(\eta,y,z) \right)
        \right\}
        \subset
        A
    \end{align*}
    and
    \begin{align*}
        \Pi_{\Zc \x \R^d \x \R}(A)
        :=
        \left\{ (\eta,y,z):\;\;\exists u \in U\;\mbox{s.t.}\;f(\eta,u)=y,\;z \le h(\eta,u) 
        \right\}
        =
        \left\{
            (\eta,y,z):\;\;\Uc(\eta,y,z) \in U
        \right\}.
    \end{align*}
    Given the assumption verified by $(f,h)$, for any $q \in \Pc(U)$, $\left( \eta,\;\int_U f(\eta,u')q(\mathrm{d}u'), \int_U h(\eta,u')q(\mathrm{d}u') \right)$ belongs to  $\Pi_{\Zc \x \R^d \x \R}(A)$. This means that 
    \begin{align*}
        \int_{U} f(\eta,u')q(\mathrm{d}u')
        =
        f \left(\eta, \;\Uc \left(\eta,\; \int_{U} f(\eta,u')q(\mathrm{d}u'),\;\int_{U} h(\eta,u')q(\mathrm{d}u') \right) \right)
    \end{align*}
    and
    \begin{align*}
        \int_{U} h(\eta,u')q(\mathrm{d}u')
        \le
        h \left(\eta, \;\Uc \left(\eta,\; \int_{U} f(\eta,u')q(\mathrm{d}u'),\;\int_{U} h(\eta,u')q(\mathrm{d}u') \right) \right).
    \end{align*}
\end{proof}

\end{appendix}

\end{document}